\pdfoutput=1
\RequirePackage{ifpdf}
\ifpdf 
\documentclass[pdftex]{sigma}
\else
\documentclass{sigma}
\fi

\numberwithin{equation}{section}

\usepackage{mathrsfs}
\usepackage{stmaryrd}
\usepackage{tikz}
\usepackage{tikz-cd}
\usetikzlibrary{shapes, shapes.geometric, calc, intersections, decorations.markings, decorations.pathmorphing, arrows.meta}
\usepackage{upgreek, enumitem}
\usepackage{bm}
\usepackage{mathtools}
\usepackage[all]{xy}

\newtheorem{Theorem}{Theorem}[section]
\newtheorem*{Theorem*}{Theorem}
\newtheorem{Corollary}[Theorem]{Corollary}
\newtheorem{Lemma}[Theorem]{Lemma}
\newtheorem{Proposition}[Theorem]{Proposition}

\newtheorem{mainthm}{Theorem}
\newtheorem{maincor}[mainthm]{Corollary}

\newtheorem{mainthmp}{Theorem}

\theoremstyle{definition}
\newtheorem{assumption}[Theorem]{Assumption}
\newtheorem{Definition}[Theorem]{Definition}

\newtheorem{Example}[Theorem]{Example}
\newtheorem{Remark}[Theorem]{Remark}

\tikzset{
 blob/.style={circle, draw=black, fill=black, inner sep=0, minimum size=\blobsize},
 flin/.style={circle, draw=black, fill=white, inner sep=0, minimum size=\blobsize, line width=1pt},
 flout/.style={rectangle, draw=black, fill=white, inner sep=0, minimum size=\blobsize, line width=1pt}
}
\def\blobsize{2mm}
\def\circwidth{0.7pt}
\def\LLwidth{2.5pt}

\newcommand{\flow}[4][]{\begin{scope}[#1]
\pgfmathsetseed{#4}
\draw[decorate, decoration={random steps, segment length=0.11cm, amplitude=0.04cm, pre=lineto, pre length=0.07cm, post=lineto, post length=0.07cm}, rounded corners=0.04cm, ->, shorten >=-1pt, line cap=butt] #2 -- ($#2!0.5!#3$) node(midpt){};
\draw[decorate, decoration={random steps, segment length=0.11cm, amplitude=0.04cm, pre=lineto, pre length=0.07cm, post=lineto, post length=0.07cm}, rounded corners=0.04cm, line cap=butt] ($#2!0.5!#3$) -- #3;
\end{scope}
} 

\renewcommand{\phi}{\varphi}
\newcommand{\id}{\operatorname{id}}
\newcommand{\Ham}{\operatorname{Ham}}
\newcommand{\coker}{\operatorname{coker}}
\renewcommand{\min}{\mathrm{min}}
\newcommand{\val}{\operatorname{\mathfrak{v}}}

\newcommand{\CC}{\mathbb{C}}
\newcommand{\RR}{\mathbb{R}}
\newcommand{\ZZ}{\mathbb{Z}}
\newcommand{\PP}{\mathbb{P}}
\newcommand{\LL}{L^\flat}
\newcommand{\hatS}{\widehat{S}}
\newcommand{\tilS}{\widetilde{S}}
\newcommand{\hatR}{\widehat{R}}
\newcommand{\hatCO}{\widehat{\CO}{}^0_\bL}
\newcommand{\hatCOb}{\widehat{\CO}{}^0_{(\bL, \mathbf{b})}}
\newcommand{\hatB}{\widehat{\calB}}
\newcommand{\hatPhi}{\widehat{\Phi}}
\newcommand{\hatTheta}{\widehat{\Theta}_{\LL}}
\newcommand{\tilalpha}{\widetilde{\alpha}}

\newcommand{\calA}{\mathcal{A}}
\newcommand{\calB}{\mathcal{B}}
\newcommand{\calC}{\mathcal{C}}
\newcommand{\calD}{\mathcal{D}}
\newcommand{\calE}{\mathcal{E}}
\newcommand{\calL}{\mathcal{L}}
\newcommand{\calM}{\mathcal{M}}
\newcommand{\calN}{\mathcal{N}}
\newcommand{\calP}{\mathcal{P}}
\newcommand{\calQ}{\mathcal{Q}}
\newcommand{\calS}{\mathcal{S}}
\newcommand{\eps}{\varepsilon}
\newcommand{\bim}[2]{\mu^{#1 \vert 1 \vert #2}}

\setenumerate{label=(\roman*)}

\newcommand{\kk}{\mathbf{k}}
\newcommand{\Char}{\operatorname{char}}
\newcommand{\ev}{\operatorname{ev}}
\newcommand{\m}{\mathfrak{m}}
\newcommand{\frn}{\mathfrak{n}}

\newcommand{\ks}{\operatorname{\mathfrak{ks}}}
\renewcommand{\mod}{\mathbin{\mathrm{mod}}}
\newcommand{\gr}{\operatorname{gr}}
\newcommand{\bL}{\mathbf{L}}
\newcommand{\op}{\mathrm{op}}
\newcommand{\Frob}{\operatorname{\mathfrak{f}}}
\newcommand{\trho}{\widetilde{\rho}}
\newcommand{\hateta}{\widehat{\eta}}

\newcommand{\mf}{\mathrm{mf}}
\newcommand{\diff}{\mathrm{d}}
\newcommand{\rH}{\operatorname{H}}
\newcommand{\rC}{\operatorname{C}}
\newcommand{\HH}{\operatorname{HH}}
\newcommand{\rCC}{\operatorname{CC}}
\newcommand{\tCC}{\operatorname{{}_2CC}}
\newcommand{\HF}{\operatorname{HF}}
\newcommand{\Fuk}{\mathcal{F}}
\newcommand{\subFuk}{\mathcal{G}}

\newcommand{\Hom}{\operatorname{Hom}}
\newcommand{\CF}{\operatorname{CF}}

\newcommand{\Spec}{\operatorname{Spec}}

\newcommand{\LMF}[1][]{\operatorname{\mathcal{LM}}_{#1}}
\newcommand{\Ext}{\operatorname{Ext}}
\newcommand{\eend}{\operatorname{end}}
\newcommand{\End}{\operatorname{End}}
\newcommand{\CO}{\operatorname{\mathcal{CO}}}
\newcommand{\tCO}{\operatorname{{}_2\mathcal{CO}}}
\newcommand{\OC}{\operatorname{\mathcal{OC}}}
\newcommand{\QH}{\operatorname{QH}}
\newcommand{\Jac}{\operatorname{Jac}}
\newcommand{\qprod}{\mathbin{\star}}
\newcommand{\xto}[1]{\xrightarrow{\ \ #1 \ \ }}

\newcommand{\unsign}[1]{\mathop{\overline{#1}}\nolimits}

\newcommand{\PSS}{\operatorname{PSS}}

\begin{document}

\allowdisplaybreaks

\newcommand{\arXivNumber}{2308.03438}

\renewcommand{\PaperNumber}{013}

\FirstPageHeading

\ShortArticleName{Hochschild Cohomology of the Fukaya Category}

\ArticleName{Hochschild Cohomology of the Fukaya Category \\ via Floer Cohomology with Coefficients}

\Author{Jack SMITH}

\AuthorNameForHeading{J.~Smith}

\Address{Cambridge, UK}
\Email{\mail{jackesmith100@gmail.com}}

\ArticleDates{Received January 28, 2025, in final form January 25, 2026; Published online February 11, 2026}

\Abstract{Given a monotone Lagrangian $L$ in a compact symplectic manifold $X$, we construct a commutative diagram relating the closed-open string map $\mathcal{CO}_\lambda \colon \operatorname{QH}^*(X) \to \operatorname{HH}^*(\mathcal{F} (X)_\lambda)$ to a variant of the length-zero closed-open map on $L$ incorporating $\mathbf{k}[\operatorname{H}_1(L; \mathbb{Z})]$ coefficients, denoted $\mathcal{CO}^0_\mathbf{L}$. The former is categorically important but very difficult to compute, whilst the latter is geometrically natural and amenable to calculation. We further show that, after a suitable completion, injectivity of $\mathcal{CO}^0_\mathbf{L}$ implies injectivity of $\mathcal{CO}_\lambda$. Via Sheridan's version of Abouzaid's generation criterion, this gives a powerful tool for proving split-generation of the Fukaya category. We illustrate this by showing that the real part of a monotone toric manifold (of minimal Chern number at least~2) split-generates the Fukaya category in characteristic~2. We also give a short new proof (modulo foundational assumptions in the non-monotone case) that the Fukaya category of an arbitrary compact toric manifold is split-generated by toric fibres.}


\Classification{53D37; 53D40; 53D12}

\section{Introduction}

\subsection{Hochschild cohomology of the Fukaya category}

Let $(X, \omega)$ be a compact monotone symplectic manifold of dimension $2n$. The monotone Fukaya category $\Fuk(X)$ of $X$, constructed by Sheridan \cite{SheridanFano} (see Ritter--Smith \cite{RitterSmith} for the non-compact version), is an $A_\infty$-category which encodes rich information about the Lagrangian submanifolds of $X$ and their Floer theory. More precisely, it is a collection of categories $\Fuk(X)_\lambda$, labelled by $\lambda$ in our coefficient field $\kk$. This $\lambda$ prescribes the weighted count of pseduoholomorphic discs with boundary on each Lagrangian.

A powerful way to understand $\Fuk(X)_\lambda$, for example, to prove homological mirror symmetry, is to find an object $\LL$ that split-generates it, in the sense that every object in the category can be built from $\LL$ by repeatedly taking mapping cones and splitting off summands. Here~$\LL$ denotes a~La\-gran\-gian $L$ equipped with various extra data. Given such a split-generator, the Yoneda embedding lets us view $\Fuk(X)_\lambda$ concretely as a category of $A_\infty$-modules over $\eend^*\bigl(\LL\bigr) = \CF^*\bigl(\LL, \LL\bigr)$. More geometrically, if $\LL$ split-generates $\Fuk(X)_\lambda$, then $L$ must intersect every other Lagrangian supporting a non-zero object. Identifying split-generators is therefore a central problem in categorical symplectic topology.

One of the key tools for proving split-generation is the following result, which is discussed in more detail, and refined in various ways, in Appendix~\ref{secGen}.

\begin{Theorem}[{generation criterion, Abouzaid \cite{AbouzaidGeometricCriterion}, Sheridan \cite[Corollary~2.19]{SheridanFano}}]\samepage
\label{thmGeneration}
If $\LL$ is an~object in~$\Fuk(X)_\lambda$, and if the composition
\begin{equation}
\label{eqCOrestrict}
\CO_{\LL} \colon \ \QH^*(X) \xto{\CO_\lambda} \HH^*(\Fuk(X)_\lambda) \xto{\mathrm{restrict}} \HH^*\bigl(\CF^*\bigl(\LL, \LL\bigr)\bigr)
\end{equation}
is injective, then $\LL$ split-generates $\Fuk(X)_\lambda$.
\end{Theorem}

Here $\HH^*(\Fuk(X)_\lambda)$ is the Hochschild cohomology of the Fukaya category, which is an important Hodge-theoretic invariant in its own right, \smash{$\HH^*\bigl(\CF^*\bigl(\LL, \LL\bigr)\bigr)$} is the Hochschild cohomology of the Floer $A_\infty$-algebra of $\LL$, and $\CO_\lambda$ is the closed-open string map: a geometrically defined unital $\kk$-algebra homomorphism. These Hochschild invariants are algebraically formidable objects, involving all of the $A_\infty$-operations. They, and the map $\CO_\lambda$, are also extremely difficult to calculate directly, since they are defined by counts of pseudoholomorphic discs with arbitrarily many boundary inputs. These require delicate perturbation schemes to makes sense of, especially when inputs are repeated.

The goal of this paper is to take these objects and connect them to more intuitive and computable geometric constructions, arising from Floer cohomology with coefficients. Our approach is related to the family Floer theory of Fukaya \cite{FukayaFamilies}, Abouzaid \cite{AbouzaidFamily}, and Cho--Hong--Lau \cite{ChoHongLauLMF}, and is consistent with the philosophy that repeated insertion of degree $1$ boundary inputs should be equivalent to formal expansion of local systems. The main innovation is that we work with mixed objects lying between the two sides of the mirror: Hochschild cochains of the Fukaya category with coefficients in matrix factorisations. In this introduction, we briefly summarise Floer cohomology with coefficients, before describing the motivating example for this work. We then set up and state the main results.

\subsection{Floer cohomology with coefficients}

In the Fukaya category, we typically allow our Lagrangians to be equipped with rank-$1$ local systems $\calL$ over $\kk$, i.e., flat $\kk$-line bundles. These line bundles weight the operations on the category according to their monodromies or parallel transport maps around the boundaries of the pseudoholomorphic discs being counted.

When working with a single Lagrangian $L$, we can essentially consider all local systems at once, by working over the group algebra $S = \kk[\rH_1(L; \ZZ)]$ and weighting each disc by the monomial describing its boundary homology class. We will denote this \emph{Floer algebra with coefficients} by~$\CF_S^*(\bL, \bL)$, for reasons that will become clear shortly. Whilst this is a very natural geometric construction, it does not make sense within the Fukaya category over $\kk$; see Section~\ref{sscFukS}. Instead, we enlarge our coefficient ring to $S$, denoting the new category by $\Fuk_S(X)$, and define the object~${\bL \in \Fuk_S(X)}$ to be $L$ equipped with the tautological rank-$1$ local system over $S$. Then the Floer algebra with coefficients is exactly the ordinary Floer algebra of $\bL$, justifying our choice of notation. More generally, we will write $\CF_S^*$ and $\HF_S^*$ for Floer complexes and their cohomology computed in $\Fuk_S(X)$. This may implicitly involve viewing objects of $\Fuk(X)$ as objects of $\Fuk_S(X)$ by extension of coefficients.

Returning momentarily to the map \smash{$\CO_{\LL}$} from \eqref{eqCOrestrict}, it has a \emph{length-zero} version
\[
\CO^0_{\LL} \colon \ \QH^*(X) \to \HF^*\bigl(\LL, \LL\bigr),
\]
obtained by projecting the bar-type Hochschild cochain complex onto its shortest piece. In~general, this loses much of the information of \smash{$\CO_{\LL}$}, but it is much easier to calculate since the discs it counts have no boundary inputs and so sophisticated perturbation schemes are not necessary. Its codomain is also much easier to understand since it only involves the differential on~$\CF^*\bigl(\LL, \LL\bigr)$, and not the higher operations. By now incorporating weights from $S$, we can define analogous map
\[
\CO^0_\bL \colon \ \QH^*(X) \to \HF_S^*(\bL, \bL).
\]
In practice, it is no harder to compute than \smash{$\CO^0_{\LL}$}, but carries more information in the form of the universal boundary weights, as opposed to the specific weights corresponding to $\calL$.

\subsection{The motivating example}
\label{sscMotivatingExample}

If $X$ is a compact monotone toric manifold and $L$ is the monotone toric fibre, then it is well-known, following Batyrev \cite{Batyrev} and Givental \cite{Givental1, Givental2}, that $\QH^*(X)$ can be described as the Jacobian ring~${\Jac W_L}$ of the superpotential of $L$, whose definition we recall in Section~\ref{sscIngredients}. Assuming that~$\kk$ is algebraically closed, or at least contains all critical points of $W_L$, the quantum cohomology decomposes as a product of ideals $\QH^*(X)_\alpha$ labelled by critical points $\alpha$. Under the isomorphism $\QH^*(X) \cong \Jac W_L$ these correspond to the completions $(\Jac W_L)_\alpha$ of $\Jac W_L$ at each $\alpha$. There is an induced decomposition of the Fukaya category into pieces $\Fuk(X)_\alpha$, each of which contains the object $L_\alpha$ given by $L$ equipped with the local system (whose monodromy is) described by~$\alpha$.

In \cite[Corollary~1.3.1]{EvansLekiliGeneration}, Evans--Lekili prove that each~$\Fuk(X)_\alpha$ is split-generated by $L_\alpha$, concluding the proof of one direction of homological mirror symmetry in this setting \cite[Corollary~9.2]{ChoHongLauTorus}. They do not use the generation criterion, which here says that~$L_\alpha$ split-generates $\Fuk(X)_\alpha$ if $\CO_{L_\alpha}$ is injective on $\QH^*(X)_\alpha$. Instead, they exploit the fact that $L$ is a free orbit of a Hamiltonian torus action on $X$. However, it is a consequence of their argument \cite[Corollary~6.2.8]{EvansLekiliGeneration} that
\begin{equation}
\label{eqELisom}
\QH^*(X)_\alpha \cong \HH^*(\Fuk(X)_\alpha) \cong \HH^*(\CF^*(L_\alpha, L_\alpha))
\end{equation}
for each $\alpha$, where the second isomorphism uses Morita invariance of Hochschild cohomology. It is natural to expect that the composition of the two isomorphisms in \eqref{eqELisom} coincides with $\CO_{L_\alpha}$.

To connect this to Floer cohomology with coefficients, the Jacobian ring $\Jac W_L$ arises in this setting as $\HF_S^*(\bL, \bL)^\op$ (see Proposition~\ref{propHFtoric}; the $\op$ is irrelevant here since the algebra is commutative, but it will be important later) and the isomorphism $\QH^*(X) \cong \Jac W_L$ is realised by precisely the map \smash{$\CO^0_\bL$} (see Theorem~\ref{thmSmallQH}, which is derived from \cite{FOOOtoricMS, SmithQH}). We can then view each $(\Jac W_L)_\alpha$ as the Floer cohomology \smash{$HF_{\hatS}^*(\bL, \bL)$} with coefficients in a suitable completion~\smash{$\hatS$} of~$S$, and the isomorphism \smash{$\QH^*(X)_\alpha \cong (\Jac W_L)_\alpha$} as the corresponding completion~\smash{$\hatCO$}. We~thus have a diagram
\begin{equation*}
\begin{tikzcd}[column sep=5em, row sep=2.5em]
\QH^*(X)_\alpha \arrow{r}{\CO_{L_\alpha}} \arrow{d}{\hatCO}[swap, anchor=center, rotate=-90, yshift=-1ex]{\cong} & \HH^*(\CF^*(L_\alpha, L_\alpha))
\\ \HF_{\hatS}^*(\bL, \bL)^\op &
\end{tikzcd}
\end{equation*}
in which all groups are isomorphic to each other. The vertical arrow is known to be an isomorphism and the horizontal arrow is expected to be one. Given the formal similarity of these two closed-open-type maps, it is natural to ask whether they can be related to each other, for example by extending the diagram to a commuting square in which the bottom and right-hand arrows are isomorphisms. The present work provides an affirmative answer to this question, and hence (by showing that $\CO_{L_\alpha}$ is an isomorphism) an alternative proof that each $L_\alpha$ split-generates~$\Fuk(X)_\alpha$.

\begin{Remark}
Partial results towards toric split-generation via injectivity of $\CO$ have previously been obtained by Ritter \cite{RitterFanoToric}, in the case when $W$ is Morse (he also deals with some non-compact~$X$), and Tonkonog \cite{TonkonogCO}, when $W$ has at worst $A_2$ singularities.
\end{Remark}

We also show, under foundational assumptions, that the analogous generation result, Theorem~\ref{TheoremL}, holds for general (not necessarily monotone) compact toric manifolds. These foundations are expected to be addressed in work announced by Abouzaid--Fukaya--Oh--Ohta--Ono around~2016, which we believe also considers their application to the toric generation question.

\subsection{Ingredients}
\label{sscIngredients}

We now introduce the subjects of our main results. We follow Seidel's conventions for $A_\infty$-algebras and -categories \cite{SeidelBook} except for the signs used to define the opposite algebra or category; see Section~\ref{sscSignConventions}.

Let $L \subset X$ be a monotone Lagrangian submanifold defining an object $\LL$ in the Fukaya category $\Fuk(X)_\lambda$ as in Section~\ref{sscFuk}. Associated to $L$ is its superpotential $W_L \in S = \kk[\rH_1(L; \ZZ)]$, which counts rigid pseudoholomorphic discs with boundary on $L$ that send a boundary marked point to a generically chosen point in $L$, each weighted by its boundary homology class. We view this as a function on the space $\rH^1(L; \kk^\times) \cong \Spec S$ of rank-$1$ local systems on $L$. For our present purposes, the object $\LL$ comprises $L$ equipped with such a local system $\calL$ satisfying~$
{W_L(\calL) = \lambda}$.

Given $L$ and $\lambda$, there is a localised mirror functor
\[
\LMF[L,\lambda] \colon \ \Fuk(X)_\lambda \to \mf(S, W_L-\lambda)
\]
to the dg-category of matrix factorisations of $W_L-\lambda$. This functor was introduced by Cho--Hong--Lau in \cite{ChoHongLauTorus}, and is essentially defined by extending scalars from $\kk$ to $S$ and then applying the covariant hom-functor associated to $\bL$. In particular, it sends each object $K$ in $\Fuk(X)_\lambda$ to $\CF_S^*(\bL, K)$. We recall the construction in more detail in Section~\ref{sscLMF}. Cho, Hong, and Lau focused on the case where $L$ is a torus, but their ideas immediately extend to general~$L$.

\begin{Remark}
One has to be a little careful about calling it a hom-functor, since $\bL$ lies in~$\Fuk_S(X)_{W_L}$ rather than $\Fuk_S(X)_\lambda$. This is what causes the `differential' on $\CF_S^*(\bL, K)$ to square to~${W_L - \lambda}$ rather than $0$, and hence what causes the functor to land in matrix factorisations rather than cochain complexes.
\end{Remark}

Via the functor $\LMF[L,\lambda]$ we can view $\mf(S, W_L-\lambda)$ as an $\Fuk(X)_\lambda$-bimodule, giving a Hochschild cochain complex $\rCC^*(\Fuk(X)_\lambda, \mf(S, W_L - \lambda))$ and a pushforward map
\[
(\LMF[L,\lambda])_* \colon \ \rCC^*(\Fuk(X)_\lambda) \to \rCC^*(\Fuk(X)_\lambda, \mf(S, W_L - \lambda)).
\]
This is the first ingredient in our results.

The second ingredient is the $A_\infty$-algebra $\CF_S^*(\bL, \bL)$. More precisely, we want the \emph{opposite} algebra $\CF_S^*(\bL, \bL)^\op$, for reasons that will become clear the proof of Proposition~\ref{propThetaHom}. This is obtained by reversing the order of the inputs to operations and introducing an associated sign twist, namely \eqref{eqopdefinition}. Its cohomology algebra $\rH^*(\CF_S^*(\bL, \bL)^\op)$ is the ordinary graded-opposite algebra to $\HF_S^*(\bL, \bL)$, which we denote by $\HF_S^*(\bL, \bL)^\op$ (\emph{graded}-opposite just means that we include Koszul signs).

\subsection{Main results}
\label{sscMainResults}

We are now in a position to state our main results. The first relates the algebra $\CF_S^*(\bL, \bL)^\op$ to the Hochschild cochain algebra $\rCC^*(\Fuk(X)_\lambda, \mf(S, W_L - \lambda))$.

\begin{mainthm}[Proposition~\ref{propThetaHom}, Lemma~\ref{lemThetaModuleAction} and Corollary~\ref{corThetaUnital}]
\label{TheoremA}
There is a geometrically defined $S$-linear $A_\infty$-algebra homomorphism
\[
\Theta\colon \ \CF_S^*(\bL, \bL)^\op \to \rCC^*(\Fuk(X)_\lambda, \mf(S, W_L - \lambda)).
\]
It is cohomologically unital and extends the module action of $\CF_S^*(\bL, \bL)^\op$ on $\LMF[L,\lambda]$ in a natural sense.
\end{mainthm}

The functor $\LMF[L,\lambda]$ and homomorphism $\Theta$ are independent of auxiliary choices in a manner explained in Section~\ref{sscIndependence}. In particular, the induced maps on cohomology are canonical.

Our second main result is that $\Theta$ is compatible with the relevant closed-open maps as follows.

\begin{mainthm}[Corollary~\ref{corThmBCommutes}]
\label{TheoremB}
The following diagram commutes:
\begin{equation}
\label{eqThmB}
\begin{tikzcd}[column sep=5em, row sep=2.5em]
\QH^*(X) \arrow{r}{\CO_\lambda} \arrow{d}{\CO^0_\bL} & \HH^*(\Fuk(X)_\lambda) \arrow{d}{\rH((\LMF[L,\lambda])_*)}
\\ \HF_S^*(\bL, \bL)^\op \arrow{r}{\rH(\Theta)} & \HH^*(\Fuk(X)_\lambda, \mf(S, W_L - \lambda)).
\end{tikzcd}
\end{equation}
\end{mainthm}

We reiterate that the top arrow is categorically important but algebraically complicated and difficult to calculate, whilst the left-hand arrow is simple, geometric, and relatively computable.

\begin{Remark}$\QH^*(X)$ carries an $A_\infty$-structure given by viewing it as the Floer cohomology of the diagonal in $X^- \times X$, but this has been little-studied. Since all of the maps in Theorem~\ref{TheoremB} are (the $1$-ary terms of) $A_\infty$-homomorphisms, one may hope to use them to compute this $A_\infty$-structure. For example, using knowledge of the $A_\infty$-structure on $\CF_S^*(\bL, \bL)$ when $L$ is the equator in $X = \CC\PP^1$ and $\Char \kk = 2$, one can recover the non-trivial $A_\infty$-product on $\QH^*\bigl(\CC\PP^1; \kk\bigr)$ computed in \cite[Example 7.3.6]{EvansLekiliGeneration}.
\end{Remark}

To apply the generation criterion to the object $\LL$ in $\Fuk(X)_\lambda$, we need to consider the composition $\CO_{\LL}$ of $\CO_\lambda$ with the restriction map from $\HH^*(\Fuk(X)_\lambda)$ to $\HH^*\bigl(\CF^*\bigl(\LL, \LL\bigr)\bigr)$. Let~$\calE$ denote the matrix factorisation \smash{$\LMF[L,\lambda]\bigl(\LL\bigr) = \CF_S^*\bigl(\bL, \LL\bigr)$}, and let $\calB$ denote its endomorphism dg-algebra. The localised mirror functor and $\Theta$ then induce $A_\infty$-algebra homomorphisms
\[
\Phi\colon \ \CF^*\bigl(\LL, \LL\bigr) \to \calB \qquad \text{and} \qquad \Theta_{\LL}\colon \ \CF_S^*(\bL, \bL)^\op \to \rCC^*\bigl(\CF^*\bigl(\LL, \LL\bigr), \calB\bigr),
\]
respectively. Restricting \eqref{eqThmB} from $\Fuk(X)_\lambda$ to $\LL$ tells us that the following diagram commutes
\begin{equation*}
\begin{tikzcd}[column sep=5em, row sep=2.5em]
\QH^*(X) \arrow{r}{\CO_{\LL}} \arrow{d}{\CO^0_\bL} & \HH^*\bigl(\CF^*\bigl(\LL, \LL\bigr)\bigr) \arrow{d}{\rH(\Phi_*)}
\\ \HF_S^*(\bL, \bL)^\op \arrow{r}{\rH(\Theta_{\LL})} & \HH^*\bigl(\CF^*\bigl(\LL, \LL\bigr), \calB\bigr).
\end{tikzcd}
\end{equation*}

It turns out that a slight algebraic modification will allow us to better understand these maps. To set this up, note that the local system $\calL$ on $\LL$ corresponds, via its monodromy, to a~group homomorphism $\rho \colon \rH_1(L; \ZZ) \to \kk^\times$, and hence to an~algebra homomorphism $S \to \kk$. Let $\m \subset S$ denote the kernel of this algebra homomorphism, and let $\widehat{S}$ denote the $\m$-adic completion of~$S$. The constructions of $\CF_S^*(\bL, \bL)^{(\op)}$, $\CO^0_\bL$, $\calB$, $\Phi$, and $\Theta_{\LL}$ have immediate analogues over~\smash{$\widehat{S}$}, which we denote by \smash{$\CF_{\hatS}^*(\bL, \bL)^{(\op)}$}, \smash{$\hatCO$}, \smash{$\hatB$}, \smash{$\hatPhi$}, and \smash{$\hatTheta$}, respectively. Note that the completion depends on the choice of $\calL$ although our notation does not explicitly show this. We then have a~commuting diagram
\begin{equation}
\label{eqCompletedDiagram}
\begin{tikzcd}[column sep=5em, row sep=2.5em]
\QH^*(X) \arrow{r}{\CO_{\LL}} \arrow{d}{\hatCO} & \HH^*\bigl(\CF^*\bigl(\LL, \LL\bigr)\bigr) \arrow{d}{\rH(\hatPhi_*)}
\\ \HF_{\hatS}^*(\bL, \bL)^\op \arrow{r}{\rH(\hatTheta)} & \HH^*(\CF^*\bigl(\LL, \LL\bigr), \hatB).
\end{tikzcd}
\end{equation}

\begin{mainthm}[Theorem~\ref{thmEMcomparison} and Proposition~\ref{propThmC}]
\label{TheoremC}
The map $\hatTheta$ is a quasi-isomorphism. So to prove that $\CO_{\LL}$ is injective on a subspace of $\QH^*(X)$ it suffices to prove that \smash{$\hatCO$} is injective on that subspace. More generally, the same result holds when $S$ is replaced by the quotient $R = S/I$ before completing, for any ideal $I$ in $S$ contained in $\m$.
\end{mainthm}

The generalisation to $S/I$ is useful in applications, for example the following.

\begin{mainthm}[Proposition~\ref{propTonkonogInterpretation}]
\label{TheoremD}
Tonkonog's criterion, Theorem~$\ref{thmTonkonogCriterion}$ {\rm \cite[{\it Theorem} 1.7]{TonkonogCO}}, for partial injectivity of $\CO_{\LL}$ is really a criterion for partial injectivity of \smash{$\hatCO$} for $I = \m^2$. This implies the corresponding statement for \smash{$\CO_{\LL}$} by Theorem~{\rm \ref{TheoremC}}. $($In this case, the $\m$-adic completion does nothing, since $I$ contains a power of $\m$.$)$
\end{mainthm}

The proof of Theorem~\ref{TheoremC} uses a spectral sequence argument to reduce $S$ or $R$ coefficients to~$\kk$ coefficients, where the result can be deduced from the general nonsense of $A_\infty$-(bi)modules. The $\m$-adic completion allows us to apply the Eilenberg--Moore comparison theorem. Geometrically, it may seem a little mysterious, but it can be understood as follows: we cannot expect $\CF^*\bigl(\LL, \LL\bigr)$ to know about the Floer theory of $L$ away from the local system $\calL$, so to relate its Hochschild invariants to $\CF_S^*(\bL, \bL)^\op$ we should expect to have to complete, or at least localise, the latter~at~$\m$.

When combined with the generation criterion, Theorem~\ref{TheoremC} gives a powerful new tool for proving split-generation. For example, suppose $X$ is a compact monotone toric manifold whose minimal Chern number is at least $2$, and assume our coefficient field $\kk$ has characteristic $2$. Let $\LL$ be the real locus of $X$ equipped with the trivial local system. By building on work of Tonkonog \cite{TonkonogCO}, for an appropriate choice of $I$ as in Theorem~\ref{TheoremC} we construct a commutative diagram of $\kk$-algebras
\[
\begin{tikzcd}[column sep=5em, row sep=2.5em]
\QH^*_R(X) \arrow{d}{\pi} \arrow{r}{\Frob_R} & \QH^{2*}_R(X) \arrow{d}{\calD_R}[swap, anchor=center, rotate=-90, yshift=-1ex]{\cong}
\\ \QH^*(X) \arrow{r}{\hatCO} & \HF^*_R(\bL, \bL).
\end{tikzcd}
\]
Here $\QH^*_R(X)$ is an extension of quantum cohomology to $R = S/I$ coefficients, $\pi$ is reduction $R \to \kk$ modulo $\m$, and $\Frob_R$ is a $\kk$-linear extension of the Frobenius morphism $x \mapsto x^2$, in a sense explained in Section~\ref{sscThmDOutline}. Meanwhile, $\calD_R$ is an isomorphism we construct, extending the known isomorphisms $\rH^{2*}(X) \to \rH^*(L)$ and $\QH^{2*} \to \HF^*\bigl(\LL, \LL\bigr)$ of Duistermaat \cite{Duistermaat} and Haug and Hyvrier \cite{Haug,Hyvrier}. Using this diagram, we show the following.

\begin{mainthm}[Proposition~\ref{propKerFrob}]
\label{TheoremE}
The map \smash{$\hatCO$} is injective, so
\begin{enumerate}[label=$(\roman*)$,ref=(\roman*)]\itemsep=0pt
\item\label{itmOrientable} If the real locus $\LL$ is orientable, then it split-generates $\Fuk(X)_0$.
\item In general, $\LL$ split-generates the ungraded Fukaya category $\Fuk(X)^\mathrm{un}_0$.
\end{enumerate}
\end{mainthm}

\begin{Remark}
Our default Fukaya category is $\ZZ/2$-graded, and each Lagrangian in it must be equipped with a $\ZZ/2$-grading in the sense of~\cite{SeidelGraded}. In our setup, existence of such a grading is equivalent to orientability; see Remark~\ref{rmkGradings}. When $\kk$ has characteristic~$2$, we can drop the $\ZZ/2$-gradings (we no longer need them to define various signs in our formulae), and obtain an ungraded category, $\Fuk(X)_\lambda^\mathrm{un}$, which allows non-orientable Lagrangians. The algebraic structures, including the generation criterion, carry over to this ungraded setting as discussed in~\cite{AmorimCho}.

The one relevant property which does not carry over to the ungraded setting is the connection between decompositions of the Fukaya category and decompositions of $\QH^*(X)$ into generalised eigenspaces $\QH^*(X)_\lambda$ of quantum multiplication by $c_1(X)$, discussed in Appendix~\ref{sscEval}. The breakdown of this connection explains the potentially puzzling fact that in Theorem~\ref{TheoremE} the map~\smash{$\hatCO$}, and hence $\CO_{\LL}$, is non-zero on every $\QH^*(X)_\lambda$ even though $\LL$ lies in $\Fuk(X)^\mathrm{un}_0$ and~$\lambda$ may be non-zero. For example, if $\kk$ is an algebraically closed field of characteristic $2$ and $n$ is even, then $\QH^*(\CC\PP^n)$ splits into eigenspaces with eigenvalues given by the primitive $n$th roots of unity, yet the composition
\[
\CO_{\LL}\colon \ \QH^*(\CC\PP^n) \xto{\CO_0} \HH^*(\Fuk(\CC\PP^n)^\mathrm{un}_0) \xto{\text{restrict}} \HH^*\bigl(\CF^*\bigl(\LL, \LL\bigr)\bigr)
\]
is injective.

If the real locus $L$ of $X$ is orientable, then the usual argument \cite[Proposition~2.9]{SheridanFano} shows that
\[
\CO_0\colon \ \QH^*(X) \to \HH^*(\Fuk(X)_0)
\]
vanishes on $\QH^*(X)_\lambda$ for all $\lambda \neq 0$. Combining this with Theorem~\ref{TheoremE} shows that in this case we must have $\QH^*(X) = \QH^*(X)_0$. In fact, orientability of $L$ is equivalent to vanishing of the first Stiefel--Whitney class $w_1(L) \in \rH^1(L; \ZZ/2)$, which is equivalent (via the Duistermaat isomorphism of Section~\ref{sscDuistermaat}) to the vanishing of $c_1(X)$ in characteristic $2$. So if $L$ is orientable, then $c_1(X)$ is actually zero in $\QH^*(X)$.
\end{Remark}

Injectivity of $\CO_{\LL}$, and hence the split-generation results of Theorem~\ref{TheoremE}, were previously shown by Tonkonog \cite{TonkonogCO} in several families of examples, including $\RR\PP^n \subset \CC\PP^n$. Using a totally different approach, Evans--Lekili \cite[Example~7.2.4]{EvansLekiliGeneration} proved part~\ref{itmOrientable}. The general question of split-generation by the real locus, however, has remained open.

\begin{Remark}
Split-generation of the ungraded category by an orientable Lagrangian neither implies nor is implied by split-generation of the graded category: the ungraded category includes more objects, but one is allowed to construct more (i.e., ungraded) twisted complexes and to break off more summands.
\end{Remark}

\begin{Remark}
The minimal Chern number assumption is used to ensure that the real locus has minimal Maslov number at least~$2$, and hence that it defines a valid object in the monotone Fukaya category. The same hypothesis appears in~\cite{EvansLekiliGeneration} (via the requirement that the real locus is orientable) and in~\cite{TonkonogCO}. Alston--Amorim~\cite{AlstonAmorim} have shown that even when the minimal Chern number is $1$ the real locus is Floer-theoretically unobstructed in the sense of Fukaya--Oh--Ohta--Ono~\cite{FOOOBook}, i.e., that one can equip it with a bounding cochain so that it becomes a valid object in the Fukaya category. In Section~\ref{sscNonMonotone}, we discuss some extensions of our results that incorporate bounding cochains, but they require additional assumptions related to strict unitality. We are not aware of these assumptions having been established in the case of the real locus, so we do not state our results for minimal Chern number~$1$.
\end{Remark}

\begin{maincor}
\label{CorollaryF}
In the setting of Theorem~{\rm \ref{TheoremE}}, if $K$ is any monotone Lagrangian in $X$ that supports a non-zero object in $\Fuk(X)^\mathrm{un}_\lambda$ for some $\lambda$, then $K$ is non-displaceable from the real part $L$ by Hamiltonian isotopies.
\end{maincor}

\begin{proof}
In this situation, $K \times K$ supports a non-zero object in $\Fuk(X \times X)^\mathrm{un}_0$, which is split-generated by $\LL \times \LL$ by Theorem~\ref{TheoremE} applied to $X \times X$. So $K \times K$ is non-displaceable from $L \times L$, and hence $K$ is non-displaceable from $L$.
\end{proof}

\begin{Remark}
The idea of passing to $X \times X$ to make the obstruction number $\lambda$ vanish was used by Abreu--Macarini \cite[Remark~2.9]{AbreuMacarini}, Alston--Amorim \cite[Theorem~1.3]{AlstonAmorim}, and Tonkonog \cite[Theorem~1.5]{TonkonogCO} in similar situations. In particular, Alston--Amorim \cite[Corollary~1.4]{AlstonAmorim} showed that $\LL$ is non-displaceable from the toric fibre, whilst Tonkonog \cite[Corollary~3.10]{TonkonogCO} proved the analogue of Corollary~\ref{CorollaryF} for a specific family of $X$.
\end{Remark}

\begin{Example}
Amorim--Cho \cite{AmorimCho} introduced ungraded matrix factorisations and the idea that they should be mirror to non-orientable Lagrangians. They studied $\LL = \RR\PP^2 \subset \CC\PP^2$ using the localised mirror functor associated to the monotone toric fibre $K$, and showed that it is mirror~to
\begin{gather*}
\left( \kk\bigl[x^{\pm 1}, y^{\pm 1}\bigr]^{\oplus 4}, \begin{pmatrix} 0 & 1 & 1 & \frac{1}{xy} \\ y & 0 & \frac{1}{x} & 1 \\ x & \frac{1}{y} & 0 & 1 \\ 1 & x & y & 0 \end{pmatrix}\right) \\\qquad{}
\in \mf^\mathrm{un}\bigg(\kk[\rH_1(K; \ZZ)] \cong \kk\bigl[x^{\pm 1}, y^{\pm 1}\bigr], W_K = x + y + \frac{1}{xy}\bigg).
\end{gather*}
We instead study $\LL$ using the localised mirror functor associated to $L$ itself, working over the ring $R = S / I = \kk[\rH_1(L; \ZZ/2)]$. In the case of $\LL = \RR\PP^n \subset \CC\PP^n$ we have $R \cong \kk[z] / \bigl(z^2-1\bigr)$, $W_L = 0$, and
\[
\HF^*_R(\bL, \bL) \cong \kk[h, z] / \bigl(h^{n+1} - z, z^2 - 1\bigr) \cong \kk[h] / \bigl(h^{2(n+1)} - 1\bigr),
\]
where $h$ is (the PSS image of) the real hyperplane class. Under the localised mirror functor associated to $L$, $\LL$ is sent to the ungraded matrix factorisation
\begin{gather*}
\bigl( \kk[h, z] / \bigl(h^{n+1} - z, z^2 - 1\bigr), h(z-1) \bigr) \cong \left( R^{n+1},
\begin{pmatrix} 0 & 0 & \dots & 0 & z-1 \\ z-1 & 0 & \dots & 0 & 0 \\ 0 & z-1 & \dots & 0 & 0 \\ \vdots & \vdots & \ddots & \vdots & \vdots \\ 0 & 0 & \dots & z-1 & 0 \end{pmatrix}\right)
\end{gather*}
in $\mf^\mathrm{un}(R, 0)$.
\end{Example}

Returning now to the case of general $X$ and $\kk$, having seen that the horizontal map~\smash{$\rH\bigl(\hatTheta\bigr)$} in~\eqref{eqCompletedDiagram} is an isomorphism it is natural to ask how much information the vertical map~\smash{$\rH\bigl(\hatPhi_*\bigr)$}~remembers. We are able to answer this question in two important situations. The first is when~$L$ is a torus, which is of particular interest in mirror symmetry.

\begin{mainthm}[Proposition~\ref{propThmF}]
\label{TheoremG}
If $L$ is a torus and $\calL$ is a critical point of $W_L$, then \smash{$\rH\bigl(\hatPhi_*\bigr)$} is an isomorphism.
\end{mainthm}

By arguments analogous to those used in Section~\ref{sscIndependence} to prove that~$\Theta$ is independent of auxiliary choices, this statement is also independent of choices. So it suffices to prove it for a specific choice of auxiliary data, and we can choose these data so that in fact \smash{$\hatPhi$} is itself a~quasi-isomorphism (using~\cite{SmithSuperfiltered}, which builds on~\cite{ChoHongLauTorus}), from which the result follows.

Theorem~\ref{TheoremG} has the following immediate consequence.

\begin{maincor}\label{CorollaryH}
If $L$ is a torus and $\calL$ is a critical point of $W_L$, then the map $\CO_{\LL}$ corresponds to the formal expansion of \smash{$\CO^0_\bL$} about $\calL$, after suitable identification of the codomains. In~particular, $\CO_{\LL}$ is injective on a subspace of $\QH^*(X)$ if and only if the expansion of \smash{$\CO^0_\bL$} is injective on the same subspace.
\end{maincor}

\begin{Example}\label{exToricGen}
Returning to the example of Section~\ref{sscMotivatingExample}, we know that $\CO^0_\bL$ is an isomorphism
\[
\QH^*(X) \to \Jac W_L
\]
and that its completion about each critical point $\alpha$ induces an isomorphism
\[
\QH^*(X)_\alpha \to (\Jac W_L)_\alpha.
\]
Corollary~\ref{CorollaryH} then tells us that
\[
\CO_{L_\alpha}\colon \ \QH^*(X)_\alpha \to \HH^*(\CF^*(L_\alpha, L_\alpha))
\]
is an isomorphism, so $L_\alpha$ split-generates $\Fuk(X)_\alpha$ by the generation criterion, as proved by Evans--Lekili using completely different methods.
\end{Example}

\begin{Remark}
Suppose $\QH^*(X)$ decomposes as a product of algebras $Q \times Q^\perp$. It follows from general structural properties \cite[Theorem 8]{GanatraAutomaticGeneration} that if $\LL$ lies in the $Q$-summand of $\Fuk(X)_\lambda$ and if~\smash{$\CO_{\LL}|_Q$} is injective then, in addition to $\LL$ split-generating this piece of the Fukaya category, we get that \smash{$\CO_{\LL}|_Q$} is an \emph{isomorphism}, i.e., it is automatically surjective (for more detail on these decompositions of the category see Appendix~\ref{sscDecomp}). Corollary~\ref{CorollaryH} thus has the following curious consequence: if $L$ is a torus and the formal expansion \smash{$\hatCO$} of \smash{$\CO^0_\bL$} about a critical point of $W_L$ is injective on $Q$ and vanishes on $Q^\perp$ (the latter ensures that $\LL$ lies in the $Q$-summand of the category), then \smash{$\hatCO|_Q$} is an isomorphism. In particular, these hypotheses can never hold if the critical point is non-isolated, since then \smash{$\HF^*_{\hatS}(\bL, \bL)$} is infinite-dimensional.
\end{Remark}

Corollary~\ref{CorollaryH} can also be used to prove new, non-toric, generation results for Lagrangian tori. In favourable situations this can even be done with no knowledge beyond the ambient quantum cohomology ring and $W_L$ itself. One useful tool here is the famous result of Auroux \cite[Lemma~6.7]{AurouxTDuality}, Kontsevich, and Seidel that $\CO^0_{\LL}$ sends twice the first Chern class $c_1$ of $X$ to~${2W_L(\calL) \cdot 1_{\LL}= 2\lambda \cdot 1_{\LL}}$; see Lemma~\ref{lemAKS} for a precise statement. The same argument shows the following.

\begin{maincor}\label{CorollaryI}
For any $L$ $($not necessarily a torus$)$, the map $\CO^0_\bL$ sends $2c_1$ to $2W_L \cdot 1_\bL$, and if $L$ is orientable, then we can cancel the factors of $2$. So if $L$ is a torus and $\calL$ is a critical point of $W_L$, then, under the correspondence in Corollary~{\rm \ref{CorollaryH}}, \smash{$\CO_{\LL}(c_1)$} corresponds to the formal expansion of~$W_L$ about $\calL$.
\end{maincor}

\begin{Example}
Let $X \subset \CC\PP^4$ be the quadric threefold and let us temporarily work over $\ZZ$. The quantum cohomology is
\[
\QH^*(X) \cong \ZZ[H, E]/\bigl(H^2-2E, E^2-H\bigr),
\]
with $c_1 = 3H$. $X$ admits a degeneration to the singular toric manifold whose polytope is a square-based pyramid. This gives a Lagrangian $3$-sphere $V \subset X$ as the vanishing cycle, and a disjoint monotone Lagrangian torus $L \subset X$ as the parallel transport of the barycentric torus from the central fibre. After choosing a basis for $\rH_1(L; \ZZ)$ to give an identification $S \cong \ZZ\bigl[x^{\pm 1}, y^{\pm 1}, z^{\pm 1}\bigr]$, and equipping $L$ with the standard spin structure, we have by \cite[Remark 7.1.3]{EvansLekiliGeneration} (using \cite[Theorem~1]{NishinouNoharaUeda12}) that
\[
W_L = x+y+z+\frac{1}{xy}+\frac{1}{yz}.
\]
By Corollary~\ref{CorollaryI}, we then have \smash{$\CO^0_\bL(3H) = \CO^0_\bL(c_1) = W_L \cdot 1_\bL$}. Moreover, the $S$-subalgebra of $\HF^*_S(\bL, \bL)$ generated by the unit is
\begin{align*}
\Jac W_L &= \ZZ\bigl[x^{\pm 1}, y^{\pm 1}, z^{\pm 1}\bigr] / \bigl(x-z, y-1/x^2, 2x^3-1\bigr) \\
&= \ZZ\bigl[x^{\pm 1}\bigr] / \bigl(2x^3-1\bigr) = \ZZ\bigl[\tfrac{1}{2}, x\bigr] / \bigl(x^3 - \tfrac{1}{2}\bigr).
\end{align*}
Under these identifications we get \smash{$\CO^0_\bL(3H) = 6x \in \ZZ\bigl[\tfrac{1}{2}, x\bigr] / \bigl(x^3 - \tfrac{1}{2}\bigr)$}, and hence $\CO^0_\bL(H) = 2x$.

Let us now pass back to field coefficients $\kk$, with $\Char \kk \neq 2$. Then $\QH^*(X)$ decomposes as a~direct sum of ideals \smash{$Q_0 = \bigl(1-H^3/4\bigr)$} and $Q_1 = (H, E)$, whilst $\Jac W_L$ becomes \smash{$\kk[x] / \bigl(x^3 - \tfrac{1}{2}\bigr)$}, and the calculation $\CO^0_\bL(H) = 2x$ tells us that $\CO^0_\bL$ is injective on $Q_1$. This in turn implies that~$L$ split-generates the $Q_1$-summand of the Fukaya category. More precisely:
\begin{itemize}\itemsep=0pt
\item If $\Char \kk = 3$, then $\Jac W_L$ is local, with maximal ideal $\m = (x+1, y-1, z+1)$. The $\m$-adic completion \smash{$\hatCO$} is therefore still injective on $Q_1$, so $L$ equipped with the local system defined by $x=z=-1$ and $y=1$ split-generates the $Q_1$-summand of the category.
\item If $\Char \kk \neq 3$, then, assuming $\kk$ contains all cube roots of $2$, $\Jac W_L$ contains three maximal ideals, corresponding to three local systems on $L$. Meanwhile $Q_1$ decomposes as $\kk \times \kk \times \kk$, and each completion \smash{$\hatCO$} at a maximal ideal is injective on one of these factors. So the $Q_1$-summand of the category splits into three pieces, each split-generated by $L$ with one of the three critical local systems.
\end{itemize}
\end{Example}

\begin{Remark}
When $\Char \kk \neq 2$, the $Q_0$-summand of the category is split-generated by $V$, by~\mbox{\cite[Proposition~7.1.1]{EvansLekiliGeneration}} (or by \cite[Lemma 4.6]{SmithQuadrics} if $\kk = \CC$). When $\Char \kk = 2$, the splitting into~$Q_0$ and $Q_1$ breaks down and the ring $\QH^*(X)$ is local. In this case, $V$ split-generates the whole category, again by \cite[Proposition~7.1.1]{EvansLekiliGeneration}.
\end{Remark}

Returning once more to the general setting, the second situation in which we can describe the map \smash{$\rH\bigl(\hatPhi_*\bigr)$} is when $L$ is simply connected. In this case, we have $S = \kk$, so \smash{$\CF^*_S(\bL, \bL)^{(\op)}$} reduces to \smash{$\CF^*\bigl(\LL, \LL\bigr)^{(\op)}$} and completions are irrelevant. So \smash{$\rH\bigl(\hatTheta\bigr)^{-1} \circ \rH\bigl(\hatPhi_*\bigr)$} is a map
\begin{equation}
\label{eqThmIa}
\HH^*\bigl(\CF^*\bigl(\LL, \LL\bigr)\bigr) \to \HF^*\bigl(\LL, \LL\bigr)^\op.
\end{equation}
If $\LL$ is also weakly exact, i.e., bounds no topological discs of positive symplectic area, then \smash{$\CF^*\bigl(\LL, \LL\bigr)$} reduces to $\rC^*(L)$. We can then use the isomorphism
\begin{equation}
\label{eqLoopSpace}
\HH^*(\rC^*(L)) \cong \rH_{-*}\bigl(\Lambda L^{-TL}\bigr),
\end{equation}
which holds for simply connected $L$, to view \smash{$\rH\bigl(\hatTheta\bigr)^{-1} \circ \rH\bigl(\hatPhi_*\bigr)$} as a map
\begin{equation}
\label{eqThmIb}
\rH_{-*}\bigl(\Lambda L^{-TL}\bigr) \to \rH^*(L).
\end{equation}
Here $\Lambda L$ is the free loop space of $L$ and $\Lambda L^{-TL}$ is the Thom spectrum associated to the stable normal bundle of $L$, pulled back under the evaluation-at-basepoint map ${\rm ev} \colon \Lambda L \to L$. (We have dropped the $\op$ from the codomain since $\rH^*(L)$ is graded-commutative.) Our final main result describes the maps~\eqref{eqThmIa} and \eqref{eqThmIb} algebraically and geometrically, respectively.

\begin{mainthm}\label{TheoremJ}
\quad
\begin{enumerate}\itemsep=0pt
\item[$(i)$]
If $L$ is simply connected, then
$
\rH\bigl(\hatTheta\bigr)^{-1} \circ \rH\bigl(\hatPhi_*\bigr)\colon \HH^*\bigl(\CF^*\bigl(\LL, \LL\bigr)\bigr) \to \HF^*\bigl(\LL, \LL\bigr)^\op
$
is projection to length zero $($Proposition~$\ref{propThmIa})$.
\item[$(ii)$] 
If $L$ is simply connected and weakly exact, then \smash{$\rH\bigl(\hatTheta\bigr)^{-1} \circ \rH\bigl(\hatPhi_*\bigr)$} is identified with
$\ev_* \colon \rH_{-*}\bigl(\Lambda L^{-TL}\bigr) \to \rH_{-*}\bigl(L^{-TL}\bigr) \cong \rH^*(L)$ $($Proposition~$\ref{propThmIb})$.
\end{enumerate}
\end{mainthm}

\begin{Remark}Strictly the projection to length zero from $\HH^*\bigl(\CF^*\bigl(\LL, \LL\bigr)\bigr)$ lands in \linebreak $\HF^*\bigl(\LL, \LL\bigr)$, rather than \smash{$\HF^*\bigl(\LL, \LL\bigr)^\op$}. But it actually lands in the graded-centre of \linebreak $\HF^*\bigl(\LL, \LL\bigr)$, which is the same as the graded-centre of \smash{$\HF^*\bigl(\LL, \LL\bigr)^\op$}, so Theorem~\ref{TheoremJ} makes sense.
\end{Remark}

\begin{Remark}We have been unable to locate the original reference for \eqref{eqLoopSpace}, but it seems to have been known to experts in the 1980s (if not earlier) at the level of modules, i.e., ignoring the string topology product on $\Lambda L$. We shall refer to \cite[Corollary~11]{CohenJones}, since it proves the isomorphism as algebras and gives an explicit description of the map that we will use. Our grading of Hochschild cohomology is the negative of that used in~\cite{CohenJones}, so we obtain $\rH_{-*}$ in place of their~$\rH_*$.
\end{Remark}

\begin{Remark}\label{rmkExtensions}
Some of our results apply slightly more generally than stated. In particular:
\begin{enumerate}\itemsep=0pt
\item We could allow $X$ to be non-compact (but keeping $L$ compact) in Theorems~\ref{TheoremA},~\ref{TheoremB},~\ref{TheoremC},~\ref{TheoremG} and~\ref{TheoremJ}, and Corollaries~\ref{CorollaryH} and~\ref{CorollaryI}, as long as it is tame at infinity. If $X$ is a Liouville manifold, then we could also replace quantum cohomology with symplectic cohomology in the domain of all versions of the closed-open map. However, compactness is necessary for the generation criterion in the form of Theorem~\ref{thmGeneration}, since its proof relies on Poincar\'e duality, so Theorem~\ref{TheoremE}, Corollary~\ref{CorollaryF} and Example~\ref{exToricGen} also require compactness. Similarly, Tonkonog's criterion is based in the compact setting so Theorem~\ref{TheoremD} is too.
\item We could have relaxed the assumption that $\kk$ is a field in some places. Theorem~\ref{TheoremA} holds over an arbitrary ground ring with the same proof (modulo Remark~\ref{rmkRingUnitality}), and the commutative diagram in Theorem~\ref{TheoremB} holds over any ring if one defines the maps with compatible choices of auxiliary data. This compatibility is necessary because over a general ring it is not clear that the Hochschild cohomology of the Fukaya category is independent of choices (the usual proof of independence uses the spectral sequence associated to the length filtration on Hochschild cochains, whose first page---over a field!---is the Hochschild cohomology of the cohomology category).
\item Given a class in $\rH^2(X; \kk^\times)$, represented by a cocycle $b$, we could modify the quantum (or symplectic) cohomology and Fukaya category of $X$ by weighting every pseudoholomorphic curve $u$ by $b(u)$, as in Appendix~\ref{sscQHcoeff}. In the special case where $B$ is a closed $\CC$-valued $2$-form on $X$ and $b = e^{B}$, this modification is equivalent to turning on the \emph{B-field} $B$. All of our arguments extend automatically to this setting with the obvious modifications. Similarly, one can equip $X$ with a \emph{background class} in $\rH^2(X; \ZZ/2)$, represented by a cocycle~$\mathsf{b}$, and then consider Lagrangians equipped with \emph{relative} pin structures with respect to~$\mathsf{b}$. Again, everything goes through straightforwardly. We will not discuss these options any further.
\end{enumerate}
\end{Remark}

\subsection{(Non-)monotonicity}\label{sscNonMonotone}

So far, we have restricted attention to monotone symplectic manifolds and Lagrangians, owing to three complications that arise in the non-monotone case. First, the Fukaya category and closed-open map have not been defined in general outside the monotone setting, and the generation criterion has not been proved. These foundational issues are expected to be addressed in work of Abouzaid--Fukaya--Oh--Ohta--Ono, and we shall simply take them as black boxes via Assumption~\ref{MonAss} below. Second, there are extra algebraic complexities arising from the need to filter and complete various objects, to ensure convergence of curve counts. One consequence of this is that rather than working with the analogue of $S$, we work with unspecified algebras satisfying certain natural conditions. We will call these algebras $R$, as they play a similar role to the $R = S/I$ appearing previously. Third, an additional unobstructedness hypothesis is needed, which is automatic for monotone $L$ but is in general a non-trivial geometric condition. Modulo these issues, Theorems~\ref{TheoremA}, \ref{TheoremB} and~\ref{TheoremC} naturally generalise to the non-monotone case as we describe momentarily. Despite the apparent algebraic and geometric restrictions, these results are strong and flexible enough to apply to interesting examples, which we demonstrate with the case of toric fibres.

To start setting up our results, in place of $\kk$ we take our ground field to be the Novikov field
\[
\Lambda = \CC \bigl[\hspace{-1mm}\bigl[ T^\RR \bigr]\hspace{-1mm}\bigr] = \Biggl\{ \sum_{j = 1}^\infty a_j T^{s_j}\mid a_j \in \CC \text{ and } s_j \in \RR \text{ with } s_j \to \infty\Biggr\}.
\]
This carries a natural decreasing filtration by additive subgroups $\Lambda_{\geq s}$, for $s \in \RR$, comprising those series $\sum a_j T^{s_j}$ with $s_j \geq s$ for all $j$. We will require that the objects we work with are filtered and complete in a similar way, made precise in Definition~\ref{defFiltrations}.

We now introduce our technical assumptions for later reference.

\begin{assumption}\label{MonAss}
Sufficient technical foundations have been established to
\begin{enumerate}\itemsep=0pt
\item\label{ass1} Define the Fukaya category $\Fuk(X)_\lambda$ over $\Lambda$ and prove that the operations satisfy the $A_\infty$-relations. We assume moreover that the category is strictly unital. Objects are La\-gran\-gians~$L$ equipped with analogous auxiliary data to the monotone case (the local system must now be $c$-filtered in the sense of Definition~\ref{defcfiltered}), plus a weak bounding cochain~${b \in \CF^\mathrm{odd}\bigl(\LL, \LL\bigr)}$ in sufficiently positive filtration level, whose curvature is $\lambda \in \Lambda$.
\item\label{ass2} Construct $\CO_\lambda \colon \QH^*(X) \to \HH^*(\Fuk(X)_\lambda)$ as a unital $\Lambda$-algebra homomorphism. We assume that the construction can be made to land in strictly unital Hochschild cochains.
\item
 Prove a generation criterion for summands of $\Fuk(X)_\lambda$ based on injectivity of
\[
\CO_{\LL}\colon \ \QH^*(X) \to \HH^*\bigl(\CF^*\bigl(\LL, \LL\bigr)\bigr)
\]
on factors of $\QH^*(X)$.
\end{enumerate}
\end{assumption}

\begin{Remark}
\label{rmkStrictUnit}
The reader may be concerned about the assumptions of strict unitality, but for applications to split-generation it suffices to construct a strictly unital model for the Floer algebra and $\CO_\lambda$ map for the specific Lagrangian $L$ of interest. In our main example, when $L$ is a toric fibre, such a model is provided by \cite{FOOOToricI}.
\end{Remark}

Suppose Assumption~\ref{MonAss}\,\ref{ass1} holds and fix a complete filtered augmented $\Lambda$-algebra $R$. The augmentation $R \to \Lambda$ corresponds to a closed maximal ideal $\m$ in $R$. Using the same machinery as for $\Fuk(X)_\lambda$ we can define the Fukaya category $\Fuk_R(X)_r$ over $R$ for any $r \in R$.

Fix an object \smash{$\bigl(\LL, b\bigr) \in \Fuk(X)_\lambda$}.

\begin{Definition}
We say \smash{$\bigl(\LL, b\bigr)$} \emph{lifts to $R$} if the following holds. There exists a $c$-filtered rank-$1$ local system $\calL_R$ over $R$ on $L$, whose reduction modulo $\m$ is the local system $\calL$ on \smash{$\LL$}. We denote~$L$ equipped with this local system \big(and the same grading and pin structure as \smash{$\LL$}\big) by $\bL$. We then require that \smash{$\bL$} admits a weak bounding cochain $\mathbf{b}$ whose reduction mod $\m$ is $b$. In this case, we denote the curvature of $\mathbf{b}$ by $W_L \in R$. Note that there may be multiple choices for $\calL_R$ and $\mathbf{b}$, and that $W_L$ depends on these choices, although our notation does not reflect this.
\end{Definition}

Assume now that $\bigl(\LL, b\bigr)$ lifts to $R$, and fix a choice of $\calL_R$ and $\mathbf{b}$. We can then mimic the construction of $\LMF[L,\lambda]$ to give a localised mirror functor
\[
\LMF[(\bL, \mathbf{b}), \lambda]\colon \ \Fuk(X)_\lambda \to \mf(R, W_L - \lambda).
\]
This allows us to define $\rCC^*(\Fuk(X)_\lambda, \mf(R, W_L - \lambda))$, and the analogue of Theorem~\ref{TheoremA} is as follows.

\begin{mainthmp}[Proposition~\ref{propThmAp}]
\label{TheoremAp}
Suppose Assumption~$\ref{MonAss}\,\ref{ass1}$ holds and $\bigl(\LL, b\bigr)$ lifts to $R$. Given a choice of lift $(\bL, \mathbf{b})$, there is a geometrically defined $R$-linear $A_\infty$-algebra homomorphism
\[
\Theta\colon \ \CF^*_R((\bL, \mathbf{b}), (\bL, \mathbf{b}))^\op \to \rCC^*(\Fuk(X)_\lambda, \mf(R, W_L - \lambda)).
\]
It is cohomologically unital and extends the module action of $\CF^*_R((\bL, \mathbf{b}), (\bL, \mathbf{b}))^\op$ on $\LMF[(\bL, \mathbf{b}), \lambda]$.
\end{mainthmp}

\begin{Remark}
In order to define $\LMF[(\bL, \mathbf{b}), \lambda]$ and $\Theta$, there was no need to start with $\bigl(\LL, b\bigr)$ or to assume that $R$ was augmented. We could have just taken $R$ to be a complete filtered $\Lambda$-algebra (not augmented), $W_L$ to be an arbitrary element of $R$, and $(\bL, \mathbf{b})$ to be an arbitrary object in~$\Fuk_R(X)_{W_L}$. We could have done analogously in the monotone setting too. However, we chose to start with $\bigl(\LL, b\bigr)$ since our main interest is in showing that $\bigl(\LL, b\bigr)$ split-generates a~piece of the Fukaya category.
\end{Remark}

If Assumption~\ref{MonAss}\,\ref{ass2} also holds, then we can similarly define
\[
\CO^0_{(\bL, \mathbf{b})}\colon \ \QH^*(X) \to \HF^*_R((\bL, \mathbf{b}), (\bL, \mathbf{b}))^{(\op)},
\]
and Theorem~\ref{TheoremB} adapts in the obvious way.

\begin{mainthmp}[Proposition~\ref{propThmBp}]
\label{TheoremBp}
In the setting of Theorem~{\rm \ref{TheoremAp}}, suppose Assumption~$\ref{MonAss}\,\ref{ass2}$ also holds. Then the following diagram commutes:
\[
\begin{tikzcd}[column sep=5em, row sep=2.5em]
\QH^*(X) \arrow{r}{\CO_\lambda} \arrow{d}{\CO^0_{(\bL, \mathbf{b})}} & \HH^*(\Fuk(X)_\lambda) \arrow{d}{\rH((\LMF[(\bL, \mathbf{b}),\lambda])_*)}
\\ \HF_R^*((\bL, \mathbf{b}), (\bL, \mathbf{b}))^\op \arrow{r}{\rH(\Theta)} & \HH^*(\Fuk(X)_\lambda, \mf(R, W_L - \lambda)).
\end{tikzcd}
\]
\end{mainthmp}

\begin{Remark}
Recall that the main point of Theorem~\ref{TheoremB} was to make \smash{$\CO_\lambda$}, and in particular~\smash{$\CO_{\LL}$}, more amenable to computation. If $\mathbf{b}$ is non-zero, then \smash{$\CO^0_{(\bL, \mathbf{b})}$} is rather difficult to calculate because it suffers from the problem of repeated inputs, which is what we wanted to avoid. So we will mostly be interested in the case where $\mathbf{b}$ (and hence $b$) vanishes. We previously used a~pearl model for computations, but we now have in mind a de Rham model as in \cite{FOOOToricI}.
\end{Remark}

In the monotone case, the ring $R = S/I$ is automatically Noetherian, which is used in the proof of Theorem~\ref{TheoremC}. In our present setting, we have to include this assumption separately, but the analogue of Theorem~\ref{TheoremC} then holds.

\begin{mainthmp}[Proposition~\ref{propThmCp}]\label{TheoremCp}
In the setting of Theorem~{\rm \ref{TheoremAp}}, suppose $R$ is Noetherian. Then the $\m$-adically completed map \smash{$\hatTheta$} is a quasi-isomorphism. So, under Assumption~$\ref{MonAss}\,\ref{ass2}$, to prove that \smash{$\CO_{\LL}$} is injective on a subspace of \smash{$\QH^*(X)$} it suffices to prove that \smash{$\hatCOb$} is injective on that subspace.
\end{mainthmp}

\begin{Remark}
In our main application, $R$ will be local and finite-dimensional over $\Lambda$, in which case it is automatically Noetherian and $\m$-adically complete; see Lemma~\ref{lemFinDimProperties}\,\ref{FDNoetherianComplete}.
\end{Remark}

From Theorem~\ref{TheoremCp}, we obtain a generation criterion in terms of \smash{$\hatCOb$}.

\begin{maincor}\label{CorollaryK}
Suppose that Assumption~$\ref{MonAss}$ holds in full, that $R$ is a complete filtered augmented $\Lambda$-algebra which is Noetherian, and that \smash{$\bigl(\LL, b\bigr)$} is an object in $\Fuk(X)_\lambda$ which lifts to~$R$ as $(\bL, \mathbf{b})$. If \smash{$\hatCOb$} is injective on a factor of $\QH^*(X)$, then \smash{$\bigl(\LL, b\bigr)$} split-generates the corresponding summand of $\Fuk(X)_\lambda$.
\end{maincor}

Using this, and by taking $R$ to be a local factor of $\QH^*(X)$, we show the following.

\begin{mainthm}[Proposition~\ref{propThmL}]\label{TheoremL}
Under Assumption~$\ref{MonAss}$, suppose $X$ is a compact toric manifold,~$R$~is a local factor of $\QH^*(X)$, and $\lambda$ is the eigenvalue of the generalised eigenspace of $c_1 \qprod$ containing $R$. Then the $R$-summand of $\Fuk(X)_\lambda$ is split-generated by a~specific toric fibre with a~specific rank-$1$ local system.
\end{mainthm}

As mentioned in Section~\ref{sscMotivatingExample}, we believe work announced by Abouzaid--Fukaya--Oh--Ohta--Ono around 2016 establishes a version of Assumption~\ref{MonAss} and also considers this toric generation question.

\begin{Remark}
If we also assume that $\CO_\lambda$ respects the eigenvalue decomposition of $\QH^*(X)$, in the sense that it vanishes on the generalised $\lambda'$-eigenspace of $c_1 \qprod$ if $\lambda \neq \lambda'$, then we conclude that in fact all pieces of the Fukaya category of $X$ are split-generated by toric fibres with rank-$1$ local systems.
\end{Remark}

\begin{Remark}
In a similar spirit to Remark~\ref{rmkExtensions}, we could allow $X$ to be non-compact in Theorems~\ref{TheoremAp}, \ref{TheoremBp} and~\ref{TheoremCp} if a suitable form of Assumption~\ref{MonAss} holds. And we could allow quantum cohomology and Floer theory to be modified by bulk deformations everywhere.
\end{Remark}

\subsection{Structure of the paper}

Section~\ref{secSetup} describes the technical setup for the paper, which is brought together in Section~\ref{secThmsAB} to construct $\Theta$ and prove Theorems~\ref{TheoremA} and~\ref{TheoremB}. Section~\ref{secThmC}, which is almost entirely algebraic, then proves Theorem~\ref{TheoremC}. The remaining sections can be read somewhat independently of each other: Section~\ref{secTonk} looks at Tonkonog's criterion in light of our results and proves Theorem~\ref{TheoremD}; Section~\ref{secRealToric} studies the case of real toric Lagrangians and proves that they split-generate the Fukaya category in characteristic $2$ (Theorem~\ref{TheoremE}); Section~\ref{sechatPhi} analyses the map $\hatPhi$ in the settings of Theorems~\ref{TheoremG} and~\ref{TheoremJ}; and Section~\ref{secNonMonotone} discusses the non-monotone versions of our results, as sumarised in Section~\ref{sscNonMonotone}, including the proof that toric fibres split-generate (Theorem~\ref{TheoremL}). The paper ends with three appendices, covering respectively the generation criterion, some $A_\infty$-algebra lemmas used in the proof of Theorem~\ref{TheoremC}, and the quantum cohomology of monotone toric manifolds as needed for Section~\ref{secRealToric}.

\section{Setup}
\label{secSetup}

Fix from now on a monotone symplectic $2n$-manifold $(X, \omega)$ which is compact or tame at infinity and equipped with a $\ZZ/2$-grading \cite{SeidelGraded}. Fix also a coefficient field $\kk$. The Fukaya category and all (co)homology groups and algebraic operations will be over $\kk$ unless stated otherwise. In this section we summarise the background ideas we will use later. (In Section~\ref{secNonMonotone}, we will drop the monotonicity condition and work over the Novikov field $\Lambda$ instead of $\kk$, but we will flag this at the time.)

\subsection{Conventions}
\label{sscSignConventions}

Recall that we mostly follow Seidel's conventions for $A_\infty$-algebras and -categories \cite{SeidelBook}. (We~will mainly discuss categories, with algebras given by the special case where there is only one object.) In particular, the differential $\diff$ on a dg-category is related to the $\mu^1$ operation on the corresponding $A_\infty$-category by
\begin{equation}
\label{eqdgAinftyDiff}
\mu^1(a) = (-1)^{\lvert a \rvert} \diff a.
\end{equation}
Similarly, the composition $\circ$ on a dg-category or its cohomology is related to the $\mu^2$ operation~via
\begin{gather}
\label{eqdgAinftyProd}
\mu^2(a_2, a_1) = (-1)^{\lvert a_1 \rvert} a_2 \circ a_1.
\end{gather}
Higher order compositions are written from right to left, for example
\[
\mu^k\colon \ \hom^*(O_{k-1}, O_k) \otimes \dots \otimes \hom^*(O_0, O_1) \to \hom^*(O_0, O_k).
\]
We write $\hom^*$ for chain-level morphism spaces and $\Hom^*$ for their cohomology.

The place where we differ from \cite{SeidelBook} is in our signs for the opposite $A_\infty$-algebra or -category.

\begin{Definition}
\label{defOpposite}
Given an $A_\infty$-category $\calC$, with operations $\mu_\calC$, its \emph{opposite category} $\calC^\op$ has the same objects as $\calC$, and the same morphism spaces as graded modules. Its operations $\mu_{\calC^\op}$ are defined by
\begin{equation}
\label{eqopdefinition}
\mu_{\calC^\op}^l(c_l, \dots, c_1) = (-1)^{\triangle_l + l - 1} \mu_\calC^l(c_1, \dots, c_l)
\end{equation}
for all $c_1, \dots, c_l$, where
\[
\triangle_l = \sum_{i < j} (\lvert c_i \rvert - 1) (\lvert c_j \rvert - 1).
\]
These are easily checked to satisfy the $A_\infty$-relations, and if $\calC$ is cohomologically unital then so is $\calC^\op$, with the same cohomological units.
\end{Definition}

\begin{Lemma}
\label{lemOpposite}
The cohomology category $\rH^*(\calC^\op)$ is the ordinary graded-opposite category $\rH^*(\calC)^\op$ to $\rH^*(\calC)$.
\end{Lemma}

\begin{Remark}
By \emph{graded}-opposite we mean that it includes Koszul signs.
\end{Remark}

\begin{proof}[Proof of Lemma~\ref{lemOpposite}]
Let $\circ_\calC$, $\circ_{\calC^\op}$, and $\circ_\calC^\op$ denote the compositions in $\rH^*(\calC)$, $\rH^*(\calC^\op)$, and $\rH^*(\calC)^\op$, respectively. For cocycles $c_2$ and $c_1$ in $\calC$ or $\calC^\op$, representing cohomology classes $[c_2]$ and~$[c_1]$, we have
\begin{align*}
[c_2] \circ_{\calC^\op} [c_1] &= (-1)^{\lvert c_1 \rvert}\bigl[\mu^2_{\calC^\op}(c_2, c_1)\bigr] = (-1)^{\lvert c_2 \rvert\lvert c_1 \rvert + \lvert c_2 \rvert}\bigl[\mu^2_\calC(c_1, c_2)\bigr] \\
&= (-1)^{\lvert c_2 \rvert \lvert c_1 \rvert}[c_1] \circ_\calC [c_2] = [c_2] \circ_\calC^\op [c_1].\qedhere
\tag*{\qed}
\end{align*}
\renewcommand{\qed}{}
\end{proof}


\subsection{The Fukaya category}
\label{sscFuk}

Next we recap the definition of the monotone Fukaya category.

Recall that a Lagrangian submanifold $L \subset X$ is monotone if there is a positive real constant~$\tau$ such that the Maslov index and area homomorphisms
$\mu, \omega\colon \pi_2(X, L) \to \RR$
are related by $\omega = \tau \mu$. Recall also that the minimal Maslov number $N_L$ of $L$ is defined to be the positive generator of the image $\mu(\pi_2(X, L)) \subset \ZZ$ if this image is non-zero, and infinity otherwise.

For each $\lambda \in \kk$, the compact monotone Fukaya category $\Fuk(X)_\lambda$ is $\ZZ/2$-graded $A_\infty$-category defined as follows (see \cite{RitterSmith, SheridanFano} for details). The objects are Lagrangians $L$ in $X$ such that
\begin{enumerate}\itemsep=0pt
\item\label{itm1} $L$ is compact, connected, and monotone, with minimal Maslov number $N_L$ at least $2$.
\item\label{itm2} The image of $\pi_1(L)$ in $\pi_1(X)$ is trivial.
\item\label{itm3} $L$ is equipped with a pin structure (not needed if $\Char \kk = 2$) and a $\ZZ/2$-grading with respect to the $\ZZ/2$-grading of $X$.
\item\label{itm4} $L$ carries a rank-$1$ local system $\calL$ over $\kk$ such that $W_L(\calL) = \lambda$, where $W_L$ is the superpotential of $L$ as defined as in Section~\ref{sscIngredients}.
\end{enumerate}
Formally, one usually fixes a finite or countable collection of such Lagrangians to work with, although each may be allowed to carry arbitrarily many different local systems.

\begin{Remark}
\label{rmkUngraded}
We drop the $\ZZ/2$-gradings on $X$ and $L$ if $\Char \kk = 2$, and obtain the \emph{ungraded} Fukaya category $\Fuk(X)_\lambda^\mathrm{un}$ appearing in Theorem~\ref{TheoremE}. We will give all arguments assuming the existence of $\ZZ/2$-gradings, but these are really only used to define various signs in our formulae, and the same arguments go through in the ungraded case.
\end{Remark}

\begin{Remark}
\label{rmkMonotonicity}
Some authors replace our definition of monotonicity with the slightly stronger condition that the Maslov index and area are positively proportional on $\rH_2(X, L; \ZZ)$ rather than $\pi_2(X, L)$, or that the Maslov class and symplectic form are positively proportional in $\rH^2(X, L; \RR)$. Assumption \ref{itm2} can then be weakened to $\rH_1(L; \ZZ)$ being trivial in $\rH_1(X; \ZZ)$.
\end{Remark}

\begin{Remark}
\label{rmkGradings}
Recall from \cite{SeidelGraded} that a $\ZZ/2$-grading of $X$ is a homotopy class of square root of the complex line bundle \smash{$\bigl(\Lambda_\CC^n TX\bigr)^{\otimes 2}$}. There is always a canonical choice, namely \smash{$\Lambda_\CC^n TX$}, and other choices are all obtained by twisting by an element $t$ of $\rH^1(X; \ZZ/2)$, i.e., tensoring with an isomorphism class of real line bundle $\mathcal{T}$ on $X$. A $\ZZ/2$-grading of $L$ with respect to this $\ZZ/2$-grading of $X$ is then a homotopy class of trivialisation of \smash{$\bigl(\Lambda_\RR^n TL\bigr) \otimes \mathcal{T}|_L$}, i.e., an orientation of~$L$ twisted by $\mathcal{T}$. In light of \ref{itm2}, or its variant in Remark~\ref{rmkMonotonicity}, the bundle $\mathcal{T}|_L$ is trivial, so $L$ has a $\ZZ/2$-grading if and only if it is orientable.
\end{Remark}

The morphism spaces are Floer cochain complexes and the $A_\infty$-operations are defined by counting rigid pseudoholomorphic discs using a choice of regular perturbation data. (We will freely use the terminology of Floer data and perturbation data from \cite[Section~(8e)]{SeidelBook}, but the casual reader uninterested in technicalities can safely ignore such discussions.)

\begin{Definition}
We write $\mu_\Fuk$ for the $A_\infty$-operations on $\Fuk(X)_\lambda$.
\end{Definition}

From now on, all Lagrangians will be assumed to satisfy \ref{itm1}--\ref{itm3}. We will write $\LL$ for a~pair~$(L, \calL)$ as in \ref{itm4}, as we may want to distinguish the submanifold $L \subset X$ from the object~$\LL$ of~$\Fuk(X)_\lambda$, but for Lagrangians named with other letters we will gloss over this distinction and simply say things like `take a Lagrangian $K$ in $\Fuk(X)_\lambda$'.

There are different geometric models available for Floer cochain complexes in this setting. For arbitrary Lagrangians $K_1$ and $K_2$, we may take $\CF^*(K_1, K_2)$ to be generated by intersection points between $K_1$ and $K_2$, if they intersect transversely, or more generally by Hamiltonian chords from $K_1$ to $K_2$. This chord model is the one set up by Seidel \cite{SeidelBook} and adapted to the monotone case by Sheridan \cite{SheridanFano} and Ritter--Smith \cite{RitterSmith}. In this model, inputs and outputs of pseudoholomorphic discs are boundary punctures asymptotic to chords. In the special case when $K_1$ and $K_2$ are equal to a common $\LL$, we can instead use a \emph{pearl model}, as developed by Cornea--Lalonde \cite{CorneaLalonde}, Biran--Cornea \cite{biran2007quantum}, and Sheridan \cite{SheridanCY}. Here generators of $\CF^*\bigl(\LL, \LL\bigr)$ are critical points of a Morse function on $L$, and $A_\infty$-operations count `pearly trees' built from Morse flowlines and pseudoholomorphic discs.

In \cite{SmithSuperfiltered}, following Cho--Hong--Lau \cite{ChoHongLauTorus}, we used the pearl model for $\CF^*\bigl(\LL, \LL\bigr)$ but could equally well have used the chord model as long as the perturbations were chosen appropriately. Likewise, in the present paper we are free to use either model. We will describe moduli spaces and draw pictures in the chord model, i.e., using discs without Morse trees, as they are cleaner to work with, but everything can be straightforwardly translated by the reader who prefers the pearl model.


\subsection[The Fukaya category over S]{The Fukaya category over $\boldsymbol{S}$}
\label{sscFukS}

Fix an object $\LL = (L, \calL)$ in $\Fuk(X)_\lambda$, with superpotential $W_L$ in $S = \kk[\rH_1(L; \ZZ)]$. The Fukaya category over $S$, denoted by $\Fuk_S(X)$ and with Floer complexes and cohomology denoted by $\CF_S^*$ and $\HF_S^*$, has already appeared informally and is defined completely analogously to the category over $\kk$. We denote the operations on the category by $\mu_S$. Recall that we may implicitly view objects of $\Fuk(X)_\lambda$ as objects of $\Fuk_S(X)_\lambda$ by extending scalars from $\kk$ to $S$, i.e., applying $S\otimes_{\kk} {-}$ to their local systems.

Concretely, if the $L_i$ are Lagrangians in $\Fuk_S(X)$, equipped with rank-$1$ local systems $\calE_i$ over~$S$, then
\[
\CF^*((L_i, \calE_i), (L_j, \calE_j)) = \bigoplus_{\substack{\text{Ham chords $\gamma$} \\ \text{from $L_i$ to $L_j$}}} \Hom_S\bigl(\calE_i|_{\gamma(0)}, \calE_j |_{\gamma(1)}\bigr).
\]
The differential and $A_\infty$-operations count pseudoholomorphic discs, using parallel transport around boundary segments to map between the fibres in these $\Hom_S$ spaces.

Recall also that the object $\bL \in \Fuk_S(X)_{W_L}$, central to our results, is given by $L$ with the tautological local system, whose monodromy around a loop $\gamma \in \rH_1(L; \ZZ)$ is the monomial in~$S$ corresponding to $\gamma$. The Floer algebra of $L$ with coefficients is $\CF_S^*(\bL, \bL)$, and $\CF_S^*(\bL, \bL)^\op$ is obtained from this by reversing the order of inputs and twisting signs as in Definition~\ref{defOpposite}. We~denote the operations on $\CF_S^*(\bL, \bL)^\op$ by $\mu_\bL$.

We could try to interpret $\CF_S^*(\bL, \bL)$ purely within $\Fuk(X)$ by considering the object $(L, \calS)$, where $\calS$ is the universal abelian local system over $\kk$. Explicitly, this would involve equipping~$L$ with the rank-$\lvert \rH_1(L; \ZZ) \rvert$ local system over $\kk$ whose fibres are modelled on $S$, and whose monodromy around $\gamma \in \rH_1(L; \ZZ)$ is again given by multiplication by the corresponding monomial. Use of such higher-rank local systems is permitted in principle but leads to various complications~\cite{Konstantinov}. In any case, the resulting Floer complexes do not give what we want: $\CF^*((L, \calS), (L, \calS))$ has the right differential but the underlying module is too large, as each Floer generator gives rise to a copy of $\End_{\kk}(S)$ instead of $S$; and $\CF^*\bigl((L, \calS), \LL\bigr)$ has the right module but the differential is wrong.

\begin{Remark}
The complex $\CF^*\bigl((L, \calS), \LL\bigr)$ coincides with $\CF_S^*\bigl(\bL, \LL\bigr)$ and can be interpreted as the lifted Floer complex of $L$, in the sense of Damian \cite{Damian}, associated to the universal abelian cover.
\end{Remark}


\subsection{The localised mirror functor}
\label{sscLMF}

Next we describe more precisely the localised mirror functor
\[
\LMF[L,\lambda]\colon \ \Fuk(X)_\lambda \to \mf(S, W_L-\lambda)
\]
mentioned in Section~\ref{sscIngredients}. For further details see the original paper of Cho--Hong--Lau \cite{ChoHongLauTorus}. Related ideas appear in \cite{Damian}. From now on, we will abbreviate $\LMF[L,\lambda]$ and $\mf(S, W_L - \lambda)$ simply to $\LMF$ and $\mf$ respectively. We write $\mu_\mf^1$ and $\mu_\mf^2$ for the operations on $\mf$.

Recall that the $A_\infty$-Yoneda embedding \cite{FukayaFloerHomologyFor3Manifold, FukayaFloerHomologyAndMirrorSymmetryII} associates to each object $\LL \in \Fuk(X)_\lambda$ an~$A_\infty$-module $Y_{\LL}$ over $\Fuk(X)_\lambda$. By definition, such a module is an $A_\infty$-functor from $\Fuk(X)_\lambda$ to the dg-category of chain complexes over $\kk$. The specific module $Y_{\LL}$ is defined to send each object~${K \in \Fuk(X)_\lambda}$ to the cochain complex $\CF^*\bigl(\LL, K\bigr)$ with differential \smash{$\diff a = (-1)^{\lvert a \rvert}\mu_\Fuk^1(a)$}; the~\smash{$(-1)^{\lvert a \rvert}$} is our standard sign change for passing between dg- and $A_\infty$-language as in \eqref{eqdgAinftyDiff}. For each~${k \geq 1}$, the $k$th component
\[
Y_{\LL}^k\colon \ \CF^*(K_{k-1}, K_k)[1] \otimes \dots \otimes \CF^*(K_0, K_1)[1] \to \hom^*\bigl(\CF^*\bigl(\LL, K_0\bigr), \CF^*\bigl(\LL, K_k\bigr)\bigr)[1]
\]
is then given by the $\mu_\Fuk^{k+1}$ operation, with sign twisted, via
\begin{equation}
\label{eqYonedaSignTwist}
Y_{\LL}^k(a_k, \dots, a_1)(a_0) = (-1)^{\lvert a_0 \rvert}\mu_\Fuk^{k+1}(a_k, \dots, a_1, a_0).
\end{equation}

\begin{Remark}
With these conventions, $Y_{\LL}$ is a left module and the Yoneda embedding is contravariant. One could instead, with suitable sign changes, take $Y_{\LL}(K) = \CF^*\bigl(K, \LL\bigr)$, which would give a right module and make the Yoneda embedding covariant. We use the above conventions because it is really the module $Y_{\LL}$ that we care about. In particular, we want this module to be covariant when viewed as a functor to chain complexes.
\end{Remark}

\begin{Definition}\label{defLMF}
The localised mirror functor $\LMF$ is defined in exactly the same way as $Y_{\LL}$, but with $\bL$ in place of $\LL$ and $\CF_S^*$ in place of $\CF^*$. Explicitly, $K \in \Fuk(X)_\lambda$ is sent to $\CF_S^*(\bL, K)$ with `differential' $\diff a = (-1)^{\lvert a \rvert} \mu_S^1(a)$. Following~\cite{SmithSuperfiltered}, we refer to this $\diff$ as the \emph{squifferential}. The usual proof that $\diff^2 = 0$ in monotone Floer theory shows that on $\CF_S^*(\bL, K)$ we have $\diff^2 = W_L - \lambda$. Morally, this is because
\[
\diff^2 a = -\mu_S^1\bigl(\mu_S^1(a)\bigr) = \mu_S^2\bigl(a, \mu_\bL^0\bigr) + (-1)^{\lvert a \rvert - 1}\mu_S^2\bigl(\mu_K^0, a\bigr) = a \cdot W_L - \lambda \cdot a.
\]
The components \smash{$\LMF^k$} are similarly given by the \smash{$\mu_S^{k+1}$} operation, twisted by a sign as in \eqref{eqYonedaSignTwist}. The fact that these do indeed define an $A_\infty$-functor is shown in \cite[Theorem 6.4]{ChoHongLauTorus}, with different sign conventions, and in \cite[Lemma 3.3]{SmithSuperfiltered} in a slightly different setting but with the same sign conventions. We sketch the argument pictorially below.
\end{Definition}

\begin{Remark}
In characteristic $2$, the localised mirror functor can be extended to the ungraded Fukaya category, now landing in ungraded matrix factorisations~\cite{AmorimCho}.
\end{Remark}

\subsection{Diagrams for operations}
\label{sscDiagrams}

We depict the $\mu_\Fuk^k$ operation on $\Fuk(X)_\lambda$ schematically as a disc with $k$ Floer inputs (open circles) and one Floer output (an open square) on the boundary, as shown in the left-hand part of Figure~\ref{figDiscs}.
\begin{figure}[t]
\centering
\begin{tikzpicture}
\def\r{1.8cm}
\begin{scope}[xshift=-5.5cm]
\draw[line width=\circwidth] (0, 0) circle[radius=\r];
\draw (-180: \r) node[flout]{};
\draw (-130:\r) node[flin]{};
\draw (-60: \r) node[flin]{};
\draw (-10: \r) node[flin]{};
\draw (30: \r) node[flin]{};
\draw (110: \r) node[flin]{};
\end{scope}
\begin{scope}
\draw[line width=\circwidth] (0, 0) circle[radius=\r];
\draw[line width=\LLwidth] (-180: \r) arc[start angle=-180, end angle=0, radius=\r];
\draw (-180: \r) node[flout]{};
\draw (0:\r) node[flin]{};
\end{scope}
\begin{scope}[xshift=5.5cm]
\draw[line width=\circwidth] (0, 0) circle[radius=\r];
\draw[line width=\LLwidth] (-180: \r) arc[start angle=-180, end angle=0, radius=\r];
\draw (-180: \r) node[flout]{};
\draw (0:\r) node[flin]{};
\draw (40: \r) node[flin]{};
\draw (80: \r) node[flin]{};
\draw (140: \r) node[flin]{};
\end{scope}
\end{tikzpicture}
\caption{The $A_\infty$-operation $\mu_\Fuk^5$ (left), the squifferential on $\LMF(K) = \CF_S^*(\bL, K)$ (centre), and the component $\LMF^3$ (right).}\label{figDiscs}
\end{figure}
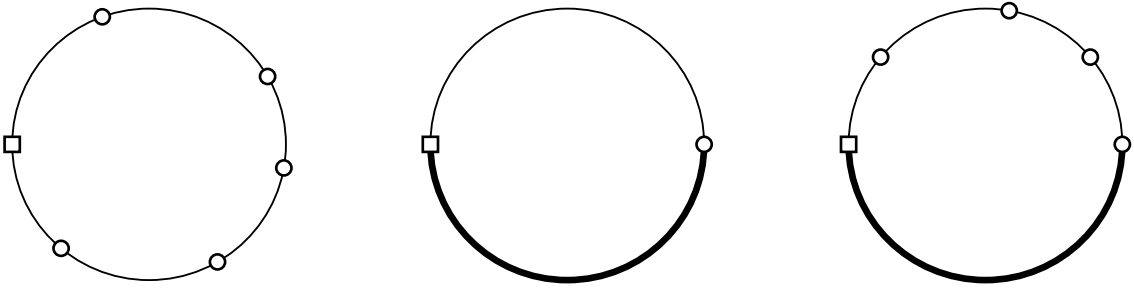
The unlabelled inputs are implicitly in anticlockwise order when read from right to left. We depict the squifferential on a matrix factorisation $\LMF(K) = \CF_S^*(\bL, K)$ and the components $\LMF^k$ similarly, with a thick boundary segment indicating the position of the $\bL$ in these operations. See the centre and right-hand part of Figure~\ref{figDiscs}. We refer to the Floer input and output at the end of the thick segment of the boundary as the \emph{distinguished} input and output.

The $A_\infty$-functor equations for $\LMF$ can be verified by considering equivalent moduli spaces to those defining $\LMF$ but in virtual dimension $1$, rather than $0$. These moduli spaces can be compactified in the usual way, and the fact that the boundary comprises zero (signed) points gives us the required relations. To see this, consider the possible codimension-$1$ degenerations of the right-hand disc in Figure~\ref{figDiscs}. Note that we are talking about degenerations of the \emph{map}; these may involve degeneration of the domain, but may instead only involve strip breaking at the inputs or output. Four things can happen: some of the non-distinguished inputs can bubble off to give a $\mu_\Fuk$ operation; the squifferential can bubble off at either the distinguished input or output, corresponding to the differential on $\mf$; or the disc can break along the thick segment into two discs of the same form, corresponding to the composition on $\mf$.
\begin{figure}[t]
\centering
\begin{tikzpicture}[scale=0.9]
\def\r{1.05cm}
\def\s{0.75cm}

\begin{scope}[xshift=-8cm, yshift=-0cm]
\draw[line width=\circwidth] (0, 0) circle[radius=\r];
\draw[line width=\LLwidth] (-180: \r) arc[start angle=-180, end angle=0, radius=\r];
\draw (-180: \r) node[flout]{};
\draw (0:\r) node[flin]{};
\draw (110: \r) node[flin]{};
\begin{scope}[shift={(50:\r+\s)}]
\draw[line width=\circwidth] (0, 0) circle[radius=\s];
\draw (-20: \s) node[flin]{};
\draw (100: \s) node[flin]{};
\end{scope}
\draw (50: \r) node[flout, rotate=50]{};
\draw (50: \r) node[flin]{};
\end{scope}

\begin{scope}[xshift=-4.6cm]
\draw[line width=\circwidth] (0, 0) circle[radius=\r];
\draw[line width=\circwidth] (0:\r+\s) circle[radius=\s];
\draw[line width=\LLwidth] (-180: \r) arc[start angle=-180, end angle=0, radius=\r];
\begin{scope}[shift={(0:\r+\s)}]
\draw[line width=\LLwidth] (\s, 0) arc[start angle=0, end angle=-180, radius=\s];
\end{scope}
\draw (-180: \r) node[flout]{};
\draw (0:\r) node[flout]{};
\draw (0:\r) node[flin]{};
\draw (40: \r) node[flin]{};
\draw (100: \r) node[flin]{};
\draw (150: \r) node[flin]{};
\draw (0:\r+2*\s) node[flin]{};
\end{scope}

\begin{scope}[xshift = 1.4cm]
\draw[line width=\circwidth] (0, 0) circle[radius=\r];
\draw[line width=\circwidth] (-\r-\s, 0) circle[radius=\s];
\draw[line width=\LLwidth] (-180: \r) arc[start angle=-180, end angle=0, radius=\r];
\draw[line width=\LLwidth] (-180: \r+2*\s) arc[start angle=-180, end angle=0, radius=\s];
\draw (-180: \r) node[flout]{};
\draw (-180: \r) node[flin]{};
\draw (0:\r) node[flin]{};
\draw (30: \r) node[flin]{};
\draw (60: \r) node[flin]{};
\draw (120: \r) node[flin]{};
\draw (-\r-2*\s, 0) node[flout]{};
\end{scope}

\begin{scope}[xshift=4.4cm]
\draw[line width=\circwidth] (0, 0) circle[radius=\r];
\draw[line width=\LLwidth] (-180: \r) arc[start angle=-180, end angle=0, radius=\r];
\begin{scope}[shift={(0: 2*\r)}, rotate=0]
\draw[line width=\circwidth] (0, 0) circle[radius=\r];
\draw[line width=\LLwidth] (-180: \r) arc[start angle=-180, end angle=0, radius=\r];
\draw (0: \r) node[flin]{};
\draw (70: \r) node[flin]{};
\end{scope}
\draw (-180: \r) node[flout]{};
\draw (0:\r) node[flout]{};
\draw (0:\r) node[flin]{};
\draw (60: \r) node[flin]{};
\draw (130: \r) node[flin]{};
\end{scope}

\end{tikzpicture}
\caption{Codimension-$1$ degenerations of a $\LMF^3$ disc.}\label{figLMFdegenerations}
\end{figure}
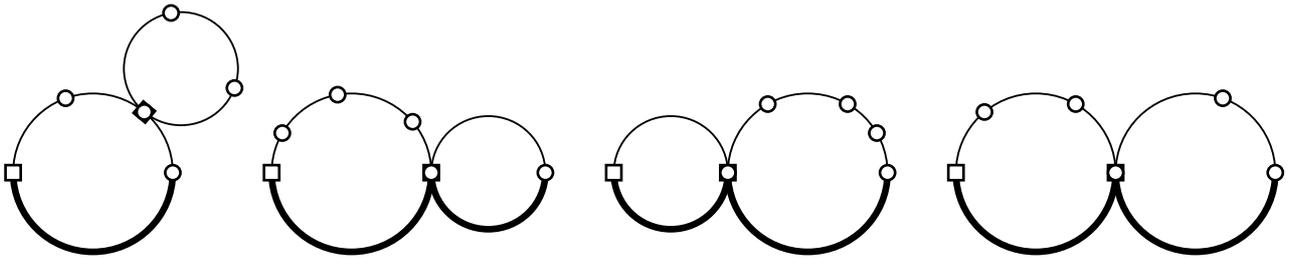
These four possibilities are shown in Figure~\ref{figLMFdegenerations}, from left to right. An open circle on top of an open square denotes a Floer generator which is an output of one disc and an input of another. Note that an~unstable disc bubble cannot form, for if it did then it would carry Maslov index at least $2$ (by monotonicity), so deleting the bubble would produce an element of a moduli space similar to the original configuration but of negative virtual dimension. This moduli space must be empty by regularity.

In a similar way, the operations $\mu_\bL$ on $\CF_S^*(\bL, \bL)^\op$ can be depicted as discs whose entire boundary is thick, as shown in Figure~\ref{figcCF}.
\begin{figure}[t]
\centering
\begin{tikzpicture}
\def\r{1.8cm}
\def\R{2.15cm}
\def\astart{-50}
\def\aend{60}

\begin{scope}[xshift=-5cm]
\draw[line width=\LLwidth] (0, 0) circle[radius=\r];

\draw (-180: \r) node[flout]{};

\draw (\astart:\r) node[flin]{};
\draw (\aend:\r) node[flin]{};

\draw (-120:\r) node[flin]{};
\draw (140:\r) node[flin]{};

\foreach \i in {1, 2, 3} \draw ({\astart+\i*(\aend-\astart)/4}: \R) node{$\cdot$};
\end{scope}

\end{tikzpicture}
\caption{The operations $\mu_\bL$ on $\CF_S^*(\bL, \bL)^\op$.}\label{figcCF}
\end{figure}
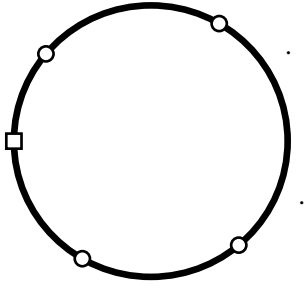
Inputs are now implicitly in \emph{clockwise} order when read from right to left. More generally, our convention is that if the boundary has both thick and thin segments, then inputs on the \emph{interior} of the thick segment are clockwise and the rest are anticlockwise.

These diagrams do not in themselves encode the signs attached to the discs we are counting. We specify these separately by relating them to the standard signs attached to discs defining $\mu_S$ operations, or to other well-established operations. For example, we could say that $\LMF^3(a_3, a_2, a_1)(a_0)$ is defined by the right hand picture in Figure~\ref{figDiscs}, counted with the same signs as for $\mu_S^4(a_3, \dots, a_0)$ but with a sign twist of $(-1)^{\lvert a_0 \rvert}$. Or we could say that the $\mu_\bL^l$ operation on $\CF_S^*(\bL, \bL)^\op$ is defined by Figure~\ref{figcCF} but with a sign twist of $(-1)^{\triangle_l+l-1}$.

\subsection{Hochschild cohomology}

Next we recall various versions of the Hochschild cochain complex. If $\calC$ is a $\ZZ$- (or $\ZZ/2$-)graded $A_\infty$-category over $\kk$, then the $\ZZ$- (respectively $\ZZ/2$-)graded Hochschild cochain complex $\rCC^*(\calC)$ is given by
\[
\rCC^t(\calC) = \prod_{k \geq 0} \prod_{\substack{\text{objects} \\ O_0, \dots , O_k}} \hom^t_{\kk}(\hom^*_\calC(O_{k-1}, O_k)[1] \otimes \dots \otimes \hom^*_\calC(O_0, O_1)[1], \hom^*_\calC(O_0, O_k)).
\]
The differential on this complex, with sign shifted as in \eqref{eqdgAinftyDiff}, satisfies
\begin{gather}
\mu^1_{\rCC} (\phi)^k(a_k, \dots, a_1) \nonumber\\
\qquad{}=\sum_{i,j} (-1)^{(\lvert \phi \rvert-1)\maltese_i}\mu_\calC^{k-j+1}\bigl(a_k, \dots, a_{i+j+1}, \phi^j(a_{i+j}, \dots, a_{i+1}), a_i, \dots, a_1\bigr)\nonumber\\
\quad\qquad{}+ \sum_{i,j} (-1)^{\maltese_i+\lvert \phi \rvert}\phi^{k-j+1}\bigl(a_k, \dots, a_{i+j+1}, \mu_\calC^j(a_{i+j}, \dots, a_{i+1}), a_i, \dots, a_1\bigr)\label{eqCCdiff1}
\end{gather}
for all $\phi \in \rCC^*(\calC)$, where $\maltese_i = \lvert a_1 \rvert + \dots + \lvert a_i \rvert - i$. The complex carries $A_\infty$-operations defined~by
\begin{gather*}
\mu_{\rCC}^2(\phi_2, \phi_1)^k(a_k, \dots, a_1)\\
\qquad{} = \sum (-1)^{(\lvert \phi_2 \rvert -1) \maltese_l + (\lvert \phi_1 \rvert-1) \maltese_i}\mu_\calC^{k-j-m+1}\bigl(a_k, \dots, a_{l+m+1},\\
\qquad\ {}\hphantom{=\sum } \phi_2^m(a_{l+m}, \dots, a_{l+1}), a_l, \dots, a_{i+j+1}, \phi_1^j(a_{i+j}, \dots, a_{i+1}), a_i, \dots, a_1\bigr)
\end{gather*}
and similarly for \smash{$\mu_{\rCC}^{\geq 3}$}. If $\calC$ is strictly unital, then $\rCC^*(\calC)$ is quasi-isomorphic to the \emph{reduced} Hochschild complex \smash{$\overline{\rCC}^*(\calC)$} given by those $\phi \in \rCC^*$ which vanish whenever one or more input is the strict unit.

We will mainly be interested in the case where $\calC = \Fuk(X)_\lambda$. We will also be interested in the Hochschild cohomology of $\Fuk(X)_\lambda$ with coefficients in $\mf$, viewed as a $\Fuk(X)_\lambda$-bimodule via~$\LMF$. This mixed Hochschild complex is given by
\begin{gather*}
\begin{gathered}
\rCC^t(\Fuk(X)_\lambda, \mf) = \prod_{k \geq 0} \prod_{\substack{\text{Lagrangians} \\ K_0, \dots , K_k}} \hom^t_{\kk}\bigl(\CF^*(K_{k-1}, K_k)[1] \otimes \dots \otimes \CF^*(K_0, K_1)[1],\\
\hspace{47mm}
\hom^*_\mf(\LMF(K_0), \LMF(K_k))\bigr).
\end{gathered}
\end{gather*}
The differential, with sign modified via \eqref{eqdgAinftyDiff}, is
\begin{gather*}
\mu^1_{\rCC}(\phi)^k(a_k, \dots, a_1)
\\ \qquad{}=\sum_{\substack{i,j,p,r \\ q_1, \dots, q_p \\ s_1, \dots, s_r}} (-1)^{(\lvert \phi \rvert-1)\maltese_i}\mu_\mf^{p+r+1}\bigl(\LMF^{s_r}(a_k, \dots, a_{k-s_r+1}), \dots, \LMF^{s_1}(a_{i+j+s_1}, \dots, a_{i+j+1}),
\\
\hspace{25mm}\phi^j(a_{i+j}, \dots, a_{i+1}), \LMF^{q_p}(a_i, \dots, a_{i-q_p+1}), \dots, \LMF^{q_1}(a_{q_1}, \dots, a_1)\bigr)
\\ \qquad\quad{} + \sum_{i,j} (-1)^{\lvert \phi \rvert+\maltese_i}\phi^{k-j+1}\bigl(a_k, \dots, a_{i+j+1}, \mu_\Fuk^j(a_{i+j}, \dots, a_{i+1}), a_i, \dots, a_1\bigr).
\end{gather*}
To declutter the notation, we will write expressions like this as
\begin{align*}
\mu^1_{\rCC} (\phi) ={}& \sum (-1)^{(\lvert \phi \rvert-1)\maltese_i}\mu_\mf(\LMF(\dots), \dots, \LMF(\dots), \phi(\dots), \LMF\underbrace{(\dots), \dots, \LMF(\dots)}_i) \\
&{+}\, \sum (-1)^{\lvert \phi \rvert+\maltese_i} \phi(\dots, \mu_\Fuk(\dots), \stackrel{i}{\dots}).
\end{align*}
Note that the first $i$ refers to the total number of inputs to the $\LMF$ operations to the right of~$\phi$, not to the number of $\LMF$ operations themselves. Since $\mf$ is a dg-category, i.e., \smash{$\mu_\mf^{\geq 3} = 0$}, this simplifies to
\begin{align}
\mu^1_{\rCC} (\phi) ={}& \mu^1_\mf(\phi(\dots)) + \sum \mu^2_\mf(\LMF(\dots), \phi(\dots))\nonumber
\\ & + \sum (-1)^{(\lvert \phi \rvert - 1)\maltese_i}\mu^2_\mf(\phi(\dots), \LMF(\stackrel{i}{\dots})) \nonumber\\
& + \sum (-1)^{\lvert \phi \rvert+\maltese_i} \phi(\dots, \mu_\Fuk(\dots), \stackrel{i}{\dots}).\label{eqCCdiff}
\end{align}
Similarly, the product on $\rCC^*(\Fuk(X)_\lambda, \mf)$ (with sign modified by \eqref{eqdgAinftyProd}) simplifies to
\begin{equation}
\label{eqCCprod}
\mu_{\rCC}^2(\psi, \phi) = \sum (-1)^{(\lvert \psi \rvert - 1)\maltese_i}\mu^2_\mf(\psi(\dots),\phi(\stackrel{i}{\dots})),
\end{equation}
and we have \smash{$\mu^{\geq 3}_{\rCC} = 0$}. See \cite[Section~(1d)]{SeidelBook} for further discussion of this mixed Hochschild complex, viewed as endomorphisms of $\phi$ in the $A_\infty$-category of $A_\infty$-functors $\Fuk(X)_\lambda \to \mf$.

\begin{Remark}
For the ungraded Fukaya category, in characteristic $2$, we can define everything by the same formulae but ignoring the signs. Hochschild cohomology is itself ungraded in this~case.
\end{Remark}

\subsection{The closed-open map}
\label{sscCOsetup}

The last thing we need to recall for now is the closed-open string map, so fix a class $\alpha$ in $\QH^*(X)$ and a pseudocycle \cite{Zinger} $Z \to X$ representing its Poincar\'e dual. We will use $Z$ to define $\CO_\lambda(\alpha)$ and $\CO^0_\bL(\alpha)$.

The class \smash{$\CO_\lambda(\alpha) \in \HH^{\lvert \alpha \rvert}(\Fuk(X)_\lambda)$} is represented by the cocycle $\sigma$ defined as follows. For each $k \geq 0$, the component
\[
\sigma^k\colon \ \CF^*(K_{k-1}, K_k)[1] \otimes \dots \otimes \CF^*(K_0, K_1)[1] \to \CF^*(K_0, K_k),
\]
counts rigid pseudoholomorphic discs with $k$ inputs (anticlockwise) and a single output on the boundary, and with an interior marked point constrained to lie on~$Z$. We depict this as shown on the left in Figure~\ref{figCO}, with the solid dot indicating the marked point on $Z$.
\begin{figure}[t]
\centering
\begin{tikzpicture}
\def\r{1.8cm}
\def\R{2.15cm}
\def\astart{-70}
\def\aend{40}

\begin{scope}
\draw[line width=\circwidth] (0, 0) circle[radius=\r];

\draw (-180: \r) node[flout]{};

\draw (0, 0) node[blob]{};

\draw (\astart:\r) node[flin]{};
\draw (\aend:\r) node[flin]{};

\draw (-140:\r) node[flin]{};
\draw (110:\r) node[flin]{};

\foreach \i in {1, 2, 3} \draw ({\astart+\i*(\aend-\astart)/4}: \R) node{$\cdot$};
\end{scope}

\begin{scope}[xshift=6cm]
\draw[line width=\LLwidth] (0, 0) circle[radius=\r];

\draw (-180: \r) node[flout]{};

\draw (0, 0) node[blob]{};
\end{scope}
\end{tikzpicture}
\caption{The cocycles $\sigma$ (left) and $\sigma_\bL$ (right), representing $\CO_\lambda(\alpha)$ and $\CO^0_\bL(\alpha)$, respectively.}\label{figCO}
\end{figure}
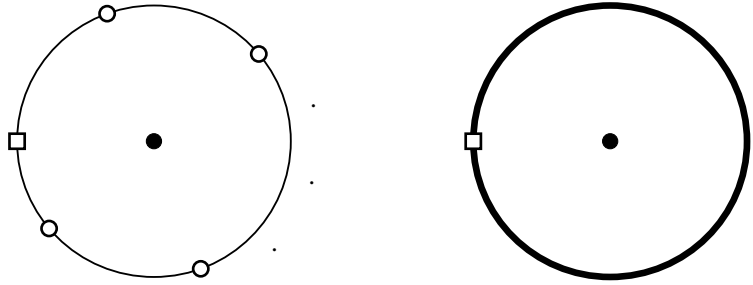
This cocycle depends on a choice of regular perturbation data, which should be compatible with degeneration of the domain in the following sense. Whenever the domain degenerates, the induced perturbation data on the disc component carrying the interior marked point should agree with those defining $\sigma$, whilst the induced perturbation data on the other components should agree with those defining the $A_\infty$-operations $\mu_\Fuk$. Suitable perturbation data can be constructed by induction on $k$ in the usual way, as in \cite[Section~9]{SeidelBook}. The Hochschild cohomology class of $\sigma$ is independent of this choice.

The class \smash{$\CO^0_{\bL}(\alpha) \in \HF_S^{\lvert \alpha \rvert}(\bL, \bL)^{(\op)}$} is represented by the cocycle $\sigma_\bL$ that counts rigid pseudoholomorphic discs with a single interior marked point constrained to $Z$, and a single output on the boundary, as depicted on the right in Figure~\ref{figCO}. We can use arbitrary regular perturbation data that agree with the Floer datum for $\bL$ at the output. The signs are the same as those for~$\sigma^0$ (that is, the $k=0$ component of $\sigma$), which we will not need explicit expressions for.\looseness=1

The fact that $\CO_\lambda$ is an algebra homomorphism is proved in \cite[Proposition 2.1]{SheridanFano}. Applying the $k=0$ part of this proof to $\Fuk_S(X)_{W_L}$ shows that \smash{$\CO^0_\bL$} is an algebra homomorphism to~\smash{$\HF_S^*(\bL, \bL)$}. Since $\QH^*(X)$ is graded-commutative, \smash{$\CO^0_\bL$} is also an algebra homomorphism to~\smash{$\HF_S^*(\bL, \bL)^\op$}.

The homomorphism $\CO^0_\bL$ is unital because for $\alpha = 1_X$, represented by $Z = X$, the cocycle $\sigma_\bL$ defines the cohomological unit in $\CF_S^*(\bL, \bL)^{(\op)}$ \cite[Section~2.4]{SheridanFano}. However, unitality of $\CO_\lambda$ is more subtle. In \cite{GanatraThesis}, Ganatra introduces the \emph{$2$-pointed} Hochschild complex $\tCC^*(\Fuk(X)_\lambda)$, which describes endomorphisms of the diagonal bimodule, and a corresponding $2$-pointed closed-open map $\tCO_\lambda$. We are slightly modifying his notation here to align it with ours. He constructs a~chain~map
\[
\Psi\colon \ \rCC^*(\Fuk(X)_\lambda) \to \tCC^*(\Fuk(X)_\lambda),
\]
which intertwines $\CO_\lambda$ with $\tCO_\lambda$ \cite[Proposition 5.6]{GanatraThesis}, and which is a quasi-isomorphism when working over a field \cite[Proposition 2.5]{GanatraThesis}. Unitality of $\CO_\lambda$ is then proved over $\CC$ in \cite{RitterSmith, SheridanFano} by showing that $\tCO_\lambda$ is unital (over any ring) and invoking Ganatra's results to transfer this (over a field) to $\CO_\lambda$. Unitality of $\CO_\lambda$ over an arbitrary coefficient ring seems to be unknown.\looseness=1

\begin{Remark}
The restriction to field coefficients here is analogous to the projectivity hypothesis in \cite[Corollary 9.1.5]{Weibel}, which interprets Hochschild cohomology in terms of bimodule $\Ext$~groups.
\end{Remark}

\section{Proof of Theorems~\ref{TheoremA} and \ref{TheoremB}}
\label{secThmsAB}

With this setup in place, we are ready to prove our first two main theorems. Recall that Theorem~\ref{TheoremA} asserts the existence of a cohomologically unital $S$-linear $A_\infty$-algebra homomorphism
\[
\Theta\colon \ \CF_S^*(\bL, \bL)^\op \to \rCC^*(\Fuk(X)_\lambda, \mf),
\]
extending the module action of $\CF_S^*(\bL, \bL)^\op$ on $\LMF$ in a way we shall make precise in Section~\ref{sscModuleAction}. Theorem~\ref{TheoremB} then asserts that $\Theta$ fits into a commutative diagram
\[
\begin{tikzcd}
\QH^*(X) \arrow{r}{\CO_\lambda} \arrow{d}{\CO^0_\bL} & \HH^*(\Fuk(X)_\lambda) \arrow{d}{\rH({\LMF}_*)}
\\ \HF_S^*(\bL, \bL)^\op \arrow{r}{\rH(\Theta)} & \HH^*(\Fuk(X)_\lambda, \mf).
\end{tikzcd}
\]

\subsection[Constructing Theta]{Constructing $\boldsymbol{\Theta}$}

First we define the map $\Theta$. To do this, for each tuple $c_1, \dots, c_l$ in $\CF_S^*(\bL, \bL)^\op$ with $l \geq 1$ we need to specify the element $\Theta^l(c_l, \dots, c_1)$ in $\rCC^*(\Fuk(X)_\lambda, \mf)$. This amounts to defining, for each $k \geq 0$, each $(k+1)$-tuple of objects $K_1, \dots, K_k$ in $\Fuk(X)_\lambda$, each $k$-tuple of morphisms
\begin{equation}
\label{eqai}
(a_k, \dots, a_1) \in \CF^*(K_{k-1}, K_k) \times \dots \times \CF^*(K_0, K_1),
\end{equation}
and each $x$ in $\LMF(K_0) = \CF_S^*(\bL, K_0)$, an element
\begin{equation}
\label{eqThetaDef}
\Theta^l(c_l, \dots, c_1)(a_k, \dots, a_1)(x) \in \LMF(K_k) = \CF_S^*(\bL, K_k).
\end{equation}

\begin{Definition}\label{defTheta}
The element \eqref{eqThetaDef} is defined by counting analogous curves to those defining~\smash{$\mu^1_\mf$} and $\LMF$ but with additional inputs on the thick segment of the boundary, and twisted by a~sign. More precisely, we count rigid discs of the form shown in Figure~\ref{figTheta},
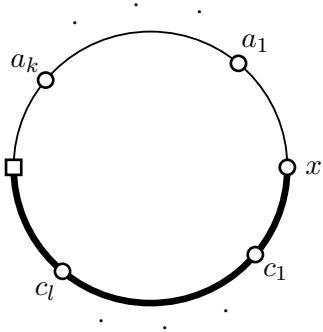
\begin{figure}[b]
\centering
\begin{tikzpicture}
\def\r{1.8cm}
\def\R{2.15cm}
\def\astart{50}
\def\aend{140}
\def\cstart{-40}
\def\cend{-130}

\draw[line width=\circwidth] (0, 0) circle[radius=\r];
\draw[line width=\LLwidth] (-180: \r) arc[start angle=-180, end angle=0, radius=\r];

\draw (-180: \r) node[flout]{};

\draw (\astart:\r) node[flin]{};
\draw (\astart:\R) node{$a_1$};
\draw (\aend:\r) node[flin]{};
\draw (\aend:\R) node{$a_k$};

\foreach \i in {1, 2, 3} \draw ({\astart+\i*(\aend-\astart)/4}: \R) node{$\cdot$};

\draw (\cstart:\r) node[flin]{};
\draw (\cstart:\R) node{$c_1$};
\draw (\cend:\r) node[flin]{};
\draw (\cend:\R) node{$c_l$};

\foreach \i in {1, 2, 3} \draw ({\cstart+\i*(\cend-\cstart)/4}: \R) node{$\cdot$};

\draw (0:\r) node[flin]{};
\draw (0:\R) node{$x$};
\end{tikzpicture}
\caption{The discs defining $\Theta(c_l, \dots, c_1)(a_k, \dots, a_1)(x)$.}\label{figTheta}
\end{figure}
with inputs $a_1, \dots, a_k$ anticlockwise on the ordinary segment of the boundary, and $c_1, \dots, c_l$ clockwise on the thick segment. Each disc contributes with the same sign as it would towards
\[
\mu_S^{k+l+1}(a_k, \dots, a_1, x, c_1, \dots, c_l)
\]
but with a twist of
\begin{equation}
\label{eqThetaTwist}
(-1)^{(\square_l - 1)(\lvert x \rvert -1) + \triangle_l + l - 1}.
\end{equation}
Here \smash{$\triangle_l = \sum_{i<j} (\lvert c_i \rvert - 1)(\lvert c_j \rvert - 1)$} as before, and \smash{$\square_l = \sum_j (\lvert c_j \rvert - 1)$}; we call this $\square$ rather than~$\maltese$ to distinguish it from the corresponding quantity for the $a_i$.

As usual, we have to choose regular perturbation data, and we do this by induction on the number of inputs, compatibly with degeneration of the domain in the following sense. After a~degeneration, each of the components looks like the domain of a curve defining $\Theta$ (but with strictly fewer $a$- or $c$-inputs), $\LMF$, $\mu_\bL$, or $\mu_\Fuk$, for which we have already defined perturbation data. We require that the perturbation data induced by the degeneration agree with those already defined on each component.
\end{Definition}

\begin{Remark}
Discs of a similar form to Figure~\ref{figTheta}, with a distinguished input and output separated by arbitrarily many inputs on either side, appear in the definition of Ganatra's $2$-pointed closed-open map \cite[Section~5.6]{GanatraThesis}.
\end{Remark}

\subsection[Theta is a homomorphism]{$\boldsymbol{\Theta}$ is a homomorphism}

The heart of Theorem~\ref{TheoremA} is the following result.

\begin{Proposition}\label{propThetaHom}
$\Theta$ is an $A_\infty$-algebra homomorphism.
\end{Proposition}

Before proving this, we introduce a useful shorthand: we add a bar to operations involving sign twists to denote the same operations without the sign twists. For example, \smash{$\unsign{\Theta}$}, \smash{$\unsign{\LMF}$}, and~\smash{$\unsign{\mu}_\bL$} count the same discs as $\Theta$, $\LMF$, and $\mu_\bL$ but with the signs that would be carried by the plain~$\mu_S$ operations with the corresponding inputs.

\begin{proof}[Proof of Proposition~\ref{propThetaHom}]
Consider the moduli spaces of discs analogous to those defining~$\Theta$ but in virtual dimension $1$. The boundaries of their compactifications comprise degenerate configurations illustrated in Figure~\ref{figThetaHom}, and give rise to relations between $\Theta$, $\LMF$, $\mu_\bL$, $\mu_\Fuk$, and~$\mu^1_S$, just as we obtained the $A_\infty$-functor equations for $\LMF$ in Section~\ref{sscDiagrams}.
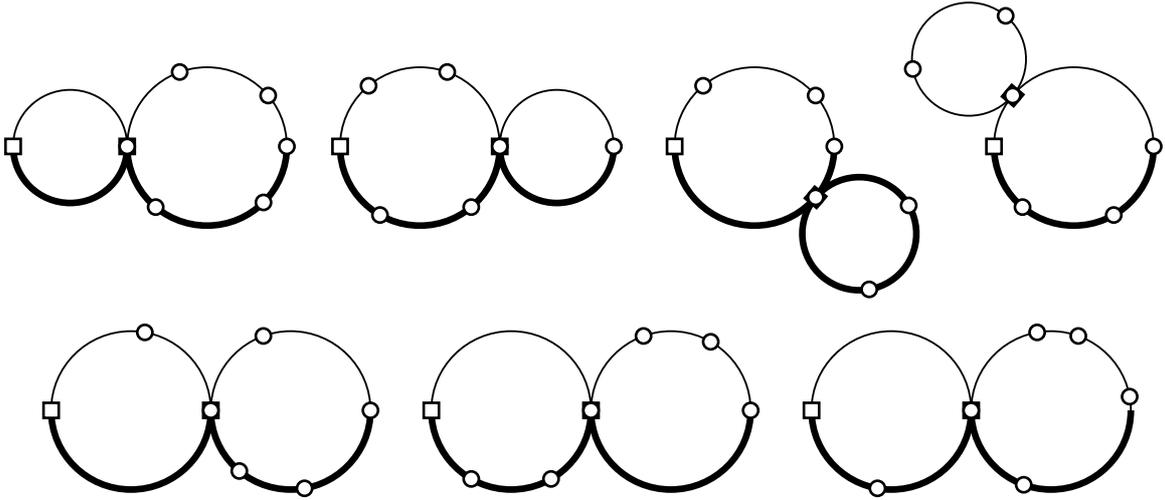
\begin{figure}[t]
\centering
\begin{tikzpicture}
\def\r{1.05cm}
\def\s{0.75cm}

\begin{scope}[xshift = -4cm]
\draw[line width=\circwidth] (0, 0) circle[radius=\r];
\draw[line width=\circwidth] (-\r-\s, 0) circle[radius=\s];
\draw[line width=\LLwidth] (-180: \r) arc[start angle=-180, end angle=0, radius=\r];
\draw[line width=\LLwidth] (-180: \r+2*\s) arc[start angle=-180, end angle=0, radius=\s];
\draw (-180: \r) node[flout]{};
\draw (-180: \r) node[flin]{};
\draw (-130:\r) node[flin]{};
\draw (-45: \r) node[flin]{};
\draw (0: \r) node[flin]{};
\draw (40: \r) node[flin]{};
\draw (110: \r) node[flin]{};
\draw (-\r-2*\s, 0) node[flout]{};
\end{scope}

\begin{scope}[xshift=-1.2cm]
\draw[line width=\circwidth] (0, 0) circle[radius=\r];
\draw[line width=\circwidth] (0:\r+\s) circle[radius=\s];
\draw[line width=\LLwidth] (-180: \r) arc[start angle=-180, end angle=0, radius=\r];
\begin{scope}[shift={(0:\r+\s)}]
\draw[line width=\LLwidth] (0: \s) arc[start angle=0, end angle=-180, radius=\s];
\end{scope}
\draw (-180: \r) node[flout]{};
\draw (0:\r) node[flout]{};
\draw (0:\r) node[flin]{};
\draw (-120: \r) node[flin]{};
\draw (-50: \r) node[flin]{};
\draw (70: \r) node[flin]{};
\draw (130: \r) node[flin]{};
\draw (0:\r+2*\s) node[flin]{};
\end{scope}

\begin{scope}[xshift=3.2cm]
\draw[line width=\circwidth] (0, 0) circle[radius=\r];
\draw[line width=\LLwidth] (-180: \r) arc[start angle=-180, end angle=0, radius=\r];
\begin{scope}[shift={(-40: \r+\s)}, rotate=-40]
\draw[line width=\LLwidth] (0, 0) circle[radius=\s];
\draw (-40: \s) node[flin]{};
\draw (70: \s) node[flin]{};
\end{scope}
\draw (-180: \r) node[flout]{};
\draw (-40:\r) node[flout, rotate=-40]{};
\draw (-40:\r) node[flin]{};
\draw (0: \r) node[flin]{};
\draw (40: \r) node[flin]{};
\draw (130: \r) node[flin]{};
\end{scope}

\begin{scope}[xshift=7.4cm]
\draw[line width=\circwidth] (0, 0) circle[radius=\r];
\draw[line width=\LLwidth] (-180: \r) arc[start angle=-180, end angle=0, radius=\r];
\draw (-180: \r) node[flout]{};
\draw (0:\r) node[flin]{};
\draw (-60: \r) node[flin]{};
\draw (-130: \r) node[flin]{};
\begin{scope}[shift={(140:\r+\s)}]
\draw[line width=\circwidth] (0, 0) circle[radius=\s];
\draw (190: \s) node[flin]{};
\draw (50: \s) node[flin]{};
\end{scope}
\draw (140: \r) node[flout, rotate=140]{};
\draw (140: \r) node[flin]{};
\end{scope}

\begin{scope}[xshift = -5cm, yshift = -3.5cm]
\draw[line width=\circwidth] (0, 0) circle[radius=\r];
\draw[line width=\circwidth] (0:2*\r) circle[radius=\r];
\draw[line width=\LLwidth] (-180: \r) arc[start angle=-180, end angle=0, radius=\r];
\begin{scope}[shift={(0:2*\r)}]
\draw[line width=\LLwidth] (0: \r) arc[start angle=0, end angle=-180, radius=\r];
\draw (0: \r) node[flin]{};
\draw (110: \r) node[flin]{};
\draw (-80:\r) node[flin]{};
\draw (-130: \r) node[flin]{};
\end{scope}
\draw (-180: \r) node[flout]{};
\draw (0:\r) node[flout]{};
\draw (0:\r) node[flin]{};
\draw (80: \r) node[flin]{};
\end{scope}

\begin{scope}[xshift = 0cm, yshift = -3.5cm]
\draw[line width=\circwidth] (0, 0) circle[radius=\r];
\draw[line width=\circwidth] (0:2*\r) circle[radius=\r];
\draw[line width=\LLwidth] (-180: \r) arc[start angle=-180, end angle=0, radius=\r];
\begin{scope}[shift={(0:2*\r)}]
\draw[line width=\LLwidth] (0: \r) arc[start angle=0, end angle=-180, radius=\r];
\draw (0:\r) node[flin]{};
\draw (110: \r) node[flin]{};
\draw (60: \r) node[flin]{};
\end{scope}
\draw (-180: \r) node[flout]{};
\draw (0:\r) node[flout]{};
\draw (0:\r) node[flin]{};
\draw (-60: \r) node[flin]{};
\draw (-120: \r) node[flin]{};
\end{scope}

\begin{scope}[xshift = 5cm, yshift = -3.5cm]
\draw[line width=\circwidth] (0, 0) circle[radius=\r];
\draw[line width=\circwidth] (0:2*\r) circle[radius=\r];
\draw[line width=\LLwidth] (-180: \r) arc[start angle=-180, end angle=0, radius=\r];
\begin{scope}[shift={(0:2*\r)}]
\draw[line width=\LLwidth] (0: \r) arc[start angle=0, end angle=-180, radius=\r];
\draw (100: \r) node[flin]{};
\draw (70: \r) node[flin]{};
\draw (10:\r) node[flin]{};
\draw (-110: \r) node[flin]{};
\end{scope}
\draw (-180: \r) node[flout]{};
\draw (0:\r) node[flout]{};
\draw (0:\r) node[flin]{};
\draw (-100: \r) node[flin]{};
\end{scope}

\end{tikzpicture}
\caption{Degenerations corresponding to terms in \eqref{eqThetaRelationUnsigned}.}\label{figThetaHom}
\end{figure}
(The $\mu_S^1$ terms represent squifferentials on matrix factorisations, modulo sign twists, and did not appear in the discussion of perturbation data in Definition~\ref{defTheta} since they do not correspond to degeneration of the domain.) The resulting relations are exactly analogous to those proving the $A_\infty$-relations for $\mu_S$, except for the twisting of signs in $\Theta$, $\LMF$, and~$\mu_\bL$. Temporarily ignoring the sign twists, we obtain
\begin{gather}
\mu^1_S\bigl(\unsign{\Theta}(\dots)(\dots)(x)\bigr)+ (-1)^{\square_l} \unsign{\Theta}(\dots)(\dots)\bigl(\mu^1_S(x)\bigr) \nonumber
\\ \qquad{} + \sum (-1)^{\square_l - \square_{l-j}}\unsign{\Theta}\bigm(\stackrel{j}{\dots}, \unsign{\mu}_\bL(\dots), \dots\bigr)(\dots)(x)\nonumber
\\ \qquad{} + \sum (-1)^{\square_l+\lvert x \rvert - 1 + \maltese_i} \unsign{\Theta}(\dots)(\dots, \mu_\Fuk(\dots), \stackrel{i}{\dots})(x) + \sum \unsign{\LMF}(\dots)\bigl(\unsign{\Theta}(\dots)(\dots)(x)\bigr)\nonumber
\\ \qquad{} + \sum (-1)^{\square_l} \unsign{\Theta}(\dots)(\dots) \bigl(\unsign{\LMF}(\dots)(x)\bigr)\nonumber
\\ \qquad{} + \sum (-1)^{\square_l - \square_j} \unsign{\Theta}(\dots)(\dots)\bigl(\unsign{\Theta}\bigl(\stackrel{j}{\dots}\bigr)(\dots)(x)\bigr) = 0.\label{eqThetaRelationUnsigned}
\end{gather}
Some of the sums may be empty if the number of inputs is small. For example, if $k=0$ (no $a$-inputs), then the fourth sum is empty, whilst if $l = 1$ (only one $c$-input), then the final sum is empty. The clockwise ordering of the $c$-inputs is needed to make the composition of $\unsign{\Theta}$ operations in the final sum come out the right way round.

Next we reinstate the sign twists for $\Theta$, $\LMF$, and $\mu_\bL$, namely \eqref{eqThetaTwist}, \smash{$(-1)^{\lvert \text{distinguished input} \rvert}$}, and~\eqref{eqopdefinition}. After simplifying, and cancelling off an overall factor of \smash{$(-1)^{\square_l(\lvert x \rvert - 1) + \triangle_l + l}$}, \eqref{eqThetaRelationUnsigned}~becomes
\begin{gather}\nonumber
\mu^1_\mf(\Theta(\dots)(\dots)) + \sum (-1)^{\square_l+\maltese_i+1} \Theta(\dots)(\dots, \mu_\Fuk(\dots), \stackrel{i}{\dots})
\\ \qquad\quad{} + \sum \mu^2_\mf (\LMF(\dots), \Theta(\dots)(\dots)) + \sum (-1)^{\square_l\maltese_i} \mu^2_\mf(\Theta(\dots)(\dots), \LMF(\stackrel{i}{\dots}))\nonumber
\\ \qquad\quad{} + \sum (-1)^{(\square_l - \square_j)\maltese_i} \mu^2_\mf\bigl(\Theta(\dots)(\dots), \Theta\bigl(\stackrel{j}{\dots}\bigr)(\stackrel{i}{\dots})\bigr)\nonumber
\\ \qquad{} = \sum (-1)^{\square_j} \Theta\bigl(\dots, \mu_\bL(\dots), \stackrel{j}{\dots}\bigr)(\dots).\label{eqThetaRelationB}
\end{gather}
Note that we have eliminated $x$ and are now viewing this as an equality in
\[
\hom_\mf^*(\LMF(K_0), \LMF(K_k))
\]
rather than $\LMF(K_k)$. Using the expressions \eqref{eqCCdiff} and \eqref{eqCCprod} for \smash{$\mu^1_{\rCC}$} and \smash{$\mu^2_{\rCC}$}, we can rewrite \eqref{eqThetaRelationB} as
\[
\mu^1_{\rCC}(\Theta(\dots)) + \sum \mu^2_{\rCC}(\Theta(\dots), \Theta(\dots)) = \sum (-1)^{\square_j} \Theta\bigl(\dots, \mu_\bL(\dots), \stackrel{j}{\dots}\bigr).
\]
This an equality in $\rCC^*(\Fuk(X)_\lambda, \mf)$, and is exactly the $A_\infty$-homomorphism equation for $\Theta$ that we wanted to prove.
\end{proof}

\subsection{Compatibility with the module action}
\label{sscModuleAction}

Next we explain, and prove, the sense in which $\Theta$ extends the module action of $\CF_S^*(\bL, \bL)^\op$ on~$\LMF$. Note that for each object $K$ in $\Fuk(X)_\lambda$ the matrix factorisation ${\LMF(K) = \CF_S^*(\bL, K)}$ is naturally a right module for $\CF_S^*(\bL, \bL)$, and hence a left module for $\CF_S^*(\bL, \bL)^\op$. This module action is described by a chain map
\[
T_K\colon \ \CF_S^*(\bL, \bL)^\op \to \eend^*_\mf(\LMF(K)),
\]
which is given explicitly by
\begin{equation}
\label{eqTexpression}
T_K(c)(x) = (-1)^{\lvert c \rvert \lvert x \rvert + \lvert c \rvert} \mu_S^2(x, c)
\end{equation}
for $c \in \CF_S^*(\bL, \bL)^\op$ and $x \in \LMF(K)$. Here the \smash{$(-1)^{\lvert c \rvert \lvert x \rvert}$} is a Koszul sign whilst the \smash{$(-1)^{\lvert c \rvert}$} corresponds to the usual sign \eqref{eqdgAinftyProd} relating products to \smash{$\mu^2$} operations.

The relationship between this module action and $\Theta$ is then as follows.

\begin{Lemma}\label{lemThetaModuleAction}
For each object $K$ in $\Fuk(X)_\lambda$, the projection of $\Theta^1$ to length zero on $K$ is homotopic to~$T_K$.
\end{Lemma}

\begin{proof}
The projection of $\Theta^1$ to length zero corresponds to the $l=1$, $k=0$ case of Definition~\ref{defTheta}. Comparing this definition with \eqref{eqTexpression} proves the result, after noting that \eqref{eqThetaTwist} reduces to \smash{$(-1)^{\lvert c \rvert \lvert x \rvert + \lvert c \rvert}$}. We only get `homotopic to' since we did not specify that the perturbation data used to define $\Theta$ should be compatible with those used to define $T_K$.
\end{proof}

The remaining assertion of Theorem~\ref{TheoremA} is that $\Theta$ is cohomologically unital. We defer this to Corollary~\ref{corThetaUnital}, since we will deduce it from Theorem~\ref{TheoremB}.

\subsection{Independence of choices}
\label{sscIndependence}

Our constructions so far depend on choices of Floer and perturbation data for $\CF_S^*(\bL, \bL)^\op$, $\Fuk(X)_\lambda$, $\LMF$, and $\Theta$, as well as a choice of Lagrangians to allow as objects in $\Fuk(X)_\lambda$. We now explain the sense in which the homomorphism $\Theta$ is independent of these choices. We break this into two parts, depending on which choices are being changed, but these parts can of course be combined to deal with arbitrary changes of auxiliary data.

First, suppose we fix a choice of Floer and perturbation data for $\CF_S^*(\bL, \bL)^\op$. That is, we fix the domain of $\Theta$. Now suppose we have defined $\Fuk(X)_\lambda$, $\LMF$, and $\Theta$ using two different sets of auxiliary choices. We indicate these with subscripts ${}_1$ and ${}_2$. The invariance result is then that the maps $\Theta_1$ and $\Theta_2$ are related by a zigzag, in the following sense.

\begin{Lemma}
There exists a Fukaya category, denoted $\Fuk(X)_{\lambda,+}$, containing $\Fuk(X)_{\lambda,1}$, $\Fuk(X)_{\lambda,2}$
as full subcategories, such that the functors ${\LMF}_i$ extend to ${\LMF}_{+} \colon \Fuk(X)_{\lambda,+} \to \mf$, and the homomorphisms $\Theta_i$ extend to $\Theta_{+} \colon \CF_S^*(\bL, \bL)^\op \to \rCC^*(\Fuk(X)_{\lambda,+}, \mf)$. The following diagram then tautologically commutes:
\[
\begin{tikzcd}
\CF_S^*(\bL, \bL)^\op \arrow{r}{\Theta_1} \arrow{d}{\Theta_2} \arrow{dr}{\Theta_{+}} & \rCC^*(\Fuk(X)_{\lambda,1}, \mf)
\\ \rCC^*(\Fuk(X)_{\lambda,2}, \mf) & \rCC^*(\Fuk(X)_{\lambda,+}, \mf) \arrow[swap]{u}{\mathrm{restrict}} \arrow[swap]{l}{\mathrm{restrict}},
\end{tikzcd}
\]
and if $\Fuk(X)_{\lambda,1}$ and $\Fuk(X)_{\lambda,2}$ have essentially the same objects, in the sense that any Lagrangian in one is split-generated by the Lagrangians in the other, then the two restriction arrows are quasi-isomorphisms.
\end{Lemma}

\begin{proof}
Except for the last part, this is a standard double category argument, as in \cite[Section~(10a)]{SeidelBook}. Concretely, we construct a category whose set of objects is the disjoint union of the sets of objects in $\Fuk(X)_{\lambda,1}$ and $\Fuk(X)_{\lambda,2}$. We call objects originating in $\Fuk(X)_{\lambda,i}$ \emph{objects of type $i$}. When defining operations involving only objects of type $i$, we use auxiliary data from~$\Fuk(X)_{\lambda,i}$,~${\LMF}_i$, and~$\Theta_i$. For operations involving mixtures of objects of the two types, we make arbitrary choices of auxiliary data, compatible with those already chosen.

Now suppose that $\Fuk(X)_{\lambda,1}$ and $\Fuk(X)_{\lambda,2}$ have essentially the same objects, so that the inclusions $\Fuk(X)_{\lambda,i} \to \Fuk(X)_{\lambda,+}$ are quasi-equivalences. The fact that the restriction maps induce isomorphisms on Hochschild cohomology then follows from standard Morita invariance properties. Explicitly, one applies the Eilenberg--Moore comparison theorem (Theorem~\ref{thmEMcomparison}) to the length filtration; because we are working over a field this reduces the problem to the cohomology categories; and there one can give an explicit homotopy inverse \cite[Lemma~2.6]{SeidelBook}.
\end{proof}

Now suppose instead that we fix auxiliary data for $\Fuk(X)_\lambda$ and $\LMF$, but choose two sets of data for $\CF_S^*(\bL, \bL)^\op$, again denoted by subscripts ${}_1$ and ${}_2$. For each $i$, we can construct a~homomorphism $\Theta_i \colon \CF_S^*(\bL, \bL)^\op_i \to \rCC^*(\Fuk(X)_\lambda, \mf)$. As above, the invariance result is that these maps are related by a zigzag, but this time it has the following form.

\begin{Lemma}\label{lemThetaIndep2}
There exists a cohomologically unital $A_\infty$-algebra $\calA$, equipped with a cohomologically unital $A_\infty$-algebra homomorphism
\[
\Theta_\calA\colon \ \calA \to \rCC^*(\Fuk(X)_\lambda, \mf)
\]
and cohomologically unital $A_\infty$-quasi-isomorphisms
\[
I_i\colon \ \CF_S^*(\bL, \bL)^\op_i \to \calA,
\]
such that the following diagram commutes:
\[
\begin{tikzcd}
\calA \arrow{dr}{\Theta_\calA} & CF_S^*(\bL, \bL)^\op_1 \arrow{d}{\Theta_1} \arrow[swap]{l}{I_1}
\\ CF_S^*(\bL, \bL)^\op_2 \arrow[swap]{u}{I_2} \arrow{r}{\Theta_2} & \rCC^*(\Fuk(X)_\lambda, \mf).
\end{tikzcd}
\]
\end{Lemma}

\begin{proof}
We adapt the arguments of \cite[Section~2.4]{PerutzSheridanRelative}, based on the triangle algebra of \cite[Section~4.5]{Keller}, in conjunction with another double category argument. First, we view~\smash{$\CF_S^*(\bL, \bL)^\op_1$} and~\smash{$\CF_S^*(\bL, \bL)^\op_2$} as the endomorphism algebras of objects $\bL_1$ and $\bL_2$ in \smash{$\Fuk_S(X)_{W_L}^\op$}. The usual Floer continuation map construction gives a quasi-isomorphism $f \colon \bL_1 \to \bL_2$, and we define $\calA$ to be
\[
\begin{pmatrix} \hom^*(\bL_1, \bL_1) & \hom^*(\bL_2, \bL_1)[1] \\ 0 & \hom^*(\bL_2, \bL_2) \end{pmatrix},
\]
with $A_\infty$-operations given by summing over $\mu_S$ with all possible insertions of $f$. All such sums are finite since we can never insert more than one copy of $f$.

This $\calA$ is the upper-triangular version of the triangle algebra in \cite{PerutzSheridanRelative}, and has cohomological~unit
\[
\left[\begin{pmatrix} e_{\bL_1} & 0 \\ 0 & 0 \end{pmatrix}\right] = \left[\begin{pmatrix} 0 & 0 \\ 0 & e_{\bL_2} \end{pmatrix}\right].
\]
Following that paper, but dualising (i.e., reversing directions of maps), we define the maps $I_i$ by taking $I_i^1$ to be the inclusion of the $i$th diagonal summand into $\calA$, and taking all higher components \smash{$I_i^{\geq 2}$} to be zero. These maps are immediately seen to be cohomologically unital $A_\infty$-algebra homomorphisms, and are quasi-isomorphisms by dualising \cite[Lemma 2.17]{PerutzSheridanRelative}.

Next, we define $\Theta_\calA$ in a similar way to $\Theta$ but with perturbation data depending on where the inputs live (i.e., in which $\hom^*(\bL_i, \bL_j)$), and again summing over all possible insertions of~$f$. We choose these perturbation data to coincide with those defining $\Theta_i$ when all inputs lie in~$\hom^*(\bL_i, \bL_i)$. The composition $\Theta_\calA \circ I_i$ then tautologically coincides with $\Theta_i$, completing the~proof.
\end{proof}

Since we are working over a field $\kk$, the $A_\infty$-quasi-isomorphisms are automatically invertible. We could therefore straighten out the zigzags in the above results if desired.


\subsection{Transferring the closed-open map to matrix factorisations}

We next prove commutativity of the diagram in Theorem~\ref{TheoremB}. Explicitly, for an arbitrary class $\alpha \in \QH^*(X)$ we want to show that
\[
\rH({\LMF}_*) \circ \CO_\lambda(\alpha) = \rH(\Theta) \circ \CO^0_\bL(\alpha).
\]
After fixing a pseudocycle $Z \to X$ Poincar\'e dual to $\alpha$, and choosing suitable perturbation data, we have cocycles \smash{$\sigma \in \rCC^{\lvert \alpha \rvert}(\Fuk(X)_\lambda)$} and \smash{$\sigma_\bL \in \CF_S^{\lvert \alpha \rvert}(\bL, \bL)^\op$} representing $\CO_\lambda(\alpha)$ and $\CO^0_\bL(\alpha)$ respectively, as defined in Section~\ref{sscCOsetup}. Our task is then to show that \smash{${\LMF}_*(\sigma) - \Theta^1(\sigma_\bL)$} is exact, which we shall do by constructing geometrically a Hochschild cochain \smash{$\tau \in \rCC^{\lvert \alpha \rvert - 1}(\Fuk(X)_\lambda, \mf)$} satisfying
\[
\mu^1_{\rCC}(\tau) = {\LMF}_*(\sigma) - \Theta^1(\sigma_\bL).
\]
Such a $\tau$ is specified by its values $\tau^k(a_k, \dots, a_1)(x)$, for $a_1, \dots, a_k$ as in \eqref{eqai} and $x \in \LMF(K_0) = \CF_S^*(\bL, K_0)$.

\begin{Definition}
The element $\tau^k(a_k, \dots, a_1)(x) \in \LMF(K_k) = \CF_S^*(\bL, K_k)$ is defined by counting rigid discs of the form shown in Figure~\ref{figGamma}.
\begin{figure}[t]
\centering
\begin{tikzpicture}
\def\r{1.8cm}
\def\R{2.15cm}
\def\astart{50}
\def\aend{140}

\draw[line width=\circwidth] (0, 0) circle[radius=\r];
\draw[line width=\LLwidth] (-180: \r) arc[start angle=-180, end angle=0, radius=\r];

\draw (-180: \r) node[flout]{};
\draw (0:\r) node[flin]{};
\draw (0:\R) node{$x$};

\draw (0, 0) node[blob]{};

\draw (\astart:\r) node[flin]{};
\draw (\astart:\R) node{$a_1$};
\draw (\aend:\r) node[flin]{};
\draw (\aend:\R) node{$a_k$};

\foreach \i in {1, 2, 3} \draw ({\astart+\i*(\aend-\astart)/4}: \R) node{$\cdot$};
\end{tikzpicture}
\caption{The definition of $\tau^k(a_k, \dots, a_1)(x)$.}\label{figGamma}
\end{figure}
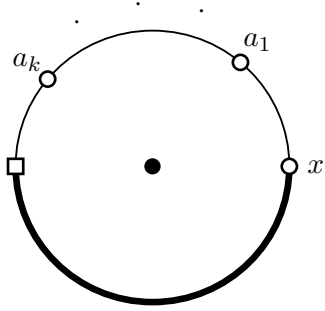
As in Section~\ref{sscCOsetup}, the solid dot indicates an interior marked point constrained to lie on $Z$. Each disc counts with the same sign as it would towards $\CO_\lambda(\alpha)(a_k, \dots, a_1, x)$ but twisted by \smash{$(-1)^{\lvert \alpha \rvert\lvert x \rvert + \lvert \alpha \rvert + \lvert x \rvert}$}. We write $\unsign{\tau}$ for the same cochain but without the sign twist. (Of course, $\CO_\lambda(\alpha)(a_k, \dots, a_1, x)$ does not strictly make sense, since $x$ is a morphism in $\Fuk_S(X)$ rather than $\Fuk(X)_\lambda$, but this is irrelevant for the purpose of defining signs.)

We choose regular perturbation data by induction, compatibly with degeneration of the domain in the following sense. After degeneration, each component looks like the domain of a~curve defining $\tau$ (but with strictly fewer inputs), $\sigma$, $\Theta^1$, $\sigma_\bL$, $\LMF$, or $\mu_\Fuk$, and we require that perturbation data induced by the degeneration agree with those already defined. See Figure~\ref{figGammaDegen} for examples of codimension-$1$ degenerations.
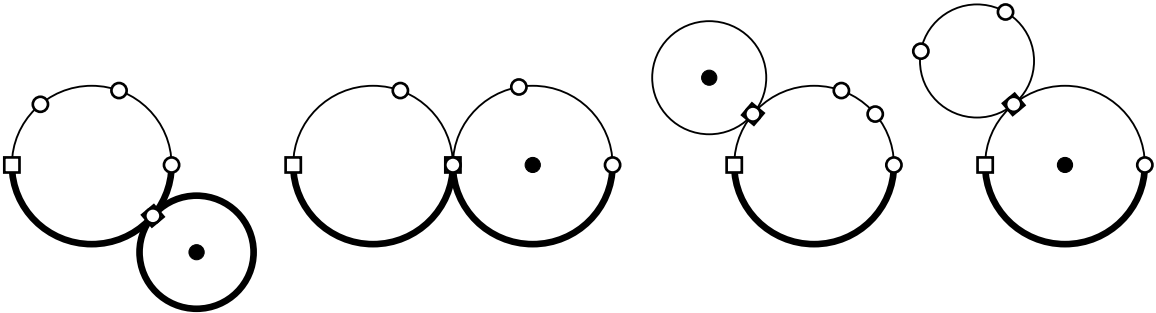
\begin{figure}[t]
\centering
\begin{tikzpicture}
\def\r{1.05cm}
\def\s{0.75cm}

\begin{scope}[xshift =-7.8cm]
\draw[line width=\circwidth] (0, 0) circle[radius=\r];
\draw[line width=\LLwidth] (-180: \r) arc[start angle=-180, end angle=0, radius=\r];
\draw (-180: \r) node[flout]{};
\draw (0:\r) node[flin]{};
\draw (70: \r) node[flin]{};
\draw (130: \r) node[flin]{};
\begin{scope}[shift={(-40:\r+\s)}]
\draw[line width=\LLwidth] (0, 0) circle[radius=\s];
\draw (0, 0) node[blob]{};
\end{scope}
\draw (-40: \r) node[flout, rotate=130]{};
\draw (-40: \r) node[flin]{};
\end{scope}

\begin{scope}[xshift=-4.1cm]
\draw[line width=\circwidth] (0, 0) circle[radius=\r];
\draw[line width=\LLwidth] (-180: \r) arc[start angle=-180, end angle=0, radius=\r];
\begin{scope}[shift={(0: 2*\r)}, rotate=0]
\draw[line width=\circwidth] (0, 0) circle[radius=\r];
\draw[line width=\LLwidth] (-180: \r) arc[start angle=-180, end angle=0, radius=\r];
\draw (0: \r) node[flin]{};
\draw (100: \r) node[flin]{};
\draw (0, 0) node[blob]{};
\end{scope}
\draw (-180: \r) node[flout]{};
\draw (0:\r) node[flout, rotate=0]{};
\draw (0:\r) node[flin]{};
\draw (70: \r) node[flin]{};
\end{scope}

\begin{scope}[xshift=1.7cm]
\draw[line width=\circwidth] (0, 0) circle[radius=\r];
\draw[line width=\LLwidth] (-180: \r) arc[start angle=-180, end angle=0, radius=\r];
\draw (-180: \r) node[flout]{};
\draw (0:\r) node[flin]{};
\draw (70: \r) node[flin]{};
\draw (40: \r) node[flin]{};
\begin{scope}[shift={(140:\r+\s)}]
\draw[line width=\circwidth] (0, 0) circle[radius=\s];
\draw (0, 0) node[blob]{};
\end{scope}
\draw (140: \r) node[flout, rotate=140]{};
\draw (140: \r) node[flin]{};
\end{scope}

\begin{scope}[xshift=5cm]
\draw[line width=\circwidth] (0, 0) circle[radius=\r];
\draw[line width=\LLwidth] (-180: \r) arc[start angle=-180, end angle=0, radius=\r];
\draw (0, 0) node[blob]{};
\draw (-180: \r) node[flout]{};
\draw (0:\r) node[flin]{};
\begin{scope}[shift={(130:\r+\s)}]
\draw[line width=\circwidth] (0, 0) circle[radius=\s];
\draw (170: \s) node[flin]{};
\draw (60: \s) node[flin]{};
\end{scope}
\draw (130: \r) node[flout, rotate=130]{};
\draw (130: \r) node[flin]{};
\end{scope}

\end{tikzpicture}
\caption{Example codimension-$1$ degenerations of a $\tau^2(a_2, a_1)$ disc. From left to right: $\Theta^1(\sigma_\bL)(a_2, a_1)$; $\LMF^1(a_2)\circ\tau^1(a_1)$; $\LMF^3\bigl(\sigma^0, a_2, a_1\bigr)$; $\tau^1\bigl(\mu_\Fuk^2(a_2, a_1)\bigr)$.}\label{figGammaDegen}
\end{figure}
We ask that the perturbation data induced by the degeneration agrees with those already defined on each component.
\end{Definition}

\begin{Remark}
The relationship between $\tau$ and $\sigma$ is analogous to the relationship between $\LMF$ and $\mu_\Fuk$ (add a single thick segment to the boundary). Similarly, the relationship between~$\tau$ and~$\LMF$ is analogous to the relationship between $\sigma$ and $\mu_\Fuk$ (add an interior marked point constrained to $Z$).
\end{Remark}

We now verify that $\tau$ has the desired property.

\begin{Proposition}
\label{propmu1Gamma}
The cochain $\tau$ satisfies
\[
\mu^1_{\rCC}(\tau) = {\LMF}_*(\sigma) - \Theta^1(\sigma_\bL).
\]
\end{Proposition}

Passing to cohomology gives the following.

\begin{Corollary}
\label{corThmBCommutes}
For any $\alpha$ in $\QH^*(X)$, we have
\[
\rH({\LMF}_*) \circ \CO_\lambda(\alpha) = \rH^*(\Theta) \circ \CO^0_\bL(\alpha),
\]
i.e., the diagram in Theorem~{\rm \ref{TheoremB}} commutes.
\end{Corollary}

\begin{proof}[Proof of Proposition~\ref{propmu1Gamma}]
Consider the boundaries of the compactified moduli spaces of discs equivalent to those defining $\tau$ but in virtual dimension $1$. This gives relations analogous to those proving that $\sigma$ is a Hochschild cocycle, i.e.,
\begin{equation}
\label{eqgammaCocycle}
\sum (-1)^{(\lvert \alpha \rvert - 1)\maltese_i} \mu_\Fuk(\dots, \sigma(\dots), \stackrel{i}{\dots}) + \sum (-1)^{\lvert \alpha \rvert + \maltese_i} \sigma(\dots, \mu_\Fuk(\dots), \stackrel{i}{\dots}) = 0,
\end{equation}
except for the presence of $\bL$ and the twisting of signs in $\tau$, $\Theta$, and $\LMF$. The analogue of \eqref{eqgammaCocycle} is then
\begin{gather*}
\unsign{\Theta}^1(\sigma_\bL)(\dots)(x) + \sum \unsign{\LMF}(\dots)(\unsign{\tau}(\dots)(x)) + \mu^1_S(\unsign{\tau}(\dots)(x))
\\ \qquad{} + \sum (-1)^{(\lvert \alpha \rvert - 1)(\lvert x \rvert - 1 + \maltese_i)} \unsign{\LMF}(\dots, \sigma(\dots), \stackrel{i}{\dots})(x) + (-1)^{\lvert \alpha \rvert} \unsign{\tau}(\dots)\bigl(\mu^1_S(x)\bigr)
\\ \qquad{} + \sum (-1)^{\lvert \alpha \rvert} \unsign{\tau}(\dots)\bigl(\unsign{\LMF}(\dots)(x)\bigr) + \sum (-1)^{\lvert \alpha \rvert + \lvert x \rvert - 1 + \maltese_i} \unsign{\tau} (\dots, \mu_\Fuk(\dots), \stackrel{i}{\dots})(x) = 0
\end{gather*}
for all $x$. Passing back to the sign-twisted versions and cancelling off $(-1)^{\lvert \alpha \rvert (\lvert x \rvert - 1)}$, this becomes
\begin{gather*}
\Theta^1\bigl(\sigma^0_\bL\bigr)(\dots)(x) + \sum (-1)^{\lvert \alpha \rvert + \maltese_i + 1} \LMF(\dots)(\tau(\stackrel{i}{\dots})(x)) + (-1)^{\lvert x \rvert} \mu^1_S(\tau(\dots)(x))
\\ \qquad{} + \sum (-1)^{(\lvert \alpha \rvert - 1) \maltese_i + 1} \LMF(\dots, \sigma(\dots), \stackrel{i}{\dots})(x) + (-1)^{\lvert x \rvert - 1} \tau(\dots)\bigl(\mu^1_S(x)\bigr)
\\ \qquad{} + \sum (-1)^{\lvert \alpha \rvert \maltese_i + \maltese_i + 1} \tau(\dots)(\LMF(\stackrel{i}{\dots})(x))
\\ \qquad{} + \sum (-1)^{\lvert \alpha \rvert - 1 + \maltese_i} \tau (\dots, \mu_\Fuk(\dots), \stackrel{i}{\dots})(x) = 0.
\end{gather*}
The third and fifth terms combine to give $\mu^1_\mf(\tau(\dots))(x)$, so we can drop the explicit $x$ and get
\begin{gather*}
\Theta^1\bigl(\sigma^0_\bL\bigr)(\dots) + \sum \mu^2_\mf(\LMF(\dots), \tau(\stackrel{i}{\dots})) + \mu^1_\mf(\tau(\dots))
\\ \qquad{} - \sum (-1)^{(\lvert \alpha \rvert - 1) \maltese_i} \LMF(\dots, \sigma(\dots), \stackrel{i}{\dots}) + \sum (-1)^{\lvert \alpha \rvert \maltese_i} \mu^2_\mf ( \tau(\dots), \LMF(\stackrel{i}{\dots}))
\\ \qquad{} + \sum (-1)^{\lvert \alpha \rvert - 1 + \maltese_i} \tau (\dots, \mu_\Fuk(\dots), \stackrel{i}{\dots}) = 0.
\end{gather*}
The second, third, fifth, and sixth terms are \smash{$\mu^1_{\rCC}(\tau)$}, using \eqref{eqCCdiff}, whilst the fourth term is~$-{\LMF}_*\sigma$, so rearranging gives the result.
\end{proof}

\subsection[Unitality of Theta]{Unitality of $\boldsymbol{\Theta}$}

It is now easy to complete the proof of Theorem~\ref{TheoremA}.

\begin{Corollary}
\label{corThetaUnital}
The $A_\infty$-algebra homomorphism $\Theta$ is cohomologically unital.
\end{Corollary}

\begin{proof}
This follows immediately from the commutative diagram in Theorem~\ref{TheoremB} using the unitality of \smash{$\CO^0_\bL$}, $\CO_\lambda$, and $\rH((\LMF[L][\lambda])_*)$.
\end{proof}

\begin{Remark}
\label{rmkRingUnitality}
We mentioned in Remark~\ref{rmkExtensions} that Theorem~\ref{TheoremA} holds over an arbitrary ground ring. However, the above proof of cohomological unitality breaks down because, as noted at the end of Section~\ref{sscCOsetup}, it is really the $2$-pointed closed-open map $\tCO_\lambda$ that is known to be unital, and this is only known to coincide with $\CO_\lambda$ over a field. To prove that $\Theta$ is cohomologically unital over an arbitrary ring, one can use the resemblance between $\rCC^*(\Fuk(X)_\lambda, \mf)$ and the $2$-pointed Hochschild complex to directly mimic the proof of unitality of the $\tCO_\lambda$ in \cite[Lemma~2.3]{SheridanFano}. We~leave the details to the interested reader.
\end{Remark}

\section{Proof of Theorem~\ref{TheoremC}}
\label{secThmC}

Unsurprisingly, in this section we prove Theorem~\ref{TheoremC}. Recall that the setup is as follows. We denote by $\calE$ the matrix factorisation $\CF^*_S\bigl(\bL, \LL\bigr)$ given by the image of $\LL \in \Fuk(X)_\lambda$ under $\LMF$, and by $\calB$ its endomorphism dg-algebra. The map $\Theta$ induces an $A_\infty$-algebra homomorphism
\[
\Theta_{\LL}\colon \ \CF_S^*(\bL, \bL)^\op \to \rCC^*\bigl(\CF^*\bigl(\LL, \LL\bigr), \calB\bigr).
\]
The monodromy of the local system $\calL$ on $\LL$ corresponds to a homomorphism $S \to \kk$, whose kernel we denote by $\m$. For any ideal $I$ in $S$ contained in $\m$ we write $R$ for $S/I$ and \smash{$\hatR$} for its $\m$-adic completion. All of the constructions we previously made over $S$ we can make over \smash{$\hatR$} instead, and we denote the results by \smash{$\CF^*_{\hatR}\bigl(\LL, \LL\bigr)^{(\op)}$}, \smash{$\hatCO$}, \smash{$\hatB$}, and so on. The statement of Theorem~\ref{TheoremC} is that
\[
\hatTheta\colon \ \CF^*_{\hatR}(\bL, \bL)^\op \to \rCC^*\bigl(\CF^*\bigl(\LL, \LL\bigr), \hatB\,\bigr)
\]
is a quasi-isomorphism, i.e., that the chain map \smash{$\hatTheta^1$} is a quasi-isomorphism.

\subsection{Simplifying the complexes}

The chain map \smash{$\hatTheta^1$} is $S$-linear (in fact, $\hatR$-linear) so respects the $\m$-adic filtrations on the domain and codomain and induces a map between the associated gradeds. The first step towards proving that it is a quasi-isomorphism is to study these filtrations and associated gradeds, which are much simpler. This is the goal of the present subsection.

As preliminary step, note that $S = \kk[\rH_1(L; \ZZ)]$ is a quotient of a Laurent polynomial ring over $\kk$ in finitely many variables (since $\rH_1(L; \ZZ)$ is finitely generated), so is Noetherian. Thus~$\m$ is finitely generated, and so the following properties hold.

\begin{Lemma}\ \label{lemCompletion}
\begin{enumerate}[label=$(\roman*)$,ref=(\roman*)]\itemsep=0pt
\item\label{compiii} $\hatR$ is $\m$-adically complete, i.e., the natural map \smash{$\widehat{R} \to \varprojlim \hatR / \m^p\hatR$} is an isomorphism.
\end{enumerate}
For all $p$ and $q$ with $p \leq q$, we have
\begin{enumerate}[resume,label=$(\roman*)$,ref=(\roman*)]\itemsep=0pt
\item\label{compi} \smash{$\m^p\hatR = \ker\bigl(\hatR \to R/\m^p\bigr)$}, and this is the $\m$-adic completion \smash{$\widehat{\m^pR}$} of \smash{$\m^pR = (\m^p + I) / I$}.
\item\label{compii} \smash{$\m^p\hatR / \m^q\hatR = \m^pR / \m^qR = (\m^p + I) / (\m^q + I)$}, and this is finite-dimensional over $\kk$.
\end{enumerate}
\end{Lemma}
\begin{proof}
Parts \ref{compiii} and \ref{compi} are special cases of \cite[Lemma 10.96.3]{stacks-project}.

The second equality in \ref{compii} is obvious, whilst the first comes from taking the short exact sequence
\[
0 \to \m^qR \to \m^pR \to \m^pR/\m^qR \to 0,
\]
applying \cite[Lemma 10.96.4]{stacks-project} to get a short exact sequence \smash{$0 \to \widehat{\m^qR} \to \widehat{\m^pR} \to \m^pR/\m^qR \to 0$}, and then using \ref{compi} to rewrite this as \smash{$0 \to \m^q\hatR \to \m^p\hatR \to \m^pR/\m^qR \to 0$}. To prove finite-dimensionality, it suffices to show that $\m^p / \m^q$ is finite-dimensional, and by induction it is enough to deal with the case $q = p+1$. For this, note that $\m^p / \m^{p+1} \cong (R/\m) \otimes_R \m^p$, and $R / \m = \kk$, so a finite $\kk$-spanning set for $\m^p / \m^{p+1}$ can be obtained by taking a finite generating set for $\m^p$ and tensoring with $1 \in R/\m$.
\end{proof}

We can now prove the two results we need.

\begin{Lemma}
\label{lemgrCF}
\smash{$\CF_{\hatR}^*(\bL, \bL)^{\op}$} is $\m$-adically complete and for each $p$ we have an identification of complexes
\[
\gr^p \CF^*_{\hatR}(\bL, \bL)^\op = \bigl((\m^p + I) \big/ \bigl(\m^{p+1} + I\bigr)\bigr) \otimes \CF^*\bigl(\LL, \LL\bigr)^\op.
\]
Undecorated tensor products are over $\kk$, as usual, and the differential on the right-hand side is defined to be the identity on the first tensor factor and the ordinary Floer differential on the~second.
\end{Lemma}

\begin{proof}
First note that \smash{$CF^*_{\hatR}(\bL, \bL)^\op$} is a free \smash{$\hatR$}-module of finite rank. Therefore, completeness follows from Lemma~\ref{lemCompletion}\,\ref{compiii}, whilst the description of \smash{$\gr^p CF^*_{\hatR}(\bL, \bL)^\op$} at the level of vector spaces follows from Lemma~\ref{lemCompletion}\,\ref{compii}. The differential on \smash{$\CF^*_{\hatR}(\bL, \bL)^\op$} is given in our diagrammatic notation by counting discs with thick boundary, a single boundary input, and a single boundary output. Reducing modulo $\m$ converts the thick boundary, which corresponds to an $S$ or $\hatR$ weight, to an ordinary boundary, corresponding to the parallel transport of the local system $\calL$. In other words, it reduces the differential to that on \smash{$\CF^*\bigl(\LL, \LL\bigr)^\op$}, with the \smash{$\bigl((\m^p + I) \big/ \bigl(\m^{p+1} + I\bigr)\bigr)$} factor just coming along for the ride.
\end{proof}

\begin{Lemma}
\label{lemgrCC}
\smash{$\rCC^*\bigl(\CF^*\bigl(\LL, \LL\bigr), \hatB\,\bigr)$} is $\m$-adically complete and for each $p$ we have an identification of complexes
\begin{gather}
\label{eqgrCC}
\gr^p \rCC^*\bigl(\CF^*\bigl(\LL, \LL\bigr), \hatB\,\bigr)\nonumber\\
\qquad{} = \bigl((\m^p + I) \big/ \bigl(\m^{p+1} +I\bigr)\bigr) \otimes \rCC^*(\CF^*\bigl(\LL, \LL\bigr), \eend_\kk^*\bigl(\CF\bigl(\LL, \LL\bigr)\bigr)\bigr).
\end{gather}
The differential on the left-hand side is the associated graded of \eqref{eqCCdiff}, restricted to the single object $\LL$ so that $\mu_{\mf}$ becomes \smash{$\mu_{\hatB}$} and $\LMF$ becomes \smash{$\hatPhi$}. The differential on the right-hand side is defined by the same formula reduced modulo $\m$ on the second tensor factor, and by the identity on the first tensor factor.
\end{Lemma}

\begin{proof}
For each $p$ the ideal $\m^p \subset S$ is finitely generated, say by $s_1, \dots, s_k$, so we have
\begin{align*}
\m^p \rCC^*\bigl(\CF^*\bigl(\LL, \LL\bigr), \hatB\,\bigr) &= \sum_j s_j \rCC^*\bigl(\CF^*\bigl(\LL, \LL\bigr), \hatB\,\bigr)
\\ &= \sum_j \rCC^*\bigl(\CF^*\bigl(\LL, \LL\bigr), s_j\hatB\,\bigr) = \rCC^*\bigl(\CF^*\bigl(\LL, \LL\bigr), \m^p\hatB\,\bigr).
\end{align*}
We also have for all $p$ and $q$ with $p \leq q$ that
\begin{align*}
\m^p \hatB / \m^q \hatB &= \hom^*_{\hatR}\bigl(\CF^*_{\hatR}\bigl(\bL, \LL\bigr), \m^p\CF^*_{\hatR}\bigl(\bL, \LL\bigr) / \m^q\CF^*_{\hatR}\bigl(\bL, \LL\bigr)\bigr)
\\ &= \hom^*_{\hatR}\bigl(\CF^*_{\hatR}\bigl(\bL, \LL\bigr), ((\m^p + I) / (\m^q + I)) \otimes \CF^*\bigl(\LL, \LL\bigr)\bigr)
\\ &= \hom^*_\kk\bigl(\CF^*\bigl(\LL, \LL\bigr), ((\m^p + I) / (\m^q + I)) \otimes \CF^*\bigl(\LL, \LL\bigr)\bigr)
\\ &= ((\m^p + I) / (\m^q + I)) \otimes \eend^*_\kk\bigl(\CF^*\bigl(\LL, \LL\bigr)\bigr).
\end{align*}
The second equality follows from similar arguments to Lemma~\ref{lemgrCF}.

We thus have for all $p$ that
\begin{align*}
& \gr^p \rCC^*\bigl(\CF^*\bigl(\LL, \LL\bigr), \hatB\,\bigr)
\\ &\qquad{} = \rCC^*\bigl(\CF^*\bigl(\LL, \LL\bigr), \m^p\hatB\,\bigr) \big/ \rCC^*\bigl(\CF^*\bigl(\LL, \LL\bigr), \m^{p+1}\hatB\,\bigr)
\\ &\qquad{} = \prod_{r\geq 0} \hom^*_\kk\bigl(\CF^*\bigl(\LL, \LL\bigr)[1]^{\otimes r}, \gr^p\hatB\,\bigr)
\\ &\qquad{} = \prod_{r\geq 0} \bigl((\m^p + I) \big/ \bigl(\m^{p+1} +I\bigr)\bigr) \otimes \hom^*_\kk\bigl(\CF^*\bigl(\LL, \LL\bigr)[1]^{\otimes r}, \eend^*_\kk\bigl(\CF^*\bigl(\LL, \LL\bigr)\bigr)\bigr)
\\ &\qquad{} = \bigl((\m^p + I) \big/ \bigl(\m^{p+1} +I\bigr)\bigr) \otimes \rCC^*\bigl(\CF^*\bigl(\LL, \LL\bigr), \eend_\kk^*\bigl(\CF\bigl(\LL, \LL\bigr)\bigr)\bigr),
\end{align*}
proving the second part of the lemma. In the fourth equality we used finite-dimensionality from Lemma~\ref{lemCompletion}\,\ref{compii}.

For the first part of the lemma, note that we similarly have
\begin{align*}
&\varprojlim \rCC^*\bigl(\CF^*\bigl(\LL, \LL\bigr), \hatB\,\bigr) \big/ \m^p \rCC^*\bigl(\CF^*\bigl(\LL, \LL\bigr), \hatB\,\bigr)
\\ &\qquad{}= \prod_{r\geq 0} \varprojlim \hom^*_\kk\bigl(\CF^*\bigl(\LL, \LL\bigr)[1]^{\otimes r}, \hatB \big/ \m^p\hatB\,\bigr)
\\ &\qquad{}= \prod_{r\geq 0} \hom^*_\kk\bigl(\CF^*\bigl(\LL, \LL\bigr)[1]^{\otimes r}, \varprojlim \hatB \big/ \m^p\hatB\,\bigr)
\\ &\qquad{}= \rCC^*\bigl(\CF^*\bigl(\LL, \LL\bigr), \hatB\,\bigr).
\end{align*}
Here in the last equality we used completeness of $\hatB$, which follows from Lemma~\ref{lemCompletion}\,\ref{compiii} plus the fact that $\hatB$ is a free $\hatR$-module of finite rank.
\end{proof}

\begin{Remark}
An equivalent way to describe the differential on the second tensor factor on the right-hand side of~\eqref{eqgrCC} is that we view \smash{$\eend_\kk^*\bigl(\CF^*\bigl(\LL, \LL\bigr)\bigr)$} as a $\CF^*\bigl(\LL, \LL\bigr)$-bimodule via~\smash{$\gr \hatPhi$}, which is the map induced by the Yoneda module \smash{$Y_{\LL}$}, viewed as an $A_\infty$-functor $\Fuk(X)_\lambda \to \mathrm{Ch}_\kk$ (chain complexes over $\kk$). Explicitly, writing \smash{$\mu^*_{\eend}$} for the $A_\infty$-algebra operations on $\eend_\kk^*\bigl(\CF^*\bigl(\LL, \LL\bigr)\bigr)$, the bimodule operations are
\begin{gather}
\bim{0}{0}(\zeta)(x) = -\mu^1_{\eend}(\zeta)(x) = (-1)^{\lvert x \rvert + 1}\bigl(\mu_\Fuk^1(\zeta(x)) - \zeta\bigl(\mu_\Fuk^1(x)\bigr)\bigr), \nonumber\\
\bim{k}{0}(a_k, \dots, a_1, \zeta)(x) = - \mu^2_{\eend}\bigl(Y_{\LL}(a_k, \dots, a_1), \zeta\bigr)(x) = (-1)^{\lvert x \rvert + 1}\mu_\Fuk(a_k, \dots, a_1, \zeta(x)), \nonumber\\
\bim{0}{l}(\zeta, a_l, \dots, a_1)(x) = (-1)^{\maltese_l + 1} \mu^2_{\eend}\bigl(\zeta, Y_{\LL}(a_l, \dots, a_1)\bigr)(x) = (-1)^{\lvert x \rvert}\zeta(\mu_\Fuk(a_l, \dots, a_1, x)), \nonumber\\
\bim{k}{l} = 0.\label{eqendBimod}
\end{gather}
Here $k$ and $l$ are positive integers, $x$ and the $a_i$ are in $\CF^*\bigl(\LL, \LL\bigr)$, and $\zeta$ is in $\eend_\kk^*\bigl(\CF^*\bigl(\LL, \LL\bigr)\bigr)$. The differential on $\rCC^*\bigl(\CF^*\bigl(\LL, \LL\bigr), \eend_\kk^*\bigl(\CF^*\bigl(\LL, \LL\bigr)\bigr)\bigr)$, as defined by \eqref{eqCCdiff} (restricted to $\LL$ and reduced modulo $\m$), is expressed in terms of these bimodule operations by \eqref{eqCCbimod}.
\end{Remark}

\subsection{Applying Eilenberg--Moore}

To leverage the simplifications, we have just made, we will use the Eilenberg--Moore comparison theorem. Before stating this, recall that a filtration $F^pC$ of a complex $C$ is \emph{exhaustive} if~$\smash{\bigcup_p F^pC = C}$ and \emph{complete} if the natural map \smash{$C \to \varprojlim C/F^pC$} is an isomorphism.

\begin{Theorem}[{Eilenberg--Moore comparison theorem~\cite{EilenbergMoore} and \cite[Theorem~5.5.11]{Weibel}}]
\label{thmEMcomparison}
Suppose~$C$ and $D$ are cochain complexes equipped with filtrations that are exhaustive and complete, and~${f \colon C \to D}$ is a filtered chain map. If there exists $r \geq 0$ such that $f$ induces an isomorphism between the $r$th pages, $E_r$, of the associated spectral sequences, then $f$ is a quasi-isomorphism.
\end{Theorem}

Unsurprisingly, we shall apply this to the chain map
\[
\hatTheta^1\colon \ \CF_{\hatR}^*(\bL, \bL)^\op \to \rCC^*\bigl(\CF^*\bigl(\LL, \LL\bigr), \hatB\,\bigr),
\]
where the domain and codomain are both equipped with the $\m$-adic filtrations. By Lemmas~\ref{lemgrCF} and~\ref{lemgrCC}, these filtrations are complete, and they are obviously exhaustive. We already noted that \smash{$\hatTheta^1$} preserves the filtrations, so, by Eilenberg--Moore, the generalised Theorem~\ref{TheoremC} is implied by the following result.

\begin{Proposition}
\label{propThmC}
The map \smash{$\hatTheta^1$} induces an isomorphism between the $E_1$ pages of the associated spectral sequences. In other words, the associated graded map
\[
\gr \hatTheta^1\colon \ \gr \CF^*_{\hatR}(\bL, \bL)^\op \to \gr \rCC^*\bigl(\CF^*\bigl(\LL, \LL\bigr), \hatB\,\bigr)
\]
is a quasi-isomorphism.
\end{Proposition}

\begin{proof}
By Lemmas~\ref{lemgrCF} and~\ref{lemgrCC}, for each $p$ we can view \smash{$\gr^p \hatTheta^1$} as a map
\begin{gather*}
\bigl((\m^p + I) \big/ \bigl(\m^{p+1} + I\bigr)\bigr) \otimes \CF^*\bigl(\LL, \LL\bigr)^\op
\\ \qquad{}\to \bigl((\m^p + I) \big/ \bigl(\m^{p+1} +I\bigr)\bigr) \otimes \rCC^*\bigl(\CF^*\bigl(\LL, \LL\bigr), \eend_\kk^*\bigl(\CF\bigl(\LL, \LL\bigr)\bigr)\bigr).
\end{gather*}
By inspecting the definition of $\Theta$, this map can be written as \smash{$\id_{(I+\m^p) / (I+\m^{p+1})} \otimes \theta$}, where $\theta$ is the reduction of \smash{$\hatTheta^1$} modulo $\m$. This can be seen by `reducing diagrams mod $\m$' as in Lemma~\ref{lemgrCF}. To show that \smash{$\gr \hatTheta^1$} is a quasi-isomorphism, it therefore suffices to show that
\[
\theta\colon \ \CF^*\bigl(\LL, \LL\bigr)^\op \to \rCC^*\bigl(\CF^*\bigl(\LL, \LL\bigr), \eend_\kk^*\bigl(\CF^*\bigl(\LL, \LL\bigr)\bigr)\bigr)
\]
is a quasi-isomorphism. And from the construction of $\Theta$ in Definition~\ref{defTheta}, we have an explicit formula for $\theta$, namely
\begin{equation}
\label{eqtheta}
\theta(c)(a_k, \dots, a_1)(x) = (-1)^{\lvert c \rvert (\lvert x \rvert - 1)}\mu_\Fuk(a_k, \dots, a_1, x, c)
\end{equation}
for all $c$, $a_i$, and $x$ in \smash{$\CF^*\bigl(\LL, \LL\bigr)^{(\op)}$}.

In Appendix~\ref{appHH}, we study the purely algebraic problem of understanding complexes of the form $\rCC^*(\calA, \hom_\kk^*(\calA, \calN))$, where $\calA$ is an $A_\infty$-algebra, $\calN$ is a (left) $\calA$-module, and both are cohomologically unital. Here $\hom_\kk^*(\calA, \calN)$ is viewed as an $\calA$-bimodule using the $\calA$-module actions on $\calA$ and $\calN$. In the case where \smash{$\calA = \calN = \CF^*\bigl(\LL, \LL\bigr)$}, this bimodule action on $\hom_\kk^*(\calA, \calN)$ agrees with the one defined on $\eend_\kk^*\bigl(\CF^*\bigl(\LL, \LL\bigr)\bigr)$ by \eqref{eqendBimod}. This can be proved by directly comparing \eqref{eqendBimod} with \eqref{eqhombimod} after setting $\mu_\calM = \mu_\calN = -\mu_\Fuk$ in the latter. The minus sign here is our standard convention for viewing an $A_\infty$-algebra as a module over itself.

We can therefore apply the analysis of Appendix~\ref{appHH} to $\rCC^*\bigl(\CF^*\bigl(\LL, \LL\bigr), \eend_\kk^*\bigl(\CF^*\bigl(\LL, \LL\bigr)\bigr)\bigr)$. In particular, Corollary~\ref{corCCN} tells us that the map
\[
\Pi\colon \ \rCC^*\bigl(\CF^*\bigl(\LL, \LL\bigr), \eend_\kk^*\bigl(\CF^*\bigl(\LL, \LL\bigr)\bigr)\bigr) \to \CF^*\bigl(\LL, \LL\bigr)
\]
given by \smash{$\Pi(\phi) = (-1)^{\lvert \phi \rvert}\phi^0(e_{\LL})$} is a quasi-isomorphism. Here $e_{\LL}$ is a chain-level representative for the unit in $\HF^*\bigl(\LL, \LL\bigr)$. To complete the proof of Proposition~\ref{propThmC}, and hence of the generalised Theorem~\ref{TheoremC}, it therefore suffices to show that the chain map $\Pi \circ \theta \colon \CF^*\bigl(\LL, \LL\bigr)^\op \to \CF^*\bigl(\LL, \LL\bigr)$ is a quasi-isomorphism.

To do this, we just plug \eqref{eqtheta} into the definition of $\Pi$. We see that for all $c \in \CF^*\bigl(\LL, \LL\bigr)^{(\op)}$
\[
\Pi \circ \theta(c) = (-1)^{\lvert c \rvert} \theta(c)^0(e_{\LL}) = \mu^2_\Fuk(e_{\LL}, c).
\]
When $c$ is a cocycle, this is cohomologous to $(-1)^{\lvert c \rvert}c$ by definition of cohomological unitality. Therefore, $\Pi \circ \theta$ induces the map $(-1)^*$ on \smash{$\HF^*\bigl(\LL, \LL\bigr)^{(\op)}$}, which is an isomorphism.
\end{proof}

\begin{Remark}
\label{rmkPiSign}
The reader may be puzzled by the fact that $\Pi \circ \theta$ is supposed to be a chain map, but the map $\CF^*\bigl(\LL, \LL\bigr)^\op \to \CF^*\bigl(\LL, \LL\bigr)$ given by $c \mapsto \mu^2_\Fuk(e_{\LL}, c)$ intertwines $\mu^1_\Fuk$ on the domain with \emph{minus} $\mu^1_\Fuk$ on the codomain. This mystery is resolved when one notes that the domain is being treated as an $A_\infty$-algebra but the codomain is being treated (via Corollary~\ref{corCCN}) as an $A_\infty$-module. So $-\mu^1_\Fuk$ is indeed the correct $\mu^1$ operation on the codomain.
\end{Remark}

\section[Reduction mod m\^{}2 and Tonkonog's criterion]{Reduction mod $\boldsymbol{\m^2}$ and Tonkonog's criterion}
\label{secTonk}

In \cite{TonkonogCO}, Tonkonog considers a loop $\gamma$ of Hamiltonian symplectomorphisms of $X$ preserving our Lagrangian $L$ setwise. Associated to $\gamma$ is a \emph{Seidel element} $S(\gamma) \in \QH^*(X)$, introduced in \cite{SeidelRep}, and Tonkonog's goal is to compute $\CO_{\LL}(S(\gamma)) \in \HH^*\bigl(\CF^*\bigl(\LL, \LL\bigr)\bigr)$. In Section~\ref{sscTonkStatement}, we recall his main result, which is a criterion for $\CO_{\LL}(S(\gamma))$ to be linearly independent of the Hochschild cohomology unit, or more generally for $\CO_{\LL}(S(\gamma)\qprod Q)$ to be linearly independent of $\CO_{\LL}(Q)$ for arbitrary fixed $Q \in \QH^*(X)$. We then discuss the reduction of $\CO^0_\bL$ modulo $\m^2$, and how this gives a new interpretation of his result (Theorem~\ref{TheoremD}), explaining the potentially mysterious hypothesis in a natural way. Throughout this section, we use a pearl model for Floer cohomology, as constructed by Biran--Cornea \cite{biran2007quantum}, and we write $\mu^k$ (without subscript) for the $A_\infty$-operations on $\CF^*\bigl(\LL, \LL\bigr)$.

\subsection{Tonkonog's criterion}\label{sscTonkStatement}

Let $\rho\colon \rH_1(L; \ZZ) \to \kk^\times$ denote the monodromy of the local system $\calL$ on $\LL$, as usual, and let $l \in \rH_1(L; \ZZ)$ denote the homology class of an orbit of $\gamma$ on $L$. Building on work of Charette--Cornea~\cite{CharetteCornea}, Tonkonog first calculates $\CO^0_{\LL}(S(\gamma))$.

\begin{Proposition}[{\cite[Theorem 1.7\,(a)]{TonkonogCO}}]
\label{propTonka}
We have
\[
\CO^0_{\LL}(S(\gamma)) = (-1)^{\eps(l)}\rho(l) \cdot 1_{\LL} \in \HF^*\bigl(\LL, \LL\bigr),
\]
where $1_{\LL}$ is the unit in $\HF^*\bigl(\LL, \LL\bigr)$ and $(-1)^{\eps(l)}$ is a specific sign that is unimportant for our purposes.
\end{Proposition}

Now assume that $\HF^*\bigl(\LL, \LL\bigr) \neq 0$. We then have a canonical linear \emph{PSS map}
\[
\PSS\colon \ \rH^1(L) \to \HF^1\bigl(\LL, \LL\bigr),
\]
introduced by Albers \cite{Albers}. In the context of the pearl model \cite{biran2007quantum}, this map comes from the inclusion of degree $1$ Morse cochains into the pearl complex. Tonkonog's criterion then reads as follows.

\begin{Theorem}[{\cite[Theorem 1.7\,(c)]{TonkonogCO}}]
\label{thmTonkonogCriterion}
For $Q \in \QH^*(X)$, if the classes $\CO_{\LL}(S(\gamma) \qprod Q)$ and $\CO_{\LL}(Q)$ are $\kk$-linearly dependent in $\HH^*\bigl(\CF^*\bigl(\LL, \LL\bigr)\bigr)$, then there exists $a \in \HF^*\bigl(\LL, \LL\bigr)$ such that for all $y \in \rH^1(L)$ we have
\begin{equation}\label{TonkHyp}
\mu^2(a, \PSS(y)) + \mu^2(\PSS(y), a) = \langle y, l \rangle \cdot \CO^0_{\LL}(Q).\tag{$\Diamond$}
\end{equation}
\end{Theorem}

\begin{Remark}
The statement in \cite{TonkonogCO} includes a factor of \smash{$(-1)^{\eps(l)}\rho(l)$} on the right-hand side, but this is irrelevant since it can be absorbed into $a$.
\end{Remark}

\subsection[Reduction modulo m\^{}2]{Reduction modulo $\boldsymbol{\m^2}$}\label{sscmodm2}

According to Theorems~\ref{TheoremB} and~\ref{TheoremC}, if $\CO_{\LL}(S(\gamma) \qprod Q)$ and $\CO_{\LL}(Q)$ are linearly dependent in $\HH^*\bigl(\CF^*\bigl(\LL, \LL\bigr)\bigr)$, then \smash{$\hatCO(S(\gamma) \qprod Q)$} and \smash{$\hatCO(Q)$} are linearly dependent in \smash{$\HF^*_{\hatR}(\bL, \bL)^\op$}. Here \smash{$\hatR$} is the $\m$-adic completion of $R = S/I$, where $I$ is any ideal of $S$ contained in $\m$. We shall show that Theorem~\ref{thmTonkonogCriterion} is really a statement about linear dependence of \smash{$\hatCO(S(\gamma) \qprod Q)$} and \smash{$\hatCO(Q)$} for $I = \m^2$. Note that in this situation $R$ is already $\m$-adically complete so \smash{$\hatR = R$}. But we will leave the hat on \smash{$\hatCO$} to distinguish if from \smash{$\CO^0_\bL$} (which is over $S$, rather than $R$).

Our first step is to describe the algebra $R = S/\m^2$ itself.

\begin{Lemma}\label{lemRDescription}
There is a natural isomorphism of unital $\kk$-algebras
\[
R \cong \kk \oplus \rH_1(L),
\]
where the product on $\kk \oplus \rH_1(L)$ is defined by $(\lambda_1, \sigma_1) (\lambda_2, \sigma_2) = (\lambda_1\lambda_2, \lambda_1\sigma_2+\lambda_2\sigma_1)$.
\end{Lemma}
\begin{proof}
As a $\kk$-vector space, $S$ splits as $\kk \oplus \m$, where the first summand is spanned by the unit. The induced product on $\kk \oplus \m$ is
\[
(\lambda_1, m_1)(\lambda_2, m_2) = (\lambda_1\lambda_2, \lambda_1m_2+\lambda_2m_1+m_1m_2).
\]
Then $S/\m^2 \cong \kk \oplus \bigl(\m / \m^2\bigr)$, with product of the form claimed in the lemma, and it remains to show that $\m/\m^2$ is naturally identified with $\rH_1(L)$.

Define $\alpha\colon S \to \rH_1(L)$ by $z^\gamma \mapsto \rho(\gamma)\gamma$, extended $\kk$-linearly. Here $z$ is the formal variable in the group algebra $S = \kk[\rH_1(L; \ZZ)]$, and $\rho$ is the monodromy representation $\rH_1(L; \ZZ) \to \kk^\times$ as above. The ideal $\m$ is generated by the $z^\gamma - \rho(\gamma)$, from which it is easily checked that $\alpha\bigl(\m^2\bigr)=0$, so $\alpha$ induces a linear map $\bar\alpha \colon \m/\m^2 \to \rH_1(L)$. This will be our isomorphism.

To construct the inverse, consider the $\ZZ$-bilinear map $\kk \times \rH_1(L; \ZZ) \to \m / \m^2$ given by
\[
(\lambda, \gamma) \mapsto \lambda\bigl(\rho(\gamma)^{-1}z^\gamma-1\bigr).
\]
Note that linearity in $\rH_1(L; \ZZ)$ follows from the fact that for all $\gamma_1$ and $\gamma_2$ the element
\[
\bigl(\rho(\gamma_1)^{-1}z^{\gamma_1}-1\bigr)\bigl(\rho(\gamma_2)^{-1}z^{\gamma_2}-1\bigr)
\]
lies in $\m^2$, so
\[
\rho(\gamma_1+\gamma_2)^{-1}z^{\gamma_1+\gamma_2} - 1 = \bigl(\rho(\gamma_1)^{-1}z^{\gamma_1}-1\bigr) + \bigl(\rho(\gamma_2)^{-1}z^{\gamma_2}-1\bigr)
\]
in $\m / \m^2$. This $\ZZ$-bilinear map induces a $\ZZ$-linear map $\beta \colon \rH_1(L) = \kk \otimes_\ZZ \rH_1(L; \ZZ) \to \m / \m^2$ which is manifestly $\kk$-linear. One can then check directly that $\bar\alpha$ and $\beta$ are mutually inverse.
\end{proof}

Our next goal is to understand the algebra $\HF^*_R(\bL, \bL)^{(\op)}$, and we will drop the $(\op)$ as it is irrelevant for present purposes. By splitting $R$ as $\kk \oplus \rH_1(L)$ we get an identification
\begin{equation}
\label{eqCFsplitting}\
\CF^*_R(\bL, \bL) = \CF^*\bigl(\LL, \LL\bigr) \oplus \bigl(\rH_1(L) \otimes \CF^*\bigl(\LL, \LL\bigr)\bigr).
\end{equation}
Here, unsurprisingly, $\CF^*_R(\bL, \bL)$ means the Floer complex of $\bL$ over $R$, i.e., the reduction of $\CF^*_S(\bL, \bL)$ modulo $I = \m^2$. With respect to this splitting, the differential has the form
\[
\begin{pmatrix}
\mu^1 & 0
\\ \mu^1_{\rH_1(L)} & \id_{\rH_1(L)} \otimes \mu^1
\end{pmatrix},
\]
where $\mu_{\rH_1(L)} \colon \CF^*\bigl(\LL, \LL\bigr) \to \rH_1(L) \otimes \CF^*\bigl(\LL, \LL\bigr)$ denotes the usual $\mu^1_S$ differential but with $S$ coefficients reduced to $\rH_1(L)$ via the map $\alpha$ from the proof of Lemma~\ref{lemRDescription}.

Filtering the complex using this splitting, which is essentially equivalent to filtering $\m$-adically, we obtain a spectral sequence whose $E_1$ page is
\[
\HF^*\bigl(\LL, \LL\bigr) \oplus \bigl(\rH_1(L) \otimes \HF^*\bigl(\LL, \LL\bigr)\bigr),
\]
with differential $\diff_1$ induced by $\mu^1_{\rH_1(L)}$. Later differentials trivially vanish, so we get a surjection
\[
\pi\colon \ \HF^*_R(\bL, \bL) \to \ker \diff_1 \subset \HF^*\bigl(\LL, \LL\bigr),
\]
induced by projection of the $E_1$ page to the first summand. The kernel of $\pi$ is $\coker \diff_1$, mapping into $\HF^*_R(\bL, \bL)$ via the map
\[
\iota\colon \ \rH_1(L) \otimes \HF^*\bigl(\LL, \LL\bigr) \to \HF^*_R(\bL, \bL)
\]
induced by inclusion of the second summand into the $E_1$ page. Using the description of the product from Lemma~\ref{lemRDescription}, we get
\[
w \cdot \iota(x) = \iota(\pi(w) \cdot x)
\]
for all $w \in \HF^*_R(\bL, \bL)$ and $x \in \rH_1(L) \otimes \HF^*\bigl(\LL, \LL\bigr)$. Note also that for $Q \in \QH^*(X)$, we have \smash{$\pi\bigl(\hatCO(Q)\bigr) = \CO^0_{\LL}(Q)$}.

\begin{Remark}
If we had used a Hamiltonian chord model for $\CF^*$, rather than a pearl model, then \eqref{eqCFsplitting} would not quite be true. $\CF^*_R(\bL, \bL)$ would still be a free $R$-module whose reduction modulo $\m$ was equal to $\CF^*\bigl(\LL, \LL\bigr)$, but there would not be a canonical identification of the former with $R \otimes \CF^*\bigl(\LL, \LL\bigr)$. This is because the generators of $\CF^*_R(\bL, \bL)$ would be morphism spaces between \emph{different} fibres (namely, the fibres over the endpoints of each chord) of an abstract rank-$1$ local system over $R$. Things simplify in the pearl model since we effectively use chords of length zero, so the generators are endomorphism spaces of single fibres, and these endomorphism spaces are canonically identified with $R$.
\end{Remark}

With this in place, we are ready to reprove Tonkonog's criterion.

\subsection{Tonkonog's criterion revisited}
\label{sscTonkRevisit}

Fix an arbitrary $Q \in \QH^*(X)$. We wish to show that if
$\CO_{\LL}(S(\gamma) \qprod Q)$ and $\CO_{\LL}(Q)$
are $\kk$-linearly dependent in $\HH^*\bigl(\CF^*\bigl(\LL, \LL\bigr)\bigr)$, then condition \eqref{TonkHyp} holds. We will in fact prove the following result, which implies Tonkonog's criterion by the discussion at the beginning of Section~\ref{sscmodm2}.

\begin{Proposition}
\label{propTonkonogInterpretation}
Taking \smash{$\hatR = R = S/\m^2$} as above, the classes \smash{$\hatCO(S(\gamma) \qprod Q)$} and \smash{$\hatCO(Q)$} are $\kk$-linearly dependent in \smash{$\HF^*_R(\bL, \bL)$} if and only if condition \eqref{TonkHyp} holds.
\end{Proposition}
\begin{proof}
Suppose first that $\CO^0_{\LL}(Q) = 0$. Then \smash{$\pi\bigl(\hatCO(Q)\bigr) = 0$} (notation as in Section~\ref{sscmodm2}) so~\smash{$\hatCO(Q) = \iota(x)$} for some \smash{$x \in \rH_1(L) \otimes \HF^*\bigl(\LL, \LL\bigr)$}. Thus
\begin{align*}
\hatCO(S(\gamma)\qprod Q) &{}= \hatCO(S(\gamma)) \cdot \hatCO(Q) = \hatCO(S(\gamma)) \cdot \iota(x) = \iota\bigl(\pi\bigl(\hatCO(S(\gamma))\bigr) \cdot x\bigr)
\\ &{}= \iota\bigl(\CO^0_{\LL}(S(\gamma)) \cdot x\bigr) = \iota\bigl((-1)^{\eps(l)}\rho(l) \cdot x\bigr) = (-1)^{\eps(l)} \rho(l) \cdot \hatCO(Q),
\end{align*}
where the fifth equality uses Proposition~\ref{propTonka}. So \smash{$\hatCO(S(\gamma) \qprod Q)$} and \smash{$\hatCO(Q)$} are linearly dependent. But also \eqref{TonkHyp} holds, for $a=0$, so the result is proved in this case.

From now on, suppose that \smash{$\CO^0_{\LL}(Q) \neq 0$}. By considering their images under $\pi$, we see that~\smash{$\hatCO(S(\gamma) \qprod Q)$} and \smash{$\hatCO(Q)$} are linearly dependent if and only if
\[
\hatCO(S(\gamma) \qprod Q) = (-1)^{\eps(l)}\rho(l) \cdot \hatCO(Q).
\]
This in turn is equivalent to
\begin{equation}
\label{eqLD}
\bigl(\hatCO(S(\gamma)) - (-1)^{\eps(l)}\rho(l) \cdot 1_\bL\bigr) \cdot \hatCO(Q) = 0.
\end{equation}
The same proof as for Proposition~\ref{propTonka} shows that \smash{$\CO^0_\bL(S(\gamma)) = (-1)^{\eps(l)}z^l \cdot 1_\bL$} in $\HF^*_S(\bL, \bL)$, so
\[
\hatCO(S(\gamma)) - (-1)^{\eps(l)}\rho(l) \cdot 1_\bL = \iota\bigl((-1)^{\eps(l)}\rho(l) \cdot l\otimes 1_{\LL}\bigr).
\]
Thus \eqref{eqLD} is equivalent to
\[
\iota\bigl((-1)^{\eps(l)}\rho(l) \cdot l \otimes \pi\bigl(\hatCO(Q)\bigr)\bigr) = 0, \qquad \text{i.e.,} \qquad \iota\bigl(l \otimes \CO^0_{\LL}(Q)\bigr) = 0.
\]
This is the case if and only if \smash{$l \otimes \CO^0_{\LL}(Q)$} is in the image of the differential $\diff_1$ induced by~\smash{$\mu^1_{\rH_1(L)}$}, i.e., if and only if there exists \smash{$a \in \HF^*\bigl(\LL, \LL\bigr)$} such that
\[
\diff_1(a) = l \otimes \CO^0_{\LL}(Q) \qquad \text{in} \ \rH_1(L) \otimes \HF^*\bigl(\LL, \LL\bigr).
\]

To test an equality in $\rH_1(L) \otimes \HF^*\bigl(\LL, \LL\bigr)$, it suffices to test it after pairing the $\rH_1(L)$ part with each $y \in \rH^1(L)$. So, by the previous paragraph, \smash{$\hatCO(S(\gamma) \qprod Q)$} and \smash{$\hatCO(Q)$} are linearly dependent if and only if there exists \smash{$a \in \HF^*\bigl(\LL, \LL\bigr)$} such that for all $y \in \rH^1(L)$ we have
\[
\langle y, \diff_1(a)\rangle = \langle y, l\rangle \otimes \CO^0_{\LL}(Q) \qquad \text{in} \ \rH_1(L) \otimes \HF^*\bigl(\LL, \LL\bigr).
\]
To complete the proof, it thus suffices to show that for all $a$ and $y$ we have
\[
\langle y, \diff_1(a)\rangle = \mu^2(a, \PSS(y)) + \mu^2(\PSS(y), a)
\]
We shall actually prove this up to an overall sign, which may depend on the degree of $a$, but that is sufficient for our purposes.

Let our pearl complexes be defined using a Morse function $f$ on $L$, and a path $J_t$ of almost complex structures. So generators of the complexes are critical points of $f$, and the differentials count pearly trajectories weighted by monodromy in $\kk$ (from $\rho$), $S$, or $R$ depending on the coefficients. Explicitly, a pearly trajectory is a Morse flowline interrupted by bivalent $J_t$-holomorphic discs, where each disc is parametrised so that after deleting the marked points the domain is~${\RR \times [0, 1]}$ with holomorphic coordinate $s+it$, input at $s \to \infty$, and output at $s \to -\infty$.

The term $\mu^2(a, \PSS(y))$ counts Y-shaped pearly trees whose inputs are (linear combinations~of) critical points of $f$ representing $a$ and $y$ respectively, weighted by monodromy from $\rho$. This is shown on the left in Figure~\ref{figPearlyDef}, where each bivalent disc in the tree is $J_t$-holomorphic, as above, whilst the trivalent disc at the centre carries a domain-dependent almost complex structure and Hamiltonian perturbation (indicated by the shading) to ensure transversality of the moduli space.
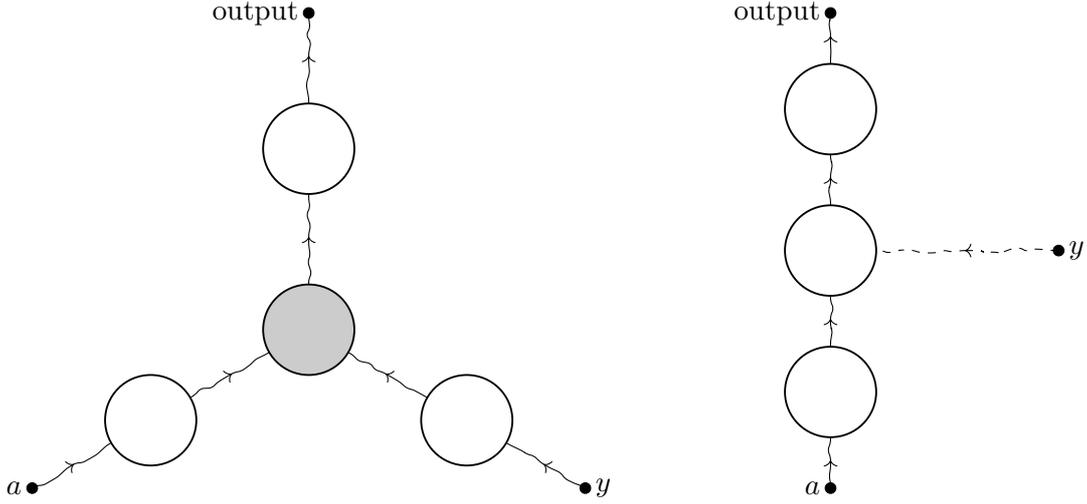
\begin{figure}[t]
\centering
\begin{tikzpicture}
\def\r{0.6cm}
\def\c{0.6}
\def\blobsize{4pt}

\begin{scope}
\def\a{\c*cos(30)}
\def\b{\c*sin(30)}
\draw (0, 0) node[blob]{};
\draw (0, 0) node[anchor=east]{$a$};
\flow{(0, 0)}{($\a*(2, 0) + \b*(0, 2)$)}{0}
\draw[line width=\circwidth] ($\a*(3, 0) + \b*(0, 3)$) circle[radius=\r];
\flow{($\a*(4, 0) + \b*(0, 4)$)}{($\a*(6, 0) + \b*(0, 6)$)}{10}

\draw[line width=\circwidth, fill=gray, fill opacity=0.4] ($\a*(7, 0) + \b*(0, 7)$) circle[radius=\r];

\flow{($\a*(10, 0) + \b*(0, 4)$)}{($\a*(8, 0) + \b*(0, 6)$)}{20}
\draw[line width=\circwidth] ($\a*(11, 0) + \b*(0, 3)$) circle[radius=\r];
\flow{($\a*(14, 0)$)}{($\a*(12, 0) + \b*(0, 2)$)}{3}
\draw ($\a*(14, 0)$) node[blob]{};
\draw ($\a*(14, 0)$) node[anchor=west]{$y$};

\flow{($\a*(7, 0) + \b*(0, 7) + \c*(0, 1)$)}{($\a*(7, 0) + \b*(0, 7) + \c*(0, 3)$)}{6}
\draw[line width=\circwidth] ($\a*(7, 0) + \b*(0, 7) + \c*(0, 4)$) circle[radius=\r];
\flow{($\a*(7, 0) + \b*(0, 7) + \c*(0, 5)$)}{($\a*(7, 0) + \b*(0, 7) + \c*(0, 7)$)}{8}
\draw ($\a*(7, 0) + \b*(0, 7) + \c*(0, 7)$) node[blob]{};
\draw ($\a*(7, 0) + \b*(0, 7) + \c*(0, 7)$) node[anchor=east]{$\mathrm{output}$};
\end{scope}

\begin{scope}[xshift=10.5cm]
\def\a{0}
\def\b{\c*1.125}

\draw (0, 0) node[blob]{};
\draw (0, 0) node[anchor=east]{$a$};
\flow{(0, 0)}{($\b*(0, 1)$)}{12}
\draw[line width=\circwidth] ($\b*(0, 1) + \c*(0, 1)$) circle[radius=\r];
\flow{($\b*(0, 1) + \c*(0, 2)$)}{($\b*(0, 2) + \c*(0, 2)$)}{17}
\draw[line width=\circwidth] ($\b*(0, 2) + \c*(0, 3)$) circle[radius=\r];
\flow{($\b*(0, 2) + \c*(0, 4)$)}{($\b*(0, 3) + \c*(0, 4)$)}{14}
\draw[line width=\circwidth] ($\b*(0, 3) + \c*(0, 5)$) circle[radius=\r];
\flow{($\b*(0, 3) + \c*(0, 6)$)}{($\b*(0, 4) + \c*(0, 6)$)}{15}
\draw ($\b*(0, 4) + \c*(0, 6)$) node[blob]{};
\draw ($\b*(0, 4) + \c*(0, 6)$) node[anchor=east]{$\mathrm{output}$};

\flow[dashed]{($(3, 0) + \b*(0, 2) + \c*(0, 3)$)}{($\c*(1, 0) + \b*(0, 2) + \c*(0, 3)$)}{19}
\draw ($(3, 0) + \b*(0, 2) + \c*(0, 3)$) node[blob]{};
\draw ($(3, 0) + \b*(0, 2) + \c*(0, 3)$) node[anchor=west]{$y$};
\end{scope}

\end{tikzpicture}
\caption{Trees computing $\mu^2(a, \PSS(y))$ (left) and $\mu^1(a)$ weighted by the intersection of the $t=1$ boundary with $y$ (right). The shaded disc and dashed flowline carry different auxiliary data from the others.}\label{figPearlyDef}
\end{figure}
The diagram shows one bivalent disc on each leg, but this is purely illustrative; there may be arbitrarily many, including zero.

Now consider deforming the almost complex structure and Hamiltonian perturbation on the trivalent disc so that it becomes a (possibly constant) $J_t$-holomorphic disc from the $a$-input to the output, with the $y$-input occurring as a marked point on the $t=1$ boundary. We can ensure transversality for these configurations by deforming the Morse function on the $y$-leg of the tree to a nearby Morse function $f'$. A standard cobordism argument shows that counts of old and new configurations agree at cohomology level.

The upshot is that $\mu^2(a, \PSS(y)) \in \HF^*\bigl(\LL, \LL\bigr)$ can also be computed by counting pearly trajectories from $a$ but with an additional input leg to the $t=1$ boundary corresponding to an $f'$ pearly trajectory from $y$. This additional input leg cannot contain any discs, otherwise deleting the leg would produce a pearly trajectory from $a$ of negative virtual dimension. So, up to an overall sign depending on $\lvert a \rvert$ and our conventions, $\mu^2(a, \PSS(y))$ computes $\mu^1(a)$ weighted by the pairing of the $t=1$ boundary path with $y$, as shown on the right in Figure~\ref{figPearlyDef}. (Again the diagram is purely illustrative and there may be any number of discs before or after the one receiving the input from $y$. The dashed flowline indicates the different Morse function, $f'$.) Similarly for~$\mu^2(\PSS(y), a)$ but now for the $t=0$ boundary path, oriented the opposite way. Note that in both cases the $f'$-leg does not contribute to the $\rho$ weighting of the configuration as it can be contracted away. Therefore, up to sign, $\mu^2(a, \PSS(y)) + \mu^2(\PSS(y), a)$ computes~$\mu^1(a)$ weighted by the pairing of the whole boundary with $y$. This is exactly $\pm \langle y, \diff_1(a)\rangle$, as claimed.
\end{proof}

\section{Real toric Lagrangians}
\label{secRealToric}

In this section, we apply our results to prove Theorem~\ref{TheoremE}. Throughout, we will take $X$ to be a compact monotone toric manifold of minimal Chern number $N_X \geq 2$, and assume that our coefficient field $\kk$ has characteristic $2$. We view $X$ as a symplectic reduction of $\CC^N$, via the Delzant construction; see \cite[Chapter~29]{CdS} for more details and Appendix~\ref{sscToricSetup} for our conventions. Our Lagrangian $\LL$ will be the real locus of $X$, defined to be the symplectic reduction of $\RR^N$ under this construction, and we equip it with the trivial local system $\calL$. This real locus is a valid object of $\Fuk(X)^\mathrm{un}_0$, and of $\Fuk(X)_0$ if it is orientable: it is monotone with minimal Maslov number equal to $N_X$; condition~\ref{itm2} from Section~\ref{sscFuk} holds automatically since $X$ is simply connected; the pin structure in condition~\ref{itm3} is not needed since we are working in characteristic $2$, whilst the $\ZZ/2$-grading exists when needed (i.e., when $L$ is orientable) by Remark~\ref{rmkUngraded}; and $W_L(\calL) = 0$ by Remark~\ref{rmkWLcancellation}. Theorem~\ref{TheoremE} asserts that $\LL$ split-generates $\Fuk(X)_0$ if it is orientable and split-generates $\Fuk(X)_0^\mathrm{un}$ in general. These will follow from the generation criterion if we can prove injectivity of
\[
\CO_{\LL}\colon \ \QH^*(X) \to \HH^*\bigl(\CF^*\bigl(\LL, \LL\bigr)\bigr).
\]

By Theorems~\ref{TheoremB} and~\ref{TheoremC}, it suffices to show that
\[
\hatCO\colon \ \QH^*(X) \to \HF_{\hatR}^*(\bL, \bL)
\]
is injective for some ideal $I \subset S = \kk[\rH_1(L; \ZZ)]$ contained in $\m$. Recall that $\hatR$ is the $\m$-adic completion of $R = S/I$, and because the local system $\calL$ is trivial we can describe $\m$ explicitly as
\[
(z^\gamma - 1 \mid \gamma \in \rH_1(L; \ZZ)),
\]
i.e., the ideal generated by expressions of the form $z^\gamma - 1$ as $\gamma$ ranges over $\rH_1(L; \ZZ)$. We will prove that \smash{$\hatCO$} is injective for $I = \bigl(z^{2\gamma} - 1 \mid \gamma \in \rH_1(L; \ZZ)\bigr)$. In this case, we have
\[
R = \kk[\rH_1(L; \ZZ) \otimes_\ZZ \ZZ/2] = \kk[\rH_1(L; \ZZ/2)].
\]

\begin{Remark}
If $\rH_1(L; \ZZ/2)$ has dimension $d$ over $\ZZ/2$, then we have
\[
R \cong \kk[z_1, \dots, z_d] \big/\bigl(z_1^2-1, \dots, z_d^2-1\bigr) = \kk[z_1, \dots, z_d] \big/\bigl((z_1-1)^2, \dots, (z_d-1)^2\bigr).
\]
(Recall that $\kk$ has characteristic $2$!) Under this isomorphism, $\m R$ is $(z_1 -1, \dots, z_d - 1)$, so by the pigeonhole principle we have \smash{$\m^{d+1} R = 0$}. Hence \smash{$\hatR = R$} and, as in Section~\ref{secTonk}, no completion is actually needed. However, also as in Section~\ref{secTonk}, we will leave the hat on \smash{$\hatCO$} in order to distinguish it from \smash{$\CO^0_\bL \colon \QH^*(X) \to \HF^*_S(\bL, \bL)$}.
\end{Remark}

\subsection{Outline of the argument}
\label{sscThmDOutline}

Recall from Section~\ref{sscMainResults} that the main step in proving injectivity of $\hatCO \colon \QH^*(X) \to \HF_R^*(\bL, \bL)$ is to establish a commutative diagram of $\kk$-algebras
\begin{equation}
\label{eqDRdiagram}
\begin{tikzcd}[column sep=5em, row sep=2.5em]
\QH^*_R(X) \arrow{d}{\pi} \arrow{r}{\Frob_R} & \QH^{2*}_R(X) \arrow{d}{\calD_R}[swap, anchor=center, rotate=-90, yshift=-1ex]{\cong}
\\ \QH^*(X) \arrow{r}{\hatCO} & \HF^*_R(\bL, \bL).
\end{tikzcd}
\end{equation}
Here $\QH^*_R(X)$ is an extension of quantum cohomology to $R$ coefficients, $\pi$ is reduction modulo~$\m$ (which sends every monomial in $R$ to $1$), $\calD_R$ is an isomorphism we will define, and $\Frob_R$ is `the $\kk$-linear extension of Frobenius' in the following sense. Take $\kk = \ZZ/2$ and consider the Frobenius map $x \mapsto x^2$ on $\QH^*_R(X)$; now pass to general $\kk$ by applying $\kk \otimes_{\ZZ/2} {-}$ and extending the map $\kk$-linearly. Explicitly, a general element of $\QH^*_R(X)$ is of the form \smash{$\sum_j \lambda_j z^{A_j} x_i$} for $\lambda_i \in \kk$, $A_i \in \rH_1(L; \ZZ/2)$, and $x_i \in \QH^*(X; \ZZ/2)$, and
\[
\Frob_R \biggl(\sum_j \lambda_j z^{A_j} x_j\biggr) = \sum_j \lambda_j z^{2A_j} x_j^2 = \sum_j \lambda_j x_j^2.
\]

\begin{Remark}
In \cite[Theorem 1.13]{TonkonogCO}, Tonkonog constructed a similar diagram over $\kk$, instead of~$R$, namely
\begin{equation}\label{eqDdiagram}
\begin{tikzcd}[column sep=5em, row sep=2.5em]
\QH^*(X) \arrow{r}{\Frob} \arrow{dr}[swap]{\CO^0_{\LL}} & \QH^{2*}(X) \arrow{d}{\calD}[swap, anchor=center, rotate=-90, yshift=-1ex]{\cong}
\\ & \HF^*\bigl(\LL, \LL\bigr).
\end{tikzcd}
\end{equation}
Here $\Frob$ is the $\kk$-linear extension of the Frobenius map on $\QH^*(X)$, and $\calD$ is an isomorphism constructed by Haug \cite{Haug} and Hyvrier \cite{Hyvrier}. Our diagram extends Tonkonog's, in the sense that reducing \eqref{eqDRdiagram} modulo $\m$ (which collapses the left-hand vertical arrow to an identity map) gives~\eqref{eqDdiagram}.

The statement of \cite[Theorem 1.13]{TonkonogCO} is not strictly correct: it has the Frobenius map itself in place of $\Frob$, but the proof given is really about $\Frob$ instead. The result is used by Evans--Lekili in \cite[Example 7.2.4]{EvansLekiliGeneration}, in their proof that $\LL$ split-generates $\Fuk(X)_0$ when $L$ is orientable, in order to show that \smash{$\CO^0_{\LL}(e_\alpha) \neq 0$} for any non-zero idempotent $e_\alpha$ in $\QH^*(X)$. But it is straightforward to modify their argument to use $\Frob$ instead, as follows. Take a $\ZZ/2$-basis $x_i$ for the kernel of Frobenius $\QH^*(X; \ZZ/2) \to \QH^{2*}(X; \ZZ/2)$, and let $y_j$ be a basis for a complementary subspace of $\QH^*(X; \ZZ/2)$. Note that the $y_j^2$ are $\ZZ/2$-linearly independent. Now view the $x_i$ and $y_j$ as elements of $\QH^*(X)$, i.e., tensor by $\kk$. Note that they form a $\kk$-basis, and that the $y_j^2$ are now $\kk$-linearly independent. Writing our non-zero idempotent $e_\alpha$ as $\sum_i \lambda_i x_i + \sum_j \mu_jy_j$ for $\lambda_i, \mu_j \in \kk$, we have
\[
\Frob(e_\alpha) = \sum_i \lambda_i x_i^2 + \sum_j \mu_j y_j^2 = \sum_j \mu_j y_j^2.
\]
We want to show \smash{$\CO^0_{\LL}(e_\alpha) \neq 0$}, and by \eqref{eqDdiagram} it suffices to show that $\Frob(e_\alpha) \neq 0$. So by $\kk$-linear~independence of the $y_j^2$ it suffices to show that the $\mu_j$ are not all zero, and this is true~since
\[
0 \neq e_\alpha = e_\alpha^2 = \sum_i \lambda_i^2 x_i^2 + \sum_j \mu_j^2 y_j^2 = \sum_j \mu_j^2 y_j^2.
\]
\end{Remark}

Taking the diagram \eqref{eqDRdiagram} as given, note that since $\pi$ is surjective we have
\[
\ker \hatCO = \pi\bigl(\ker \bigl(\hatCO \circ \pi\bigr)\bigr) = \pi(\ker (\calD_R \circ \Frob_R)) = \pi(\ker \Frob_R).
\]
The proof of Theorem~\ref{TheoremE} is thus completed by the following result.

\begin{Proposition}
\label{propKerFrob}
We have $\ker \Frob_R \subset \ker \pi$, so $\ker \hatCO = 0$.
\end{Proposition}

\begin{Remark}
If we can construct the diagram and prove this result in the case $\kk = \ZZ/2$, then we can obtain the general case by simply tensoring with the desired $\kk$. \textit{We will therefore assume for the rest of this section that $\kk = \ZZ/2$. In particular, all $($co$)$homology groups implicitly have $\ZZ/2$ coefficients, unless stated otherwise.} With this in place, $\Frob_R$ simplifies to the genuine Frobenius map.
\end{Remark}

\begin{proof}[Proof of Proposition~\ref{propKerFrob}]
Let $\nu_1, \dots, \nu_N$ be the normals to the facets of the moment polytope~$\Delta$ of $X$. We can naturally view them as elements of $\rH_1(T; \ZZ)$, where $T$ is the $n$-torus acting on $X$ as part of its toric structure. For more details on this and the general toric background, see Appendix~\ref{sscToricSetup}. We assume that the facets are ordered so that the last $n$ of them meet at a~vertex of $\Delta$. Smoothness of $X$ then tells us, via the Delzant condition, that $\nu_{N-n+1}, \dots, \nu_N$ is a basis for $\rH_1(T; \ZZ)$.

From Corollary~\ref{corToricQH}, we have identifications
\[
\QH^*_R(X) \cong \kk\bigl[Z_1^{\pm 1}, \dots, Z_N^{\pm 1}\bigr] \bigg/ \biggl(\sum_j \nu_j Z_j\biggr) + \bigl(Z^{2A}-1 \mid A \in \rH_2(X; \ZZ)\bigr)
\]
and
\[
\QH^*(X) \cong \kk\bigl[Z_1^{\pm 1}, \dots, Z_N^{\pm 1}\bigr] \bigg/ \biggl(\sum_j \nu_j Z_j\biggr) + \bigl(Z^A-1 \mid A \in \rH_2(X; \ZZ)\bigr),
\]
with $\pi$ sending $Z_j$ to $Z_j$. Here $Z^A$ \big(similarly $Z^{2A}$\big) denotes the monomial \smash{$\prod_j Z_j^{\langle A, H_j \rangle}$}, and an ideal generated by vectors like $\sum_j \nu_j Z_j$ means the ideal generated by their components with respect to any basis of $\rH_1(T; \ZZ)$. Taking the latter basis to be $\nu_{N-n+1}, \dots, \nu_N$, the components of $\sum_j \nu_j Z_j$ have the form $Z_{N-n+1} - l_1, \dots, Z_N - l_n$ for some linear functions $l_1, \dots, l_n$ of $Z_1, \dots, Z_{N-n}$. We then have
\[
\QH^*_R(X) \cong \kk\bigl[Z_1^{\pm1}, \dots, Z_{N-n}^{\pm 1}\bigr]\bigl[l_1^{-1}, \dots, l_n^{-1}\bigr] \big/ \bigl(\widetilde{Z}^{2A}-1 \mid A \in \rH_2(X; \ZZ)\bigr)
\]
and{\samepage
\[
\QH^*(X) \cong \kk\bigl[Z_1^{\pm1}, \dots, Z_{N-n}^{\pm 1}\bigr]\bigl[l_1^{-1}, \dots, l_n^{-1}\bigr] \big/ \bigl(\widetilde{Z}^A-1 \mid A \in \rH_2(X; \ZZ)\bigr),
\]
where $\widetilde{Z}^A$ \big(similarly $\widetilde{Z}^{2A}$\big) denotes \smash{$\prod_{j=1}^{N-n} Z_j^{\langle A, H_j \rangle}\prod_{j=1}^n l_j^{\langle A, H_{N-n+j} \rangle}$}.}

Given this setup, to prove that $\ker \Frob_R$ is contained in $\ker \pi$ it suffices to prove that the lift of $\ker \Frob_R$ from $\QH^*_R(X)$ to \smash{$\kk\bigl[Z_1^{\pm1}, \dots, Z_{N-n}^{\pm 1}\bigr]\bigl[l_1^{-1}, \dots, l_n^{-1}\bigr]$} is contained in the corresponding lift of $\ker \pi$. We denote these lifts by $I$ and $J$ respectively; we want to show that $I \subset J$. Take then a~general element $f$ of \smash{$\kk\bigl[Z_1^{\pm1}, \dots, Z_{N-n}^{\pm 1}\bigr]\bigl[l_1^{-1}, \dots, l_n^{-1}\bigr]$} and suppose it lies in $I$, which means that
\begin{equation}
\label{eqInI}
f^2 \in \bigl(\widetilde{Z}^{2A}-1 \mid A \in \rH_2(X; \ZZ)\bigr) \subset \kk\bigl[Z_1^{\pm1}, \dots, Z_{N-n}^{\pm 1}\bigr]\bigl[l_1^{-1}, \dots, l_n^{-1}\bigr].
\end{equation}
We want to show that $f$ also lies in $J$, i.e., that
\begin{equation}
\label{eqInJ}
f \in \bigl(\widetilde{Z}^A-1 \mid A \in \rH_2(X; \ZZ)\bigr) \subset \kk\bigl[Z_1^{\pm1}, \dots, Z_{N-n}^{\pm 1}\bigr]\bigl[l_1^{-1}, \dots, l_n^{-1}\bigr].
\end{equation}

From \eqref{eqInI}, we get that there exist
\[
g_1, \dots, g_m \in \kk\bigl[Z_1^{\pm1}, \dots, Z_{N-n}^{\pm 1}\bigr]\bigl[l_1^{-1}, \dots, l_n^{-1}\bigr] \qquad \text{and} \qquad A_1, \dots, A_m \in \rH_2(X; \ZZ)
\]
such that
\[
f^2 = \sum_{j=1}^m g_j \cdot \bigl(\widetilde{Z}^{2A_j} - 1\bigr).
\]
Multiplying through by a large power of $l_1\dots l_n$, we can ensure that no negative powers of the $l_j$ appear. More precisely, for a sufficiently large positive integer $k$ we have that $F \coloneqq f\cdot (l_1\cdots l_n)^{2k}$, all $G_j \coloneqq g_j \cdot (l_1\cdots l_n)^k$, and all $\bigl(\widetilde{Z}^{2A_j} - 1\bigr)(l_1\cdots l_n)^k$ are Laurent polynomials in $Z_1, \dots, Z_{N-n}$. We also have that
\begin{equation}
\label{eqFG}
F^2 = \sum_{j=1}^m G_j \cdot \bigl(\widetilde{Z}^{2A_j} - 1\bigr)(l_1\cdots l_n)^{2k}.
\end{equation}
Note that (since we are in characteristic $2$) only \emph{even} powers of $Z_1, \dots, Z_{N-n}$ appear in each of the Laurent polynomials $F^2$ and $\bigl(\widetilde{Z}^{2A_j} - 1\bigr)(l_1\cdots l_n)^{2k}$. So any terms in \eqref{eqFG} that arise from monomials in the $G_j$ containing odd powers must cancel each other out. In other words, if we write $G_j^\mathrm{ev}$ for the part of $G_j$ containing only even powers, then we have
\[
F^2 = \sum_{j=1}^m G^\mathrm{ev}_j \cdot \bigl(\widetilde{Z}^{2A_j} - 1\bigr)(l_1\cdots l_n)^{2k}.
\]
All of the expressions $F^2$, $G^\mathrm{ev}_j$, and \smash{$\bigl(\widetilde{Z}^{2A_j} - 1\bigr)(l_1\cdots l_n)^{2k}$} are now in the image of the Frobenius~morphism on \smash{$\kk\bigl[Z_1^{\pm 1}, \dots, Z_{N-n}^{\pm 1}\bigr]$}, which is injective since this algebra is an integral domain. So~we can apply the inverse of Frobenius (i.e., `square root both sides') to obtain
\[
F = \sum_{j=1}^m \sqrt{G^\mathrm{ev}_j} \cdot \bigl(\widetilde{Z}^{A_j} - 1\bigr)(l_1\cdots l_n)^k,
\]
where \smash{$\sqrt{G^\mathrm{ev}_j}$} denotes the Laurent polynomial obtained from \smash{$G^\mathrm{ev}_j$} by halving all of the exponents. Dividing through by $(l_1\cdots l_n)^{2k}$ now gives
\[
f = \sum_{j=1}^m \frac{\sqrt{G^\mathrm{ev}_j}}{(l_1 \dots l_n)^k} \cdot \bigl(\widetilde{Z}^{A_j} - 1\bigr),
\]
which proves \eqref{eqInJ} and hence completes the proof.
\end{proof}

\begin{Remark}
The reverse inclusion, $\ker \pi \subset \ker \Frob_R$, is an easy consequence of the fact that every monomial in $R$ squares to $1$. So we actually have $\ker \Frob_R = \ker \pi$.
\end{Remark}

The remainder of this section is devoted to constructing the commutative diagram \eqref{eqDRdiagram}. Most of the work goes into defining $\calD_R$ and showing that it is an isomorphism, and this is where we begin---first studying the classical case due to Duistermaat \cite{Duistermaat}, next recapping the Seidel and relative Seidel maps, and then combining these to construct $\calD_R$, building on work of Haug~\cite{Haug}, Hyvrier~\cite{Hyvrier}, and Tonkonog \cite{TonkonogCO}. Finally, in Section~\ref{sscCompleteDiagram}, we explain the rest of the diagram.

\subsection{The Duistermaat isomorphism}
\label{sscDuistermaat}

As promised, we start by reviewing the remarkable relationship between the classical cohomology rings of $X$ and $L$ in characteristic $2$, which goes back to Duistermaat \cite{Duistermaat}. As in the proof of Proposition~\ref{propKerFrob}, let $\nu_1, \dots, \nu_N$ be the normals to the facets of the moment polytope $\Delta$ of $X$. Let $D_1, \dots, D_N$ be the toric divisors corresponding to the respective facets, and let $H_1, \dots, H_N$ be their Poincar\'e dual cohomology classes in $\rH^2(X; \ZZ)$. As before, see Appendix~\ref{sscToricSetup} for more details. The following description of $\rH^*(X; \ZZ)$ is well-known but we outline a proof that is tailored towards our later needs.

\begin{Proposition}[{\cite[Theorem 10.8]{Danilov}}]
\label{propToricH}
As a $\ZZ$-algebra $\rH^*(X; \ZZ)$ is generated by the $H_j$ modulo the linear relations $\sum_j \nu_j H_j = 0$ and the Stanley--Reisner relations $($and no other relations are needed$)$. The latter say that for a subset $J \subset \{1, \dots, N\}$, we have
\[
\prod_{j \in J} H_j = 0 \qquad \text{if} \ \bigcap_{j \in J} D_j = \varnothing.
\]
\end{Proposition}

\begin{Remark}
\label{rmkVectorRelations}
The vectors $\nu_j$ live in the rank-$n$ lattice $\rH_1(T; \ZZ)$, and the expression ${\sum_j \nu_j H_j = 0}$ is an equality in that lattice. Concretely, it therefore imposes $n$ linear relations between the~$H_j$. These relations are independent since we may assume, as in the proof of Proposition~\ref{propKerFrob}, that~${\nu_{N-n+1}, \dots, \nu_N}$ form a basis for $\rH_1(T; \ZZ)$.
\end{Remark}

\begin{proof}[Sketch proof of Proposition~\ref{propToricH}]
All (co)homology groups in this proof are with $\ZZ$ coefficients.

As a toric manifold, $X$ comes equipped with an action of an $n$-torus $T$ and a $T$-invariant K\"ahler metric $g$. Let $\mathfrak{t}$ denote the Lie algebra of $T$, and $\mu \colon X \to \mathfrak{t}^\vee$ the moment map for the $T$-action. For generic $\xi$ in $\mathfrak{t}$ the map $f_\xi \coloneqq \langle \mu, \xi \rangle \colon X \to \RR$ is a perfect Morse function, so $\rH^*(X)$ is a~free abelian group. Moreover, the pair $(f_\xi, g)$ is Morse--Smale and the descending manifolds are toric subvarieties, so $\rH^*(X)$ is spanned by the Poincar\'e duals of the toric subvarieties. Since each toric subvariety can be written as a transverse intersection of toric divisors $D_j$, under Poincar\'e duality we deduce that $\rH^*(X)$ is generated as an algebra by the $H_j$. The Stanley--Reisner relations follow immediately from Poincar\'e duality and the intersection product.

Next consider the linear relations between the $H_j$. We wish to show that, as a $\ZZ$-module, $\rH^2(X)$ is generated by the $H_j$ modulo $\sum_j \nu_j H_j = 0$. Equivalently, if $F$ is the free $\ZZ$-module with basis $H_1, \dots, H_N$ and \smash{$\theta \colon \rH^1(T) \to F$} is the $\ZZ$-linear map given by
\[
a \in \rH^1(T) \mapsto \sum_j \langle \nu_j, a \rangle H_j \in F,
\]
then we wish to show that the image of $\theta$ is the set $R$ of relators between the $H_j$, i.e., the kernel of the projection $F \to \rH^2(X)$.

To do this, consider the dual map $\theta^\vee \colon F^\vee \to \rH_1(T)$, which has matrix $(\begin{matrix} \nu_1 \cdots \nu_N \end{matrix})$ with respect to the dual basis $H^\vee_1, \dots, H^\vee_N$ of $F^\vee$. Since (WLOG) $\nu_{N-n+1}, \dots, \nu_N$ form a basis for~$\rH_1(T)$, the map $\theta^\vee$ is surjective and thus induces a splitting $F^\vee = \ker (\theta^\vee) \oplus P$, where $P \subset F^\vee$ is a submodule mapped isomorphically by $\theta^\vee$ to $\rH_1(T)$. Note that $P$ must be a free module of rank~$n$, because~$\rH_1(T)$ is, and hence the complement $\ker (\theta^\vee)$ must be a free module of rank~${N-n}$. Dualising, we obtain a splitting $F = (\ker(\theta^\vee))^\circ \oplus P^\circ$, where ${}^\circ$ denotes annihilator, under which $\theta$ induces an isomorphism $\rH^1(T) \to (\ker \theta^\vee)^\circ$. Moreover, $(\ker(\theta^\vee))^\circ$ and $P^\circ$ are free modules of ranks $n$ and $N-n$ respectively. Our task---proving that $R$ is the image of $\theta$---is thus reduced to showing that $R = (\ker \theta^\vee)^\circ$.

If we can show that $R \subset (\ker(\theta^\vee))^\circ$, then we have
\[
\rH^2(X) = F / R \cong \bigl( \bigl(\ker\bigl(\theta^\vee\bigr)\bigr)^\circ \big/ R \bigr) \oplus P^\circ,
\]
where $P^\circ$ is a free module of rank $N-n$. We saw already that $\rH^2(X)$ is a free module, and a~more careful analysis (e.g.,~as in \cite[Lemma~2.4]{SmithQH}) shows that its rank is $N-n$, so we conclude that $(\ker(\theta^\vee))^\circ \big/ R = 0$. Hence $R = (\ker(\theta^\vee))^\circ$, which is what we want. It therefore suffices to show that $R \subset (\ker(\theta^\vee))^\circ$.

Suppose then that a linear relation $\sum_j r_j H_j = 0$ holds, and that $\sum_j s_j H^\vee_j$ lies in the kernel of $\theta^\vee$, i.e., $\sum_j s_j \nu_j = 0$, for integers $r_j$ and $s_j$. We need to show that $\sum_j r_j s_j = 0$. Let $K$ denote the monotone toric fibre in $X$. For each $j$, there is a \emph{basic disc class} $\beta_j \in \rH_2(X, K)$ which meets $D_j$ once (transversely) and is disjoint from the other $D_i$. Recall from the proof of Proposition~\ref{propKerFrob} that the $\nu_j$ naturally live in $\rH_1(T)$, which is identified with $\rH_1(K)$ since $K$ is a free $T$-orbit. Under this identification, the boundary of $\beta_j$ is exactly $\nu_j \in \rH_1(K)$, so since the $s_j$ satisfy $\sum_j s_j \nu_j = 0$ we have that $\sum_j s_j \beta_j \in \rH_2(X, K)$ lifts to a class $A$ in $\rH_2(X)$. The intersection number $A \cdot D_j = \langle A, H_j \rangle$ is exactly $s_j$, since $\beta_i \cdot D_j = \delta_{ij}$, so we conclude that
\[
\sum_j r_j s_j = \sum_j r_j \langle A, H_j \rangle = \biggl\langle A, \sum_j r_j H_j\biggr\rangle = \langle A, 0 \rangle = 0,
\]
as we wanted. Therefore, $\rH^2(X)$ is indeed generated by the $H_j$ modulo $\sum_j \nu_j H_j = 0$.

We have thus shown that $\rH^*(X)$ is generated by the $H_j$, modulo the linear relations
\[
\sum_j \nu_j H_j = 0
\]
and the Stanley--Reisner relations, plus possibly some extra relations. To show that no other relations are required takes some extra work, but that is not relevant to our later purposes so we refer the interested reader to \cite[Theorem 12.4.4]{CoxLittleSchenck}.
\end{proof}

Now consider the real Lagrangian $L \subset X$. Each toric subvariety $Z \subset X$ intersects $L$ cleanly, and \smash{$\dim Z \cap L = \frac{1}{2} \dim Z$}. Let $d_1, \dots, d_N$ be the intersections of the toric divisors with $L$, and let $h_1, \dots, h_N$ be their Poincar\'e dual cohomology classes in $\rH^1(L)$. We remind the reader that we have now returned to our implicit assumption for this section, that all coefficients are in $\kk = \ZZ/2$ unless stated otherwise. Our next result is again standard, but we explain the proof as it is less well-known.

\begin{Proposition}
\label{propRealToricH}
As a $\ZZ/2$-algebra $\rH^*(L)$ is generated by the $h_j$ modulo the linear relations $\sum_j \nu_j h_j = 0$, again interpreted as a vector equation, and the Stanley--Reisner relations $($and no other relations are needed$)$.
\end{Proposition}
\begin{proof}
Duistermaat \cite[Theorem 3.1]{Duistermaat} showed that if one takes the perfect Morse function $f_\xi$ on $X$ from the proof of Proposition~\ref{propToricH}, then the restriction $f_\xi |_L$ is a perfect Morse function on $L$, as long as mod-$2$ coefficients are used. The critical points of the restricted function are the same as those of the original function, but of half their previous index. We thus have $\dim_\kk \rH^{2j}(X) = \dim_\kk \rH^j(L)$ for each $j$. Moreover, using the restriction of the metric $g$ from Proposition~\ref{propToricH} the descending manifolds on $L$ are the intersections of those on $X$ with $L$. In~particular, $\rH^*(L)$ is spanned by the intersections of toric subvarieties with $L$. These are all transverse intersections of the $d_j$ in $L$, so $\rH^*(L)$ is generated as an algebra by the $h_j$. The Stanley--Reisner relations hold as before.

Now consider the linear relations between the $h_j$. Following an analogous argument to Proposition~\ref{propToricH}, but with $\kk = \ZZ/2$ coefficients instead of $\ZZ$ (which actually makes things much simpler since we are now working over a field), it suffices to show that if $\sum_j r_j h_j = 0$ and $\sum_j s_j \nu_j = 0$ for some $r_j$ and $s_j$ in $\kk$, then $\sum_j r_j s_j = 0$. In the proof of Proposition~\ref{propToricH}, we had to use the fact that \smash{$\rH^2(X; \ZZ) \cong \ZZ^{N-n}$}, and we now need the corresponding fact for \smash{$\rH^1(L)$}, namely $\dim_\kk \rH^1(L) = N-n$. This follows by tensoring $\rH^2(X; \ZZ) \cong \ZZ^{N-n}$ with $\kk$ and applying the equality $\dim_\kk \rH^{2j}(X) = \dim_\kk \rH^j(L)$ established in the previous paragraph.

So suppose then that we have $r_j$ and $s_j$ in $\kk$ such that $\sum_j r_j h_j = 0$ and $\sum_j s_j \nu_j = 0$. We need to show that $\sum_j r_j s_j = 0$. To do this, consider the transverse intersection $K \cap L$. It comprises~$2^n$ points and is a torsor for the $2$-torsion subgroup of the torus $T$, which is naturally identified with $\rH_1(T)$. Given a point $p \in K \cap L$ there is a unique class $\eta_j^p$ in $\rH_1(L, K \cap L)$ which starts at~$p$, crosses once transversely over $d_j$, and then returns to $K \cap L$ without meeting any other $d_i$. This can be viewed as the real part of the basic disc class $\beta_j$. The final endpoint of~$\eta_j^p$ is~${\nu_j + p}$, meaning the action of $\nu_j \in \rH_1(T)$ on $p$. Fix an arbitrary starting point $p \in K \cap L$, and suppose that $s_{j_1}, \dots, s_{j_m} = 1$ whilst all other $s_j$ are $0$. Let $p_1 = p$ and inductively define $p_{j+1} = \nu_j + p_j$. Because $\sum_j s_j \nu_j = 0$ we have that $p_{m+1} = p_1$, so the class \smash{$\eta_{j_1}^{p_1} + \dots + \eta_{j_m}^{p_m}$} has empty boundary in $\rH_0(K \cap L)$ and hence lifts to $a \in \rH_1(L)$. We then have
\[
\sum_j r_j s_j = \sum_j r_j \langle a, h_j \rangle = \biggl\langle a, \sum_j r_j h_j\biggr\rangle = \langle a, 0 \rangle = 0,
\]
which is what we wanted.

So $\rH^*(L)$ is generated by the $h_j$ modulo the claimed linear and Stanley--Reisner relations. To show there are no other relations, and hence complete the proof, it suffices to show that
\[
\dim_\kk \rH^*(L) = \dim_\kk \kk[h_1, \dots, h_N] / (\text{linear and Stanley--Reisner relations}).
\]
This follows by reducing Proposition~\ref{propToricH} mod $2$ and using again the result of the first paragraph that $\dim_\kk \rH^{2j}(X) = \dim_\kk \rH^j(L)$.
\end{proof}

The following is an immediate consequence of Propositions~\ref{propToricH} and~\ref{propRealToricH}.

\begin{Corollary}
\label{corDclass}
The map sending a toric subvariety $Z \subset X$ to $Z \cap L$ induces a well-defined algebra isomorphism
\[
\calD_\mathrm{class}\colon \ \rH_{2*}(X; \ZZ/2) \to \rH_*(L; \ZZ/2),
\]
with respect to the intersection products.
\end{Corollary}

\begin{Remark}
We call the map $\calD$ for `Duistermaat', with subscript ${}_\mathrm{class}$ for `classical', to distinguish it from its quantum counterpart $\calD_R$ that we will define shortly.
\end{Remark}

To construct $\calD_R$, we will mimic the above strategy: give a presentation of $\QH^*_R(X)$ (we have seen this already), show that $\HF^*_R(\bL, \bL)$ has the same dimension, then identify generators of $\HF^*_R(\bL, \bL)$ and show that they satisfy all of the relations satisfied by the corresponding generators of $\QH^*_R(X)$.

\subsection{The Seidel maps}

In order to implement this plan to construct $\calD_R$, we will use the Seidel representation \cite{SeidelRep} and the relative Seidel map introduced by Hu--Lalonde \cite{HuLalonde}. We now recap these, and discuss the extension of the latter to $R$ coefficients.

Recall that the Seidel representation is a map $S \colon \pi_1\Ham(X) \to \QH^*(X)^\times$, defined on a loop~$\gamma$ by counting pseudoholomorphic sections of the $X$-bundle over $S^2$ clutched by $\gamma$. This map is group homomorphism when $\pi_1\Ham(X)$ is equipped with the loop concatenation product; this is proved by gluing fibrations. By the usual Eckmann--Hilton argument, one could equivalently define the product on $\pi_1\Ham(X)$ by pointwise composition.

Recall also that there is a relative Seidel map \smash{$S_{\LL} \colon \pi_0 \calP_L \Ham(X) \to \HF^*\bigl(\LL, \LL\bigr)^\times$}, where
\[
\calP_L \Ham(X) = \{ \text{paths } \alpha\colon \ [0, 1] \to \Ham(X) \mid \alpha(0) = \id_X \text{ and } \alpha(1)(L) = L\},
\]
defined as follows. Let $H$ be the rounded half-strip
\[
H = \{s +  {\rm i}t \mid s \leq 0 \text{ and } t \in [-1, 1]\} \cup \bigl\{s+{\rm i} t \mid s^2 + t^2 \leq 1\bigr\} \subset \CC.
\]
Given $\alpha \in \calP_L \Ham(X)$, let $\Lambda_\alpha$ be the following moving Lagrangian boundary condition for $H$, i.e., the following family of Lagrangian submanifolds in $X$ parametrised by $\partial H$:
\[
\Lambda_\alpha(s \pm {\rm i}) = L \quad \text{for $s \leq 0$} \qquad \text{and} \qquad \Lambda_\alpha\bigl({\rm e}^{\pi {\rm i} (\tau - \frac{1}{2})}\bigr) = \alpha(\tau)(L) \quad \text{for $\tau \in [0, 1]$}.
\]
The element $S_{\LL}(\alpha) \in \HF^*\bigl(\LL, \LL\bigr)$ is defined by counting finite-energy solutions $u \colon H \to X$ of the Floer equation, with $u(z) \in \Lambda_\alpha(z)$ for all $z \in \partial H$, and with output at the obvious strip-like end. Equivalently, this is the continuation element associated to the Hamiltonian isotopy $\alpha$.

Degenerating domains as shown in Figure~\ref{figSeidelHom} yields the twisted-homomorphism property
\begin{equation}
\label{eqTwistedHom}
\mu^2_{\Fuk}\bigl(\alpha_1(1)_*S_{\LL}(\alpha_2), S_{\LL}(\alpha_1)\bigr) = S_{\LL}(\alpha_1 \cdot \alpha_2),
\end{equation}
where the product on $\pi_0\calP_L \Ham(X)$ is defined by
\begin{align*}
\alpha_1 \cdot \alpha_2 &= \alpha_1 \text{ followed by (i.e., concatenated with) } \alpha_1(1) \circ \alpha_2
\\ &= \alpha_1 \circ \alpha_2 \text{ pointwise}
\\ &= \alpha_2 \text{ followed by } \alpha_1 \circ \alpha_2(1).
\end{align*}
\begin{figure}[b]
\centering
\begin{tikzpicture}
\def\a{1.5cm}
\def\b{0.7cm}
\def\c{\b/(1-1/sqrt(2))}
\draw (0, \b) -- ++(\a, 0) arc[start angle=-90, end angle=-45, radius=\c] arc[start angle=135, end angle=90, radius=\c] coordinate(A) -- ++(3*\a, 0) arc[start angle=90, end angle=0, radius=\b] coordinate(C) arc[start angle=0, end angle=-90, radius=\b] -- ++(-3*\a, 0) arc[start angle=90, end angle=270, radius=\b] -- ++(3*\a, 0) arc[start angle=90, end angle=0, radius=\b] coordinate(D) arc[start angle=0, end angle=-90, radius=\b] -- ++(-3*\a, 0) coordinate(B) arc[start angle=-90, end angle=-135, radius=\c] arc[start angle=45, end angle=90, radius=\c] -- ++(-\a, 0) -- cycle;
\draw[dashed] ($(A)+(\a, 0)$) -- ++(0, -2*\b);
\draw[dashed] ($(B)+(\a, 0)$) -- ++(0, 2*\b);
\draw (C) node[anchor=west]{$\alpha_1(1)(\Lambda_{\alpha_2})$};
\draw (D) node[anchor=west]{$\Lambda_{\alpha_1}$};
\end{tikzpicture}
\caption{Degenerating the half strip defining $S_{\LL}(\alpha_1 \cdot \alpha_2)$ into two half strips and a half pair of pants, defining $\mu^2_{\Fuk}\bigl(\alpha_1(1)_*S_{\LL}(\alpha_2), S_{\LL}(\alpha_1)\bigr)$.}\label{figSeidelHom}
\end{figure}
Note that these equalities are in $\pi_0 \calP_L\Ham(X)$ and would only be homotopies rel endpoints if we were working in $\calP_L\Ham(X)$ itself.

If $\LL$ is equipped with a non-trivial local system $\calL$, then to define $S_{\LL}(\alpha)$ we need to choose an~isomorphism from $\alpha(1)_*\calL$ to $\calL$. Such a choice need not exist or be unique. (Outside characteristic~$2$ one also needs to choose an isomorphism from $\alpha(1)_*\mathfrak{s}$ and $\mathfrak{s}$, where $\mathfrak{s}$ is the spin structure on~$L$, but this is irrelevant for us.)

Suppose now that we want to extend $S_{\LL}$ to a map $S_\bL$ which lands in $\HF^*_R(\bL, \bL)^\times$. Since $\HF^*_R(\bL, \bL)$ can be viewed as the Floer cohomology of $L$ equipped with the tautological \mbox{rank-$1$} local system $\calL_R$ over $R$, the previous paragraph tells us that to define $S_\bL(\alpha)$ we need a~choice of isomorphism $\alpha(1)_*\calL_R \to \calL_R$. Recalling that $R = \kk[\rH_1(L; \ZZ/2)]$, we see that such an isomorphism exists if and only if $\alpha(1)$ acts trivially on $\rH_1(L)$ (coefficients implicitly taken to be $\kk = \ZZ/2$). In this case an isomorphism is freely and uniquely determined by an $R$-module isomorphism $(\alpha(1)_*\calL_R)_p \to (\calL_R)_p$ over a single point $p \in L$. A natural way to define such an~\mbox{$R$-module} isomorphism is by parallel transport along a homotopy class $\eta$ of path rel endpoints from $\alpha(1)^{-1}(p)$ to $p$. Given a choice of $\eta$, we denote the resulting relative Seidel element by~$S_\bL(\alpha, \eta)$. The analogue of the twisted homomorphism property \eqref{eqTwistedHom} is now
\[
\mu^2_R(\alpha_1(1)_*S_\bL(\alpha_2, \eta_2), S_\bL(\alpha_1, \eta_1)) = S_\bL\bigl(\alpha_1 \cdot \alpha_2, \bigl(\alpha_2(1)^{-1} \circ \eta_1\bigr) \cdot \eta_2\bigr),
\]
assuming $\eta_1$ and $\eta_2$ are defined using the same point $p$.

Hu--Lalonde \cite[Proposition 3.16]{HuLalonde} show that for a loop $\gamma$ in $\pi_1 \Ham(X) \subset \pi_0 \calP_L \Ham(X)$ we have $\CO^0_{\LL}(S(\gamma)) = S_{\LL}(\gamma)$; in their notation, the map $\CO^0_{\LL}$ is called $\mathscr{A}$. In this situation, $\gamma(1) = \id_X$ fixes the whole of $L$ pointwise, so we can take $\eta$ to be the constant path $\eta_p$ at $p$, for any point $p$. The same proof as for $S_{\LL}$ then goes through to show that
\begin{equation}
\label{eqSLCO}
\hatCO(S(\gamma)) = S_\bL(\gamma, \eta_p).
\end{equation}

\subsection[The isomorphism D\_R]{The isomorphism $\boldsymbol{\calD_R}$}

We are now ready to define $\QH^*_R(X)$ and assemble the above ingredients to construct the $R$-algebra isomorphism
\[
\calD_R\colon \ \QH^*_R(X) \to \HF^*_R(\bL, \bL)
\]
appearing in \eqref{eqDRdiagram}. We do this by extending ideas of Haug \cite{Haug} and Hyvrier \cite{Hyvrier}, who constructed an analogous $\kk$-algebra isomorphism $\calD \colon \QH^*(X) \to \HF^*\bigl(\LL, \LL\bigr)$.

To define $\QH^*_R(X)$, first recall that $R = \kk[\rH_1(L; \ZZ/2)]$ and that we have an isomorphism
\[
\rH_2(X; \ZZ/2) \cong \rH_1(L; \ZZ/2)
\]
from $\calD_\mathrm{class}$. We can therefore view $R$ as $\kk[\rH_2(X; \ZZ/2)]$. We then define $\QH^*_R(X)$ to be the $R$-algebra with underlying $R$-module $R \otimes \rH^*(X)$, and with product defined by counting pseudoholomorphic spheres in the usual way, but with the contribution of spheres in class $A \in \rH_2(X)$ weighted by the monomial $Z^A \in \kk[\rH_2(X; \ZZ/2)] = R$.

\begin{Remark}\label{rmkZz}
We will usually use $Z$ for the formal variable in group algebras built from the homology of $X$, and $z$ in group algebras built from the homology of $L$. So under our twin perspectives of $R$ as both $\kk[\rH_2(X; \ZZ/2)]$ and $\kk[\rH_1(L; \ZZ/2)]$, we could write \smash{$Z^A = z^{\calD_\mathrm{class}(A)}$}.
\end{Remark}

Recall that from Corollary~\ref{corToricQH}, we have an identification of $\kk$-algebras
\begin{equation}
\label{eqQHR}
\QH^*_R(X) \cong \kk\bigl[Z_1^{\pm 1}, \dots, Z_N^{\pm 1}\bigr] \bigg/ \biggl(\sum_j \nu_j Z_j\biggr) + \biggl(\prod_j Z_j^{\langle 2A, H_j \rangle}-1 \mid A \in \rH_2(X; \ZZ)\biggr),
\end{equation}
under which $H_j$ on the left corresponds to $Z_j$ on the right. The proof of Corollary~\ref{corToricQH} actually~shows that this is an identification of $R$-algebras, if \smash{$Z^{A \mod 2} \in R$} is sent to \smash{$\prod_j Z_j^{\langle A, H_j \rangle}$} on the right-hand side. (This is consistent with our earlier definition of \smash{$Z^A$} in the proof of Proposition~\ref{propKerFrob}, and is well-defined---i.e., only depends on $A \mod 2$---because we already quotiented out by even powers on the right.) The desired isomorphism $\calD_R$ is then provided by the following~result.

\begin{Proposition}\
\label{propDR}
\begin{enumerate}[label=$(\roman*)$,ref=(\roman*)]\itemsep=0pt
\item\label{HF1} We have $\dim_\kk \HF^*_R(\bL, \bL) = \dim_\kk \QH^*_R(X)$.
\item\label{HF1b} The algebra $\HF^*_R(\bL, \bL)$ is commutative.
\item\label{HF2} As an $R$-algebra, $\HF^*_R(\bL, \bL)$ is generated by $\PSS(h_1), \dots, \PSS(h_N)$.
\item\label{HF3} Each $\PSS(h_j)$ is invertible.
\item\label{HF4} For each $A \in \rH_2(X; \ZZ)$, we have \smash{$\prod_j \PSS(h_j)^{\langle H_j, A\rangle} = Z^{A \mod 2} \in R$}.
\item\label{HF5} These classes satisfy
\[
\sum_j \nu_j \PSS(h_j) = 0 \qquad \text{and, for all $A \in \rH_2(X; \ZZ)$,} \qquad \prod_j \PSS(h_j)^{\langle H_j, 2A\rangle} = 1.
\]
$($The former represents $n$ linear relations, as in Remark~$\ref{rmkVectorRelations}.)$
\end{enumerate}
Thus, by $\ref{HF1b}$, $\ref{HF2}$, $\ref{HF3}$, and $\ref{HF5}$, there is a surjective $R$-algebra homomorphism
\[
R\bigl[Z_1^{\pm 1}, \dots, Z_N^{\pm 1}\bigr] \bigg/ \bigg(\sum_j \nu_j Z_j\bigg) + \bigg(\prod_j Z_j^{\langle 2A, H_j \rangle}-1\mid A \in \rH_2(X; \ZZ)\bigg) \to \HF^*_R(\bL, \bL),
\]
given by $Z_j \mapsto \PSS(h_j)$. By~$\ref{HF4}$, we can replace the $R$ on the left with $\kk$, if we make the left-hand side into an $R$-algebra by letting $Z^{A\mod 2} \in R$ act as $\prod_j Z_j^{\langle A, H_j \rangle}$. Then~\eqref{eqQHR} tells us that we can replace the domain with $\QH^*_R(X)$, and by $\ref{HF1}$ the resulting map is an isomorphism. This map is our $\calD_R$. By construction, it is the unique $R$-algebra homomorphism $\QH^*_R(X) \to \HF^*_R(\bL, \bL)$ sending each $H_j$ to $\PSS(h_j)$.
\end{Proposition}
\begin{proof}
\ref{HF1} For this, we follow Haug, who proved the analogous result for $\HF^*\bigl(\LL, \LL\bigr)$ \cite[Theorem~A\,(i)]{Haug}. We compute $\HF^*_R(\bL, \bL)$ using a pearl model with a perfect Morse function, as~in Proposition~\ref{propRealToricH}, and with the standard toric complex structure $J_\mathrm{std}$. Transversality can be achieved by appropriate perturbation of the Morse function and metric; see \cite[Section~7.2]{Haug}. The differential counts pearly trajectories, and there is a $\ZZ/2$-action on these by reflecting discs using the antisymplectic involution $\tau \colon X \to X$ given by complex conjugation on homogeneous coordinates. Note that this reflection map $\tau$ fixes $L$ pointwise and is antiholomorphic with respect to $J_\mathrm{std}$. A disc and its reflection under $\tau$ have the same boundary class in $\rH_1(L)$---recall the implicit $\ZZ/2$ coefficients---so they are counted with the same $R$ weight. Since we are working in characteristic $2$, we conclude that all contributions to the pearly differential cancel with their reflections, except those coming from fixed points of the action. These fixed points are precisely trajectories containing no discs (see \cite[Proposition~4.29]{SmithPlatonic}), i.e., Morse trajectories, but these collectively cancel out because the Morse function is perfect. So the differential vanishes, and we get
\[
\dim_\kk \HF^*_R(\bL, \bL) = \dim_\kk R \cdot \dim_\kk \rH^*(L) = \dim_\kk R \cdot \dim_\kk \rH^*(X) = \dim_\kk \QH^*_R(X).
\]

\ref{HF1b} Here we follow \cite[Proposition 4.30]{SmithPlatonic} (see also \cite[Corollary 1.6]{FOOOAntisymplectic} for related results in characteristic $0$), and again we use a pearl model with the standard complex structure. The Floer product counts Y-shaped trajectories, and $\tau$ acts on these by reversing the order of the inputs, so the product is commutative. More precisely, we use different Morse data on each leg of the Y, and this Morse data is also exchanged by the action of $\tau$. But this does not affect the argument at the level of cohomology.

\ref{HF2} This is essentially \cite[Lemma 5.6]{Hyvrier}. Equip $\HF^*_R(\bL, \bL)$ with the filtration whose $p$th filtered piece is spanned by pearl complex generators of Morse index at most $p$. By monotonicity, the Floer product respects this filtration, and the associated graded product is the ordinary cup product. The result then follows from the fact, proved in Proposition~\ref{propRealToricH}, that the $h_j$ generate~$\rH^*(L)$ as a $\kk$-algebra.

\ref{HF3} For each $j$, there is a Hamiltonian $S^1$-action on $X$ that rotates anticlockwise around the~$j$th toric divisor $D_j$. This is the restriction of the Hamiltonian $T$-action to the $1$-parameter subgroup generated by $\nu_j$ in the Lie algebra $\mathfrak{t}$. Let $\gamma_j$ be the corresponding primitive loop in $\Ham(X)$. Doing `half of $\gamma_j$' gives a path in $\Ham(X)$, generated by~\smash{$\frac{1}{2}\nu_j$}, whose endpoints preserve~$L$ setwise. We denote this path by $\alpha_j \in \calP_L \Ham(X)$. Hyvrier \cite[Theorem 1.3]{Hyvrier} showed that~$S_{\LL}(\alpha_j)$ is given by the $\PSS$ image of $h_j$ in $\HF^*\bigl(\LL, \LL\bigr)$. Moreover, if we use a~Floer datum for~$L$ for which the Hamiltonian is small, then this $S_{\LL}(\alpha_j)$ is computed by approximately constant maps $H \to X$ landing approximately in $d_j \subset L$. This makes sense since~$\alpha_j$ fixes $d_j = D_j \cap L$ pointwise, and we can only say `approximately' because of the presence of a~(small) Hamiltonian perturbation in the Floer datum. So, if $\eta_j$ is a constant path at a point in~$d_j$, then we have $S_\bL(\alpha_j, \eta_j) = \PSS(h_j)$ in $\HF^*_R(\bL, \bL)$. Letting \smash{$\alpha^{-1}_j$} denote the pointwise inverse to~$\alpha_j$, the twisted homomorphism property for $S_\bL$ then tells us that
\[
\mu^2_R\bigl(\alpha_j(1)_*S_\bL\bigl(\alpha_j^{-1}, \eta_j\bigr), S_\bL(\alpha_j, \eta_j)\bigr) = S_\bL\bigl(\alpha_j \cdot \alpha^{-1}_j, \eta_j \cdot \eta_j\bigr).
\]
The right-hand side is $S_\bL(\id_L, \eta_j) = 1_\bL$, so we deduce that \smash{$\alpha_j(1)_*S_\bL\bigl(\alpha_j^{-1}, \eta_j\bigr)$} is inverse to $\PSS(h_j)$. In fact, $\alpha_j(1)_*$ acts trivially on $\HF^*_R(\bL, \bL)$ since it fixes each generator $\PSS(h_j)$, so we can write the inverse just as \smash{$S_\bL\bigl(\alpha_j^{-1}, \eta_j\bigr)$}.

\ref{HF4} This is an extension of \cite[Proposition 5.5]{Hyvrier}. Before getting into the argument we modify the paths $\eta_j$ used in the previous part so that we can compose them more easily. Fix once and for all a basepoint $p \in L$ in the complement of the $d_j$. There is a unique homotopy class $\hateta_j$ of path rel endpoints from $\alpha_j(1)^{-1}(p)$ to $p$ which crosses $d_j$ once, transversely, and avoids all other~$d_i$. In the notation of the proof of Proposition~\ref{propRealToricH}, the path $\hateta_j$ represents the relative homology class $\eta_j^q$, where $q = \alpha_j(1)^{-1}(p)$. By `sucking in the ends' of $\hateta_j$ to give a constant path in $d_j$, we see that \smash{$S_\bL(\alpha_j, \hateta_j) = S_\bL(\alpha_j, \eta_j)$}. Similarly, if \smash{$\hateta^{\,-1}_j$} denotes the analogous homotopy class of path from $\alpha_j(1)(p)$ to $p$, then \smash{$S_\bL\bigl(\alpha_j^{-1}, \hateta_j^{\,-1}\bigr) = S_\bL\bigl(\alpha_j^{-1}, \eta_j\bigr)$}.

Turning now to the argument itself, fix $A \in \rH_2(X; \ZZ)$ and pick $j_1, \dots, j_m \in \{1, \dots N\}$ and signs $\eps_1, \dots, \eps_m$ such that \smash{$\prod_j \PSS(h_j)^{\langle H_j, A\rangle} = \PSS(h_{j_1})^{\eps_1} \cdots \PSS(h_{j_m})^{\eps_m}$}, i.e., such that for each~$J$
\[
\sum_{i \mid j_i = J} \eps_i = \langle H_J, A \rangle.
\]
Writing each \smash{$\PSS(h_j)^{\pm 1}$} as \smash{$S_\bL\bigl(\alpha_j^{\pm 1}, \hateta_j^{\,\pm 1}\bigr)$}, and using that each $\alpha_j(1)_*$ acts trivially on $\HF^*_R(\bL, \bL)$, the twisted homomorphism property gives
\begin{equation}
\label{eqProdPSS}
\prod_j \PSS(h_j)^{\langle H_j, A\rangle} = S_\bL\bigl(\alpha_{j_1}^{\eps_1} \cdots \alpha_{j_m}^{\eps_m}, \bigl(\alpha_{j_m}(1)^{-\eps_m} \circ \dots \circ \alpha_{j_2}(1)^{-\eps_2} \circ \hateta_{j_1}^{\,\eps_1}\bigr) \cdots \hateta_{j_m}^{\,\eps_m}\bigr).
\end{equation}
Letting $\tilalpha$ denote the homotopy class of path rel endpoints \smash{$\alpha_{j_1}^{\eps_1} \cdots \alpha_{j_m}^{\eps_m} \in \pi_0\calP_L \Ham(X)$}, we claim~that
\begin{enumerate}[label=(\alph*)]\itemsep=0pt
\item\label{itmalpha} $\tilalpha$ is a closed loop (i.e., $\tilalpha(1) = \id_X$) and is nullhomotopic rel endpoints.
\item\label{itmeta} The path \smash{$\bigl(\alpha_{j_m}(1)^{-\eps_m} \circ \dots \circ \alpha_{j_2}(1)^{-\eps_2} \circ \hateta_{j_1}^{\,\eps_1}\bigr) \cdots \hateta_{j_m}^{\,\eps_m}$} from \smash{$\tilalpha(1)^{-1}(p) = p$} to $p$ in $L$ represents the homology class $\calD_\mathrm{class}(A \mod 2) \in \rH_1(L)$.
\end{enumerate}
Assuming these results, \ref{itmeta} tells us that
\begin{equation}
\label{eqSLintermediate}
S_\bL\bigl(\tilalpha, \bigl(\alpha_{j_m}(1)^{-\eps_m} \circ \dots \circ \alpha_{j_2}(1)^{-\eps_2} \circ \hateta_{j_1}^{\,\eps_1}\bigr) \cdots \hateta_{j_m}^{\,\eps_m}\bigr) = z^{\calD_\mathrm{class}(A \mod 2)} S_\bL(\tilalpha, \eta_p),
\end{equation}
where $\eta_p$ represents the constant path at $p$, as before. By \ref{itmalpha} and \eqref{eqSLCO}, we have $S_\bL(\tilalpha, \eta_p) = 1_\bL$, so using the relationship between $Z$ and $z$ from Remark~\ref{rmkZz}, we obtain
\[
z^{\calD_\mathrm{class}(A \mod 2)} S_\bL(\tilalpha, \eta_p) = Z^{A \mod 2}.
\]
Plugging this into \eqref{eqSLintermediate} and thence into \eqref{eqProdPSS} yields
\[
\prod_j \PSS(h_j)^{\langle H_j, A\rangle} = Z^{A \mod 2},
\]
as wanted. To complete the proof of \ref{HF4}, it therefore remains to prove~\ref{itmalpha} and~\ref{itmeta}.

For \ref{itmalpha}, recall that each $\alpha_j$ is generated by the vector \smash{$\frac{1}{2} \nu_j$} in the Lie algebra of the torus $T \subset \Ham(X)$. We can represent the homotopy class $\tilalpha$ by the pointwise composition of the $\alpha_{j_i}$, which is then generated by the vector
\[
\frac{\eps_1}{2} \nu_{j_1} + \dots + \frac{\eps_m}{2} \nu_{j_m} = \sum_j \langle A, H_j \rangle \frac{1}{2} \nu_j = \frac{1}{2} \biggl\langle A, \sum_j \nu_j H_j \biggr\rangle.
\]
The right-hand side vanishes since $\sum_j \nu_j H_j = 0$, so $\tilalpha$ can be represented by the path in $\Ham(X)$ generated by the zero vector, i.e., the constant path at $\id_X$.

For \ref{itmeta}, first let $A_j = \langle A, H_j \rangle \mod 2 \in \ZZ/2$, or equivalently $A _j = A \cdot D_j \mod 2$. By Corollary~\ref{corDclass}, we then have
\begin{equation}
\label{eqDclassA}
\calD_\mathrm{class}(A \mod 2) \cdot d_j = A_j \qquad \text{for all $j$}.
\end{equation}
Since the $d_j$ form a basis for $\rH_{n-1}(L)$, we conclude that $\calD_\mathrm{class}(A \mod 2)$ is uniquely determined by~\eqref{eqDclassA}. But by construction \smash{$\bigl(\alpha_{j_m}(1)^{-\eps_m} \circ \dots \circ \alpha_{j_2}(1)^{-\eps_2} \circ \hateta_{j_1}^{\,\eps_1}\bigr) \cdots \hateta_{j_m}^{\,\eps_m}$} intersects $d_j$ exactly~$A_j$ times mod $2$. This completes the proof of \ref{itmeta} and hence of~\ref{HF4}.

\ref{HF5} The linear relations follow immediately from the corresponding relations between the $h_j$ (proved in Proposition~\ref{propRealToricH}) and the fact that $\PSS$ is a linear map. The multiplicative relations follow from \ref{HF4} and the fact that every monomial in $R$ squares to $1$.
\end{proof}

\begin{Remark}
\label{rmkToricAlternative}
A possible alternative proof strategy, in the spirit of Haug \cite{Haug}, is to consider the $Y$-shaped trajectories contributing to the Floer product on $\HF^*_R(\bL, \bL)$ and use the action of~$\tau$ to show that they all cancel except those in which the only non-constant disc is at the centre of the $Y$. In each of the latter trajectories, the central disc can be glued to its reflection under~$\tau$ to give a trajectory contributing to the quantum product on $\QH^*_R(X)$, and one can use the action of~$\tau$ again to show that all other contributions to the quantum product cancel. Whilst this approach is more direct than that used above, it requires a more careful understanding of the relationship between $\tau$ and $\calD_\mathrm{class}$, and we were not easily able to verify the necessary transversality conditions on the moduli spaces.
\end{Remark}

\begin{Remark}
\label{rmkWLcancellation}
Mod-$2$ cancellation of discs with their reflections by $\tau$ shows that $W_L(\calL) = 0$, which we claimed at the beginning of Section~\ref{secRealToric}.
\end{Remark}

\subsection{Concluding the construction of the diagram}
\label{sscCompleteDiagram}

Recall from Section~\ref{sscThmDOutline} that the proof of Theorem~\ref{TheoremE} is reduced to constructing the commutative diagram \eqref{eqDRdiagram}, namely
\[
\begin{tikzcd}[column sep=5em, row sep=2.5em]
\QH^*_R(X) \arrow{d}{\pi} \arrow{r}{\Frob_R} & \QH^{2*}_R(X) \arrow{d}{\calD_R}[swap, anchor=center, rotate=-90, yshift=-1ex]{\cong}
\\ \QH^*(X) \arrow{r}{\hatCO} & \HF^*_R(\bL, \bL).
\end{tikzcd}
\]
Recall that $\pi$ is reduction modulo $\m$ and $\Frob_R$ is the $\kk$-linear extension of the Frobenius map, which is simply the Frobenius map itself since we have reduced to the case $\kk = \ZZ/2$.

Having constructed the isomorphism $\calD_R$, it remains to explain why the diagram commutes. Explicitly, for each $x \in \QH^*_R(X)$ we need to show that
\[
\calD_R\bigl(x^2\bigr) = \hatCO(x \mod \m).
\]
Since all maps involved are $\kk$-algebra homomorphisms, and $\QH^*_R(X)$ is generated as a $\kk$-algebra by the $Z^A \in R$ and by the $H_j$, it suffices to check this for $x$ equal to each $Z^A$ and to each $H_j$. For $x = Z^A$, it is immediate, since then $x^2 = 1 = x \mod \m$. We are left to deal with $x = H_j$, for which the problem reduces to showing that \smash{$\calD_R(H_j)^2 = \hatCO(H_j)$}. For this we extend \cite[Theorem 1.13]{TonkonogCO}.

From the proof of Proposition~\ref{propDR}, we have $\calD_R(H_j) = \PSS(h_j) = S_\bL(\alpha_j, \hateta_j)$, so using the twisted homomorphism property we have
\[
\calD_R(H_j)^2 = S_\bL(\alpha_j, \hateta_j)^2 = S_\bL\bigl(\alpha_j \cdot \alpha_j, \bigl(\alpha_j(1)^{-1} \circ \hateta_j\bigr) \cdot \hateta_j\bigr).
\]
The concatenation $\alpha_j \cdot \alpha_j$ is the loop $\gamma_j$ in $\Ham(X)$ that appears in the proof of Proposition~\ref{propDR}\,\ref{HF3}. The loop $\bigl(\alpha_j(1)^{-1} \circ \hateta_j\bigr) \cdot \hateta_j$, meanwhile, represents the zero class in $\rH_1(L)$ since its mod-$2$ intersection number with each $d_i$ is zero. Therefore, $\calD_R(H_j)^2 = S_\bL(\gamma_j, \eta_p)$. McDuff--Tolman \cite[p.~9]{McDuffTolman} showed that $S(\gamma_j) = H_j$, so \eqref{eqSLCO} then gives
\[
\calD_R(H_j)^2 = \hatCO(H_j),
\]
which is what we needed.

\begin{Remark}
The idea of Remark~\ref{rmkToricAlternative} suggests another possible strategy, whereby one takes trajectories contributing to $\hatCO(x)$ and glues them to their reflections under $\tau$ to obtain trajectories contributing to \smash{$\calD_R(x^2)$}. Again, however, it is unclear whether the necessary transversality can be achieved.
\end{Remark}

\section[Understanding H(hat Phi\_*)]{Understanding $\boldmath{\operatorname{H}\bigl(\widehat{\Phi}_*\bigr)}$}
\label{sechatPhi}

In this penultimate section, we prove the two remaining main results from the monotone setting, namely Theorems~\ref{TheoremG} and~\ref{TheoremJ}, which both relate to the map
\[
\rH\bigl(\hatPhi_*\bigr)\colon \ \HH^*\bigl(\CF^*\bigl(\LL, \LL\bigr)\bigr) \to \HH^*\bigl(\CF^*\bigl(\LL, \LL\bigr), \hatB\,\bigr).
\]
Recall that $\calE \in \mf(S, W_L - \lambda)$ is the image of $\LL$ under the localised mirror functor $\LMF[L,\lambda]$, $\hatB$ is the $\m$-adic completion of its endomorphism dg-algebra, and \smash{$\hatPhi$} is the $A_\infty$-algebra homomorphism $\CF^*\bigl(\LL, \LL\bigr) \to \hatB$ induced by $\LMF[L,\lambda]$. The map \smash{$\hatPhi$} makes \smash{$\hatB$} into a \smash{$\CF^*\bigl(\LL, \LL\bigr)$}-bimodule, and can also be viewed as a bimodule map \smash{$\CF^*\bigl(\LL, \LL\bigr) \to \hatB$}. Then \smash{$\rH\bigl(\hatPhi_*\bigr)$} is the map on Hochschild cohomology induced by pushforward of coefficient bimodules.

In Section~\ref{sscTori}, we prove Theorem~\ref{TheoremG}, that \smash{$\rH\bigl(\hatPhi_*\bigr)$} is an isomorphism if $L$ is a torus and its local system $\calL \in \Spec S$ is a critical point of $W_L \in S$. By arguments analogous to those in Section~\ref{sscIndependence}, this is independent of the choice of auxiliary data, so it suffices to prove it for a~single choice and this is what we do. Then in Section~\ref{sscThmI}, we prove Theorem~\ref{TheoremJ}, namely that
\begin{enumerate}\itemsep=0pt
\item If $L$ is simply connected, then
\[
\rH\bigl(\hatTheta\bigr)^{-1} \circ \rH\bigl(\hatPhi_*\bigr)\colon \ \HH^*\bigl(\CF^*\bigl(\LL, \LL\bigr)\bigr) \to \HF^*\bigl(\LL, \LL\bigr)^\op
\]
is projection to length zero.
\item If $L$ is simply connected and weakly exact, then $\rH\bigl(\hatTheta\bigr)^{-1} \circ \rH\bigl(\hatPhi_*\bigr)$ is identified with
\[
\ev_*\colon \ \rH_{-*}\bigl(\Lambda L^{-TL}\bigr) \to \rH_{-*}\bigl(L^{-TL}\bigr) \cong \rH^*(L),
\]
where $\Lambda L$ is the free loop space of $L$ and $\ev \colon \Lambda L \to L$ is the evaluation-at-basepoint map.
\end{enumerate}

\subsection{Tori}\label{sscTori}

Assume throughout this subsection that $L$ is a torus and that $\calL$ is a critical point of $W_L$. It is well-known that $\LL$ is then \emph{wide} \cite[Proposition~3.3.1]{BiranCorneaLagrangianTopology}, meaning that $\HF^*\bigl(\LL, \LL\bigr)$ is additively isomorphic to $\rH^*(L)$. This torus case was studied in \cite{SmithSuperfiltered}, building heavily on \cite{ChoHongLauLMF,ChoHongLauTorus}, and we begin by recalling some results from there. Note that in \cite[Section~6]{SmithSuperfiltered} all objects are decorated with primes to distinguish them from (different, but related!) unprimed versions appearing earlier in that paper.

\begin{Proposition}\label{propPhiqis}
For a suitable choice of auxiliary data, based on a pearl model for a carefully chosen perfect Morse function, we have the following:
\begin{enumerate}[label=$(\roman*)$,ref=(\roman*)]\itemsep=0pt
\item The matrix factorisation $\calE$ is quasi-isomorphic in $\mf(S, W_L - \lambda)$ to an explicit wedge-contraction matrix factorisation $\calE_0$, and conjugating by this quasi-isomorphism induces a dg-algebra quasi-isomorphism $\Psi \colon \calB \to \calB_0$ between their endomorphism algebras {\rm \cite[{\it Proposition} 6.11]{SmithSuperfiltered}}.
\item\label{itmPhiqis} The composition $\Psi \circ \Phi \colon \CF^*\bigl(\LL, \LL\bigr) \to \calB_0$ is a quasi-isomorphism, so $\Phi \colon \CF^*\bigl(\LL, \LL\bigr) \to \calB$ is a quasi-isomorphism {\rm \cite[{\it below Remark} 6.12]{SmithSuperfiltered}}.
\end{enumerate}
\end{Proposition}

\begin{Corollary}
For these auxiliary data, the map $\hatPhi \colon \CF^*\bigl(\LL, \LL\bigr) \to \hatB$ is a quasi-isomorphism.
\end{Corollary}
\begin{proof}
The map $\hatPhi$ can be viewed as the composition of $\Phi$ with $\hatS \otimes_S {-} \colon \calB \to \hatB$. The former is a quasi-isomorphism by Proposition~\ref{propPhiqis}\,\ref{itmPhiqis}, whilst the latter is a quasi-isomorphism because for any Noetherian ring $R$ and ideal $I \subset R$ the $I$-adic completion of $R$ is flat over $R$ \cite[Theorem~7.2\,(b)]{Eisenbud}.
\end{proof}

It is now straightforward to deduce the result we want.

\begin{Proposition}
\label{propThmF}
Using the same auxiliary data, the map
\[
\rH\bigl(\hatPhi_*\bigr)\colon \ \HH^*\bigl(\CF^*\bigl(\LL, \LL\bigr)\bigr) \to \HH^*\bigl(\CF^*\bigl(\LL, \LL\bigr), \hatB\,\bigr)
\]
is an isomorphism.
\end{Proposition}
\begin{proof}
Equip both Hochschild complexes with the length filtration. The map $\hatPhi_*$ respects these filtrations so induces a map between spectral sequences. On $E_1$ pages this looks like
\[
\rH\bigl(\hatPhi\bigr)_*\colon \ \rCC^*\bigl(\HF^*\bigl(\LL, \LL\bigr)\bigr) \to \rCC^*\bigl(\HF^*\bigl(\LL, \LL\bigr), \rH\bigl(\hatB\,\bigr)\bigr),
\]
which is an isomorphism because $\rH\bigl(\hatPhi\bigr)\colon \HF^*\bigl(\LL, \LL\bigr) \to \rH\bigl(\hatB\,\bigr)$ is. The result then follows from the Eilenberg--Moore comparison theorem (Theorem~\ref{thmEMcomparison}).
\end{proof}

\subsection{Simply connected Lagrangians}\label{sscThmI}

For this subsection, assume that $L$ is simply connected. Then $S = \kk$ and $\HF^*_S(\bL, \bL)^\op = \HF^*\bigl(\LL, \LL\bigr)$, so \smash{$\rH\bigl(\hatTheta\bigr)^{-1} \circ \rH\bigl(\hatPhi_*\bigr)$} is a map
\[
\HH^*\bigl(\CF^*\bigl(\LL, \LL\bigr)\bigr) \to \HF^*\bigl(\LL, \LL\bigr)^\op.
\]
Theorem~\ref{TheoremJ}\,(i) is the following result.

\begin{Proposition}
\label{propThmIa}
In this situation, $\rH\bigl(\hatTheta\bigr)^{-1} \circ \rH\bigl(\hatPhi_*\bigr)$ is projection to length zero.
\end{Proposition}
\begin{proof}
First we describe \smash{$\rH\bigl(\hatTheta\bigr)^{-1}$} more explicitly. Since $S = \kk$, we have that
\[
\hatB = \eend_\kk^*\bigl(\CF^*\bigl(\LL, \LL\bigr)\bigr)
\]
and that \smash{$\hatTheta$} reduces to the map
\[
\theta\colon \ \CF^*\bigl(\LL, \LL\bigr)^\op \to \rCC^*\bigl(\CF^*\bigl(\LL, \LL\bigr), \eend_\kk^*\bigl(\CF^*\bigl(\LL, \LL\bigr)\bigr)\bigr)
\]
given by \eqref{eqtheta}. As explained in the proof of Proposition~\ref{propThmC}, a quasi-inverse to this map is given~by
\[
\Pi'\colon \ \phi \mapsto \phi^0(e_{\LL}),
\]
where $e_{\LL}$ is a chain-level representative for the unit in $\HF^*\bigl(\LL, \LL\bigr)$. (In Proposition~\ref{propThmC}, we used $\Pi$, which differs from $\Pi'$ by a sign \smash{$(-1)^{\lvert \phi \rvert}$}, but that was because there we were viewing its codomain as an $A_\infty$-module rather than an $A_\infty$-algebra; see Remark~\ref{rmkPiSign}.)

Next recall that the pushforward map \smash{$\hatPhi_*\colon \rCC^*\bigl(\CF^*\bigl(\LL, \LL\bigr)\bigr) \to \rCC^*\bigl(\CF^*\bigl(\LL, \LL\bigr), \hatB\,\bigr)$} is given by
\[
\phi \mapsto \sum(-1)^{(\lvert \phi \rvert - 1)\maltese_i} \hatPhi(\dots, \phi(\dots), \stackrel{i}{\dots}).
\]

Combining these, we see that \smash{$\rH\bigl(\hatTheta\bigr)^{-1} \circ \rH\bigl(\hatPhi_*\bigr) = \rH\bigl(\Pi' \circ \hatPhi_*\bigr)$} and that \smash{$\Pi' \circ \hatPhi_*$} is given by
\[
\phi \mapsto \Pi'\Bigl(\sum(-1)^{(\lvert \phi \rvert - 1)\maltese_i} \hatPhi(\dots, \phi(\dots), \stackrel{i}{\dots})\Bigr) = \hatPhi^1\bigl(\phi^0\bigr)(e_{\LL}) = \mu^2\bigl(\phi^0, e_{\LL}\bigr).
\]
The last equality follows from the construction of $\Phi^1 = \LMF^1$ in Definition~\ref{defLMF}. In cohomology $\mu^2\bigl(\phi^0, e_{\LL}\bigr)$ coincides with $\phi^0$, so \smash{$\rH(\Pi' \circ \hatPhi_*)$} coincides with \smash{$\phi \mapsto \phi^0$}, i.e., projection to length zero.
\end{proof}

Now assume that $L$ is also weakly exact ($\omega$ vanishes on $\pi_2(X, L)$). Then Floer cochain algebras reduce to singular cochain algebras, so \smash{$\rH\bigl(\hatTheta\bigr)^{-1} \circ \rH\bigl(\hatPhi_*\bigr)$} reduces to a map
\[
\HH^*(\rC^*(L)) \to \rH^*(L)^\op = \rH^*(L).
\]
Since $L$ is simply connected, we also have an isomorphism $F\colon \rH_{-*}\bigl(\Lambda L^{-TL}\bigr) \to \HH(\rC^*(L))$ \cite[Corollary~11]{CohenJones}. Theorem~\ref{TheoremJ}\,(ii) is the following result.

\begin{Proposition}
\label{propThmIb}
In this situation, $\rH\bigl(\hatTheta\bigr)^{-1} \circ \rH\bigl(\hatPhi_*\bigr) \circ F$ coincides with
\[
\ev_*\colon \ \rH_{-*}\bigl(\Lambda L^{-TL}\bigr) \to \rH_{-*}\bigl(L^{-TL}\bigr) \cong \rH^*(L),
\]
where $\ev \colon \Lambda L \to L$ is the evaluation-at-basepoint map.
\end{Proposition}
\begin{proof}
We begin by recapping the construction of the isomorphism
\[
F\colon \ \rH_{-*}\bigl(\Lambda L^{-TL}\bigr) \to \HH^*(\rC^*(L)),
\]
following \cite{CohenJones}. Letting \smash{$\Delta^k = \bigl\{(t_1, \dots, t_k) \in \RR^{k+1} \mid 0 \leq t_1 \leq \dots \leq t_k \leq 1\bigr\}$} be the standard $k$-simplex, there is a map \smash{$\Delta^k \times \Lambda L \to L^{k+1}$} given by
\[
((t_1, \dots, t_k), \gamma ) \mapsto (\ev(\gamma) = \gamma(0), \gamma(t_1), \dots, \gamma(t_k)).
\]
This induces a map between Thom spectra
\[
f_k\colon \ \bigl(\Delta^k\bigr)_+ \wedge \Lambda L^{-TL} \to L^{-TL} \wedge \bigl(L^k\bigr)_+,
\]
where the virtual bundle $-TL$ on $\Lambda L$ implicitly means $\ev^* (-TL)$. The corresponding map on chains
\begin{align*}
\rC_{*-k}\bigl(\Lambda L^{-TL}\bigr) \to \rC_*\bigl(L^{-TL}\bigr) \otimes \rC_*(L)^{\otimes k} &\simeq \hom^*\bigl(\rC^{-*}(L)^{\otimes k}, \rC_*\bigl(L^{-TL}\bigr)\bigr)
\\ &\simeq \hom^*\bigl(\rC^{-*}(L)^{\otimes k}, \rC^{-*}(L)\bigr)
\end{align*}
simplifies to a chain map
\[
(f_k)_*\colon \ \rC_{-*}\bigl(\Lambda L^{-TL}\bigr) \to \hom^*\bigl(\rC^*(L)[1]^{\otimes k}, \rC^*(L)\bigr),
\]
and by \cite[Corollary 11]{CohenJones} these $(f_k)_*$ assemble to define an isomorphism
\[
F\colon \ \rH_*\bigl(\Lambda L^{-TL}\bigr) \to \HH^*(\rC^*(L)).
\]

Combining this description with Proposition~\ref{propThmIa}, we see that $\rH\bigl(\hatTheta\bigr)^{-1} \circ \rH\bigl(\hatPhi_*\bigr) \circ F$ is $F$ followed by projection to length zero, which is precisely $\rH((f_0)_*)$. By construction, this is $\ev_*$, followed by the isomorphism $\rH_{-*}\bigl(L^{-TL}\bigr) \to \rH^*(L)$.
\end{proof}

\section{The non-monotone case}
\label{secNonMonotone}

In this section only, we allow $X$ and $L$ to be non-monotone, and we work over the Novikov field
\[
\Lambda = \CC \bigl[\hspace{-1mm}\bigl[ T^\RR \bigr]\hspace{-1mm}\bigr] = \Biggl\{ \sum_{j = 1}^\infty a_j T^{s_j} \mid a_j \in \CC \text{ and } s_j \in \RR \text{ with } s_j \to \infty\Biggr\}
\]
instead of $\kk$ as above. Let $\val \colon \Lambda \to \RR \cup \{\infty\}$ denote the valuation
\[
\val\biggl(\sum_j a_j T^{s_j} \biggr) = \min \{s_j \mid a_j \neq 0\},
\]
with the convention that this minimum is $\infty$ if the set is empty, i.e., $\val(0) = \infty$.

In Section~\ref{sscNonMonotoneSetup}, we discuss some generalities of Floer theory in this setting. We then move on to describe how our main results adapt, before applying them to the case of toric fibres.

\subsection{Setup}
\label{sscNonMonotoneSetup}

In non-monotone Floer theory, the counts of pseudoholomorphic curves are generally infinite, so to make sense of them we must work with suitably filtered and completed algebraic objects. In~this section, we will work with a specific notion of filtration, which we now make precise.

\begin{Definition}
\label{defFiltrations}
An abelian group $A$ is \emph{filtered} if it carries a decreasing filtration $F^sA$ by subgroups, indexed by $s \in \RR$, which is exhaustive ($\bigcup_s F^sA = A$) and Hausdorff ($\bigcap_s F^sA = 0$). The abelian group $A$ is \emph{complete} if the natural map \smash{$A \to \varprojlim A/F^sA$} is an isomorphism; this automatically implies Hausdorffness. The prototypical example of a complete filtered object is $\Lambda$, with $F^s\Lambda = \Lambda_{\geq s} \coloneqq \val^{-1}(\RR_{\geq s})$. If $A$ has extra algebraic structure, then we require the filtration to be compatible with this. For example, a \emph{filtered $\Lambda$-algebra} is a $\Lambda$-algebra $A$ in the usual commutative algebra sense, carrying a filtration $F^sA$ of the sort just discussed, such that the multiplication $A \otimes A \to A$ and ring homomorphism $\Lambda \to A$ defining the $\Lambda$-algebra structure are filtered, i.e., $(F^{s_1}A)(F^{s_2}A) \subset F^{s_1+s_2}A$ for all $s_1$ and $s_2$, and $\Lambda_{\geq s}$ maps into $F^sA$ for all $s$. Similarly, if $A$ is augmented, then we require the augmentation $A \to \Lambda$ to be filtered.
\end{Definition}

\begin{Definition}
\label{defcfiltered}
Given a filtered $\Lambda$-algebra $A$, a rank-$1$ local system $\calE$ on $L$ over $A$ is \emph{$c$-filtered} for $c \in \RR_{\geq 0}$ if each fibre is equipped with a filtration such that: each fibre is isomorphic to~$A$ as a~\emph{filtered} $A$-module; for all paths $\gamma$ and for all $s$, the parallel transport ${\calP_\gamma \colon \calE_{\gamma(0)} \to \calE_{\gamma(1)}}$ maps~$F^s\calE_{\gamma(0)}$ into $F^{s-c}\calE_{\gamma(1)}$. Note that the fibre filtrations need not be (and in general cannot~be) chosen consistently in a locally constant way.
\end{Definition}

\begin{Lemma}
In terms of the monodromy representation $\xi \colon \pi_1(L) \to A^\times$, a rank-$1$ local system can be made $c$-filtered if and only if there exists a $c$ such that $\xi$ lands in $A^\times \cap F^{-c}A$.
\end{Lemma}
\begin{proof}
Fix a basepoint $p$ in $L$ and suppose there exists $c$ such that the monodromy representation $\xi \colon \pi_1(L, p) \to A^\times$ lands in $A^\times \cap F^{-c}A$. Now define the fibre filtrations as follows: fix a~filtration $F^s\calE_p$ of $\calE_p$; for each point $q$ in $L$ choose a path $\gamma_q$ from $p$ to $q$, with $\gamma_p$ taken to be the constant path; equip $\calE_q$ with the filtration $F^s\calE_q = \calP_{\gamma_q}(F^s\calE_p)$. We claim this has the desired property.

Well, take a path $\gamma$ from $q$ to $r$. We have $\gamma = \gamma_r \cdot \beta \cdot \gamma_q^{-1}$ for some loop $\beta$ based at $p$, where $\cdot$ denotes concatenation and ${}^{-1}$ denotes reversal. Then $\calP_\gamma = \calP_{\gamma_r} \circ \calP_{\beta} \circ \calP_{\gamma_q}^{-1}$, where $\calP_{\gamma_r}$ and $\calP_{\gamma_q}^{-1}$ preserve the filtration level exactly and $\calP_\beta = \xi$ shifts it down by at most $c$, so we are done.

The converse is clear: if $\calE$ is $c$-filtered, then apply the defining condition to paths from~$p$ to~$p$.
\end{proof}

Technical foundations in pseudoholomorphic curve theory are also required, in order to define compatible virtual fundamental chains on the relevant moduli spaces. Recall that we encapsulate these in Assumption~\ref{MonAss}, which asserts the existence of sufficient machinery to
\begin{enumerate}\itemsep=0pt
\item\label{ass1b} Define the Fukaya category $\Fuk(X)_\lambda$ over $\Lambda$ and prove the operations satisfy the $A_\infty$-relations. We assume the category is strictly unital, but see Remark~\ref{rmkStrictUnit}.
\item\label{ass2b} Construct $\CO_\lambda\colon \QH^*(X) \to \HH^*(\Fuk(X)_\lambda)$ as a unital $\Lambda$-algebra homomorphism. We assume that it can be made to land in strictly unital Hochschild cochains.
\item\label{ass3b} Prove a generation criterion for summands of $\Fuk(X)_\lambda$ based on injectivity of
\[
\CO_{\LL}\colon \ \QH^*(X) \to \HH^*\bigl(\CF^*\bigl(\LL, \LL\bigr)\bigr)
\]
on factors of $\QH^*(X)$.
\end{enumerate}

For the rest of the section, we assume~\ref{ass1b} holds. We will later also impose~\ref{ass2b} and~\ref{ass3b}.

Objects of $\Fuk(X)_\lambda$ are compact, connected Lagrangians $L$ in $X$, equipped with a pin structure and $\ZZ/2$-grading, plus a $c$-filtered rank-$1$ local system over $\Lambda$ (for some $c$), and a weak bounding cochain $b \in F^{>c}\CF^\mathrm{odd}\bigl(\LL, \LL\bigr)$ whose curvature is $\lambda \in \Lambda$. The $A_\infty$-operations on $\Fuk(X)_\lambda$ count pseudoholomorphic discs $u$ as before, but now additionally weighted by $T^{\omega(u)}$. Each morphism space is a complete filtered $\Lambda$-vector space, so by Gromov compactness these weighted disc counts converge.

\begin{Remark}
The reader may be surprised that in Definition~\ref{defcfiltered} we allowed the parallel transport maps to decrease the filtration level, i.e., we allowed $c$ to be positive. This flexibility is necessary for our applications, and does not cause any convergence problems for the following reason. When defining the contribution of a Floer cochain $c$ to an operation $\mu^k(a_k, \dots, a_1)$, assuming first that there are no weak bounding cochains, we count discs with inputs $a_i$ and output $c$, weighted by $T^\text{area}$ and by $k+1$ parallel transport maps around the boundary. By Definition~\ref{defcfiltered}, there is an a priori lower bound on the combined filtration shifts of the $k+1$ parallel transport maps, so Gromov compactness ensures that below any given filtration level the count of such discs is finite. Suppose we now also incorporate weak bounding cochains on our Lagrangians. By definition, if the local system on a Lagrangian $L$ is $c$-filtered, then the weak bounding cochain $b$ on $L$ must lie in $F^{>c}\CF^*(L, L)$, i.e., in $F^{c+\delta}\CF^*(L, L)$ for some $\delta > 0$. Then each insertion of $b$ introduces a filtration shift of $\geq -c$ from the additional parallel transport map and $\geq c+\delta$ from $b$ itself. So each time we insert $b$ we increase the filtration level of the output by $\geq \delta$, and thus convergence is maintained.
\end{Remark}

Using exactly the same technical foundations, one can similarly define the Fukaya category~$\Fuk_R(X)_r$ over any complete filtered $\Lambda$-algebra $R$. Objects now carry $c$-filtered rank-$1$ local systems over $R$, plus weak bounding cochains with curvature $r \in R$. Objects of $\Fuk(X)_\lambda$ can be converted into objects of $\Fuk_R(X)_\lambda$ by applying $R \otimes_\Lambda {-}$ to their local systems and bounding~cochains.

\subsection{Adapting our main results}

We now explain how to prove Theorems~\ref{TheoremAp}, \ref{TheoremBp} and~\ref{TheoremCp}, in the form of Propositions~\ref{propThmAp}, \ref{propThmBp} and \ref{propThmCp}. Using Assumption~\ref{MonAss}, these are essentially straightforward adaptations of what we have done already.

The starting point is an object \smash{$\bigl(\LL, b\bigr) \in \Fuk(X)_\lambda$}, a complete filtered augmented $\Lambda$-algebra~$R$, and a lift $(\bL, \mathbf{b})$ of \smash{$\bigl(\LL, b\bigr)$} to $R$. Recall that the latter is an object in $\Fuk_R(X)_{W_L}$, for some $W_L \in R$, whose reduction modulo the augmentation ideal $\m$ is $\bigl(\LL, b\bigr)$. Since the localised mirror functor is defined purely in terms of the $A_\infty$-operations on $\Fuk_R(X)_\bullet$ (it involves both~$\Fuk_R(X)_\lambda$ and~$\Fuk_R(X)_{W_L}$), it adapts immediately to the non-monotone setting to give $\LMF[(\bL, \mathbf{b}), \lambda] \colon \Fuk(X)_\lambda \to \mf(R, W_L - \lambda)$. This gives us a mixed Hochschild complex $\rCC^*(\Fuk(X)_\lambda, \mf(R, W_L - \lambda))$ as before, and we have the following.

\begin{Proposition}
\label{propThmAp}
There is an $A_\infty$-algebra homomorphism
\[
\Theta\colon \ \CF^*_R((\bL, \mathbf{b}), (\bL, \mathbf{b}))^\op \to \rCC^*(\Fuk(X)_\lambda, \mf(R, W_L - \lambda)),
\]
which is cohomologically unital and extends the module action of
\[
\CF^*_R((\bL, \mathbf{b}), (\bL, \mathbf{b}))^\op \qquad \text{on} \ \LMF[(\bL, \mathbf{b}), \lambda].
\]
\end{Proposition}
\begin{proof}
The map is defined by the same schematic diagrams and sign twist as in Definition~\ref{defTheta}. The pseudoholomorphic curves being counted are of the same form as those defining the $A_\infty$-operations, and the fact that $\Theta$ is an $A_\infty$-algebra homomorphism is proved by the same form of argument as the $A_\infty$-relations on the category (as in Proposition~\ref{propThetaHom}), so no foundations beyond Assumption~\ref{MonAss}\,\ref{ass1b} are needed. Compared with the monotone case, there is one new type of degeneration that can occur, namely formation of a boundary bubble carrying no marked points. This corresponds to bubbling off of $\mu^0$ terms, which are all multiples of the unit by the assumption that the Lagrangians carry weak bounding cochains. So these degenerations cancel out by strict unitality of the $A_\infty$-operations.

The fact that $\Theta$ extends the module action follows directly from its definition, as in Lem\-ma~\ref{lemThetaModuleAction}. Cohomological unitality is a consequence of Proposition~\ref{propThmBp} if one is willing to make Assumption~\ref{MonAss}\,\ref{ass2b}, or of the argument suggested in Remark~\ref{rmkRingUnitality} if not.
\end{proof}

\begin{Proposition}
\label{propThmBp}
If Assumption~$\ref{MonAss}\,\ref{ass2}$ also holds, then
\[
\rH(\Theta) \circ \CO^0_{(\bL, \mathbf{b})} = \rH\bigl(\bigl(\LMF[(\bL, \mathbf{b}),\lambda]\bigr)_*\bigr) \circ \CO_\lambda.
\]
\end{Proposition}
\begin{proof}
This follows from the same form of argument as Corollary~\ref{corThmBCommutes}, which only uses moduli spaces analogous to those used in defining $\CO_\lambda$ and in proving that it lands in Hochschild cocycles, so no foundations beyond Assumption~\ref{MonAss}\,\ref{ass1} and \ref{ass2} are needed. Again there are additional degenerations coming from bubbling off of $\mu^0$, but these cancel since we are assuming~$\CO_\lambda$ lands in strictly unital Hochschild cochains.
\end{proof}

\begin{Proposition}
\label{propThmCp}
If $R$ is Noetherian, then the $\m$-adically completed map $\hatTheta$ is a quasi-iso-morphism.
\end{Proposition}
\begin{proof}
The proof of Theorem~\ref{TheoremC} in Section~\ref{secThmC} is essentially purely algebraic, and translates directly to our new setting (working over $\Lambda$ now, in place of $\kk$, of course). The argument relies on $R$ being Noetherian, which was automatic before, but now we have imposed it directly.
\end{proof}

Before moving on to study the special case of toric fibres, we prove a general result that establishes, under natural conditions, various properties of $R$ that are useful for applying the above machinery.

\begin{Lemma}
\label{lemFinDimProperties}
Suppose $R$ is a local $\Lambda[\rH_1(L; \ZZ)]$-algebra which is complete, filtered, and finite-dimensional as a $\Lambda$-algebra. The following then hold:
\begin{enumerate}[label=$(\roman*)$,ref=(\roman*)]\itemsep=0pt
\item\label{FDaugmented} The unique augmentation $\eps \colon R \to \Lambda$, given by quotienting by the unique maximal ideal $\m \subset R$, is automatically filtered.
\item\label{FDc} If $\trho \colon \rH_1(L; \ZZ) \to R^\times$ denotes the map sending $\gamma$ to the image of $\tau^\gamma$ under the structure map $\Lambda[\rH_1(L; \ZZ)] \to R$, and if $\eps \circ \trho$ lands in $\Lambda_{\geq 0}$, then there exists $c \geq 0$ such that $\trho$ lands in $F^{-c}R$. $($Note that in the rest of the paper we use $z$ as the monomial recording $\rH_1(L; \ZZ)$ classes. Here we use $\tau$ instead, for consistency with {\rm \cite{SmithQH}}.$)$
\item\label{FDNoetherianComplete} $R$ is Noetherian and $\m$-adically complete.
\end{enumerate}
\end{Lemma}
\begin{proof}
\ref{FDaugmented} Suppose for contradiction that $\eps$ is not filtered. Then there exist $s \in \RR$ and $r \in F^sR$ such that $\eps(r) \notin \Lambda_{\geq s}$. Say $\val \circ \eps(r) = s-\delta$, with $\delta > 0$. Replacing $r$ with $\eps(r)^{-1}r$, we may assume that $r \in F^\delta R$ and that $\eps(r) = 1$. Now consider the $\Lambda$-linear endomorphism of $R$ given by multiplication by $r$, and the associated decomposition of $R$ into generalised eigenspaces. Each generalised eigenspace must be an ideal of $R$, but $R$ is local so any proper ideal must be contained in the unique maximal ideal. This means that the whole of $R$ must be a single generalised eigenspace, so there exists $\lambda \in \Lambda$ and $m = \dim_\Lambda R$ such that $(r-\lambda)^m$ annihilates $R$ and hence $(r-\lambda)^m = 0$. Applying $\eps$ to this expression gives $\lambda = 1$, and expanding it out, we then obtain
\begin{equation}
\label{eq1equals}
1 = - \sum_{j=1}^m \binom{m}{j} (-1)^{m-j} r^j.
\end{equation}
The right-hand side of \eqref{eq1equals} lies in $F^\delta R$, so $1$ also lies in $F^\delta R$. Then for all positive integers $p$ we get $1 = 1^p \in F^{p\delta}R$ and hence
\[
1 \in \bigcap_{p = 1}^\infty F^{p\delta} R = 0,
\]
which is impossible (here we used that the filtration on $R$ is Hausdorff). So $\eps$ is indeed filtered.

\ref{FDc} Fix elements $\gamma_1, \dots, \gamma_{2n}$ which generate $\rH_1(L; \ZZ)$ as a monoid (e.g.,~a basis and their negatives) and let $\trho_j = \trho(\gamma_j)$. It suffices to show that for each $j$ there exists $c_j \geq 0$ such that all positive powers of $\trho_j$ lie in $F^{-c_j}R$, since then $c = c_1 + \dots +c_{2n}$ works. So fix $j$ and consider~${\trho_j \in R^\times}$. As in \ref{FDaugmented}, we can consider the generalised eigenspace decomposition for the action of~$\trho_j$ on $R$, and deduce that $(\trho_j - \eps(\trho_j))^m = 0$ for some $m$, so
\begin{equation}
\label{eqtrho}
\trho_j^{\:m} = - \sum_{j=1}^m \binom{m}{j}(-\eps(\trho_j))^j \trho_j^{\:m-j}.
\end{equation}
By assumption, we have $\eps(\trho_j) \in \Lambda_{\geq 0}$. So if we pick $c_j \geq 0$ such that $\trho_j, \dots, \trho_j^{\:m-1} \in F^{-c_j}R$, then~\eqref{eqtrho} tells us that $\trho_j^{\:m}$ also lies in $F^{-c_j}R$. Repeatedly multiplying \eqref{eqtrho} through by $\trho_j$, we~get by induction that $\trho_j^{\:p} \in F^{-c_j}R$ for all positive integers $p$, which is what we wanted.

\ref{FDNoetherianComplete} The Noetherian condition is obvious from finite-dimensionality. $\m$-adic completeness, meanwhile, follows from the fact that $\m^p = 0$ for some $p$. The latter can be proved using the Artinian property (via \cite[Lemma 10.53.4]{stacks-project}, for example), or directly by picking a basis $r_1, \dots, r_k$ for $\m$ and using generalised eigenspace decompositions as in \ref{FDaugmented} to show that each $r_j$ is nilpotent.
 \end{proof}

\subsection{Split-generation by toric fibres}

In this subsection, we specialise to the case where $X$ is a compact toric manifold equipped with its canonical $\ZZ/2$-grading, and prove Theorem~\ref{TheoremL} which states that its Fukaya category is generated by toric fibres. The precise statement is Proposition~\ref{propThmL} below. During the exposition, we include several lemmas whose proofs we defer to the end of the subsection, to avoid distracting from the main thread of the argument.

First we introduce some notation. Let
\begin{equation}
\label{eqMomentPolytope}
\Delta = \bigl\{x \in \mathfrak{t}^\vee \mid \langle \nu_j, x \rangle \geq - \lambda_j \text{ for } i=1,\dots,N\bigr\}
\end{equation}
be the moment polytope of $X$ as in Appendix~\ref{sscToricSetup}, and let $L$ be an arbitrary toric fibre. By translating $\Delta$, we may assume that the $\lambda_j$ are all positive and that $L = \mu^{-1}(0)$. As discussed in Appendix~\ref{sscToricSetup}, the $\nu_j$ can naturally be viewed as elements of $\rH_1(L; \ZZ)$, where they are the boundaries of the \emph{basic disc classes} $\beta_j$ of area $\lambda_j$.

\begin{Definition}
Let $\Gamma_\RR$ be the submonoid of $\RR \oplus \rH_1(L; \ZZ)$ generated by $(\lambda_1, \nu_1), \dots, (\lambda_N, \nu_N)$ and $\RR_{\geq 0} \oplus 0$. Its monoid ring $\CC[\Gamma_\RR]$ naturally lies inside $A \coloneqq \CC\bigl[T^\RR\bigr][\rH_1(L; \ZZ)]$, with $(s, \gamma) \in \Gamma_\RR$ represented by the monomial $T^s \tau^\gamma$. Now define a filtration on $A$ by $F^sA = T^s\CC[\Gamma_\RR]$. This is exhaustive since the $\nu_i$ span $\rH_1(L; \ZZ)$ and is Hausdorff since all exponents of $T$ appearing in~$F^sA$ are at least $s$. Let $\overline{A}$ be the completion of $A$ with respect to this filtration. This naturally contains $\Lambda$ as the completion of $\CC\bigl[T^\RR\bigr]$, so is a complete filtered $\Lambda$-algebra.
\end{Definition}

For a specific closed ideal $I \subset \overline{A}$, Fukaya--Oh--Ohta--Ono \cite{FOOOtoricMS} construct an isomorphism of filtered $\Lambda$-algebras $\ks \colon \QH^*(X) \to \overline{A}/I$ which they call the Kodaira--Spencer map. Their filtration is defined slightly differently from ours, but a proof using the present conventions is given in \cite{SmithQH} (the fact that both $\ks$ and its inverse respect the filtration follows from the fact that the domain and codomain can both be defined over $\Lambda_0$, where the filtrations become $F^s{-} = T^s{-}$, and filteredness then follows from $\Lambda_{\geq 0}$-linearity). The inclusions of $\Lambda$ and $\CC[\rH_1(L; \ZZ)]$ into $\overline{A}$ make $\overline{A}/I$, and hence $\QH^*(X)$, into a $\Lambda[\rH_1(L; \ZZ)]$-algebra. Even though $\overline{A}$ is strictly bigger than $\Lambda[\rH_1(L; \ZZ)]$, the next result says that this difference disappears after quotienting by $I$.

\begin{Lemma}
\label{lemQHsurj}
The structure map $\sigma \colon \Lambda[\rH_1(L; \ZZ)] \to \QH^*(X)$ is surjective.
\end{Lemma}

\begin{Remark}
There is an obvious filtration of $\Lambda[\rH_1(L; \ZZ)]$ by $\Lambda_{\geq s}[\rH_1(L; \ZZ)]$, but $\QH^*(X)$ is \emph{not} generally a filtered $\Lambda[\rH_1(L; \ZZ)]$-algebra with respect to this filtration, i.e., $\sigma$ is not in general filtered.
\end{Remark}

Thus $\Spec \QH^*(X)$ is a zero-dimensional closed subscheme of $\Spec \Lambda[\rH_1(L; \ZZ)]$, and hence $\QH^*(X)$ is a product of local rings, each corresponding to a single closed point of
\[
\Spec \Lambda[\rH_1(L; \ZZ)].
\]
Since $\Lambda$ is algebraically closed \cite[Lemma A.1]{FOOOToricI}, each of these closed points corresponds to a~homomorphism $\rH_1(L; \ZZ) \to \Lambda^\times$. Let $R$ be one of the local factors of $\QH^*(X)$, equipped with the filtration induced by the projection $\pi \colon \QH^*(X) \to R$. Note that $R$ is finite-dimensional over $\Lambda$, and is a $\Lambda[\rH_1(L; \ZZ)]$-algebra via $\pi \circ \sigma$. To start putting us in the setup of the previous subsection, we have the following.

\begin{Lemma}
\label{lemRproperties}
The filtration on $R$ satisfies Definition~$\ref{defFiltrations}$ and is complete.
\end{Lemma}

We are thus in the situation of Lemma~\ref{lemFinDimProperties}, so by part \ref{FDaugmented} of that result $R$ is a complete filtered augmented $\Lambda$-algebra. The augmentation $\eps \colon R \to \Lambda$ is unique, with kernel given by the unique maximal ideal $\m$ in $R$. Let $\trho \colon \rH_1(L; \ZZ) \to R^\times$ be the homomorphism $\gamma \mapsto \pi \circ \sigma(\tau^\gamma)$, and let $\rho = \eps \circ \trho$. Note that $\rho$ is the homomorphism $\rH_1(L; \ZZ) \to \Lambda^\times$ corresponding to $\m$ as a closed point of $\Spec \Lambda[\rH_1(L; \ZZ)]$.

The next result is the key step in defining our objects $\LL$ and $\bL$.

\begin{Lemma}[{morally equivalent to \cite[Lemma~2.2.2]{FOOOtoricMS}}]
\label{lemInteriorPoint}
By translating $\Delta$ if necessary, i.e., by changing which toric fibre $L$ is, we may assume that $\rho$ lands in $\Lambda_{\geq 0}$.
\end{Lemma}

Assume now that this translation has been done. Then $\rho$ defines (the monodromy of) a~$0$-filtered rank-$1$ local system $\calL$ on $L$ over $\Lambda$, and we take $\LL$ to be $L$ equipped with this local system, the standard spin structure (as defined in Appendix~\ref{sscToricSetup}), and an arbitrary $\ZZ/2$-grading (equivalent to an orientation). We use Fukaya--Oh--Ohta--Ono's de Rham model and torus-equivariant perturbation data from \cite{FOOOToricI}, and take the weak bounding cochain $b$ to be zero. This is indeed a weak bounding cochain: in \cite[Section~12]{FOOOToricI}, they show that for any rank-$1$ local system~$\calL_0$ with monodromy $\rho_0 \colon \rH_1(L; \ZZ) \to \CC^\times$, every $b_0 \in \rH^1(L; \Lambda_{>0})$ is a weak bounding cochain, and the same argument applies with our $\calL$ and $\rho$ in place of $\calL_0$ and $\rho_0$. Let $\lambda$ be the curvature of~$\bigl(\LL, b = 0\bigr)$.

\begin{Remark}
By \cite[Lemma 11.8]{FOOOToricI}, we expect $\bigl(\LL, 0\bigr)$ to be isomorphic to $L$ equipped with~$\calL_0$ and $b_0$ if $\rho = \rho_0 \cdot e^{b_0}$.
\end{Remark}

Similarly, by Lemma~\ref{lemFinDimProperties}\,\ref{FDc} $\trho$ defines a $c$-filtered rank-$1$ local system on $L$ over $R$, whose reduction modulo $\m$ is $\calL$, and we take $\bL$ to be $L$ equipped with this local system and with the same spin structure and orientation as $\LL$. We use the same de Rham model and perturbation data, and take its weak bounding cochain $\mathbf{b}$ also to be zero. By construction, this $(\bL, 0)$ is a lift of $\bigl(\LL, 0\bigr)$ to $R$.

\begin{Lemma}
The maps
\[
\hatCO\colon \ \QH^*(X) \to R \qquad \text{and} \qquad \pi\colon \ \QH^*(X) \to R
\]
coincide.
\end{Lemma}
\begin{proof}
Since $\QH^*(X)$ is generated as a $\kk$-algebra by the $H_j$, it suffices to show that the maps agree on each $H_j$. By definition, \smash{$\hatCO(H_j) \in R$} counts holomorphic discs $u$ mapping an interior marked point to the divisor $D_j$, weighted by
\[
\trho([\partial u]) = \pi \circ \sigma\bigl(\tau^{[\partial u]}\bigr) = \pi \circ \ks^{-1}\bigl(\tau^{[\partial u]} \mod I\bigr).
\]
Meanwhile, $\ks (H_j) \in \overline{A}/I$ counts the same thing, but weighted by \smash{$\tau^{[\partial u]} \mod I$}. We can choose the same perturbation data for \smash{$\hatCO(H_j)$} as for $\ks(H_j)$, namely the torus-equivariant Kuranishi multisections of \cite[Lemma 6.5]{FOOOToricII}, and we then get
\[
\hatCO(H_j) = \pi \circ \ks^{-1} (\ks(H_j)) = \pi(H_j),
\]
as wanted.
\end{proof}

Since $R$ is Noetherian and $\m$-adically complete by Lemma~\ref{lemFinDimProperties}\,\ref{FDNoetherianComplete}, Corollary~\ref{CorollaryK} gives the following, which is the precise statement of Theorem~\ref{TheoremL}.

\begin{Proposition}
\label{propThmL}
Under Assumption~$\ref{MonAss}$, $\bigl(\LL, 0\bigr)$ split-generates the summand of $\Fuk(X)_\lambda$ corresponding to the factor $R$ of $\QH^*(X)$.
\end{Proposition}

\begin{Remark}
In \cite[Section~4.7]{FOOOtoricMS}, Fukaya--Oh--Ohta--Ono factor their isomorphism $\ks$ through Hochschild cohomology via a closed-open map, denoted by $\widehat{q}$. However, they work over $\Lambda_{\geq 0}$, whereas for the generation criterion one must work over $\Lambda$. For example, their $\widehat{q}$ is injective on the whole of $\QH^*(X; \Lambda_{\geq 0})$, regardless of which toric fibre one is looking at, whilst for split-generation and injectivity of $\CO_{\LL}$ one must choose a specific toric fibre for each factor of $\QH^*(X)$.
\end{Remark}

We wrap up this subsection by filling in the promised proofs of Lemmas~\ref{lemQHsurj}, \ref{lemRproperties} and \ref{lemInteriorPoint}.

\begin{proof}[Proof of Lemma~\ref{lemQHsurj}]
Let $\tau_j = T^{\lambda_j}\tau^{\nu_j}$. It suffices to prove the following claim: $\QH^*(X)$ is spanned over $\Lambda$ by finitely many monomials in the $\sigma(\tau_j)$. To prove this claim, we use \cite[Lemma~2.20]{SmithQH} (corresponding to \cite[Theorem~4.6]{FOOOToricI}), which tells us that $\ks(H_j) \in \tau_j + F^{>0}\overline{A}$ modulo $I$, and hence that $\sigma(\tau_j) \in H_j + F^{>0}\QH^*(X)$, where $H_j$ is the Poincar\'e dual of the $j$th toric divisor. The claim then follows from the fact that $\QH^*(X; \Lambda_{\geq 0})$ is spanned over $\Lambda_{\geq 0}$ by finitely many monomials in the $H_j$.
\end{proof}

\begin{proof}[Proof of Lemma~\ref{lemRproperties}]
The filtration on $R$ is automatically decreasing and exhaustive, so we just need to show that it is complete (since this also implies Hausdorffness). For this, note that if $e_1, \dots, e_k$ is a $\Lambda_{\geq 0}$-basis for $\QH^*(X; \Lambda_{\geq 0})$, then the filtration on $\QH^*(X)$ is defined by~$F^s\QH^*(X) = \bigoplus_{j=1}^k \Lambda_{\geq s} e_j$. Note further that we can always pick the basis $e_1, \dots, e_k$ such that the first $l$ elements form a $\Lambda_{\geq 0}$-basis for $(\ker \pi) \cap \QH^*(X; \Lambda_{\geq 0})$. Having done this, the images $\pi(e_{l+1}), \dots, \pi(e_k)$ of $e_{l+1}, \dots, e_k$ in $R$ form a $\Lambda$-basis for $R$, and the induced filtration on $R$ is~\smash{$F^sR = \bigoplus_{j=l+1}^k \Lambda_{\geq s} \pi(e_j)$}. Thus $R$ is isomorphic to $\Lambda^{k-l}$ as a filtered $\Lambda$-module, and hence is~complete.
\end{proof}

\begin{proof}[Proof of Lemma~\ref{lemInteriorPoint}]
We need to show that by translating $\Delta$ we can ensure that
\[
\val \circ \rho\colon \ \rH_1(L; \ZZ) \to \RR
\]
is non-negative, which is equivalent to it being identically zero.

To begin, note that $\val$ is naturally an element of the space $\mathfrak{t}^\vee$ in which $\Delta$ lives \big(in fact, $\mathfrak{t}^\vee$ is naturally identified with $\rH^1(L; \RR)$\big). Translating $\Delta$ by $-a \in \mathfrak{t}^\vee$ corresponds to changing $x$ to $x+a$ in \eqref{eqMomentPolytope}, or equivalently to changing $\lambda_j$ to $\lambda_j + \langle \nu_j, a \rangle$. The allowed translations are those that keep $0$ in the interior of the polytope, i.e., those satisfying $\langle \nu_j, a \rangle > -\lambda_j$ for all $j$. The effect of this translation on $\ks$ is to multiply $\ks(H_j)$ by $T^{\langle \nu_j, a \rangle}$. The effect on $\sigma$ is therefore to multiply $\sigma(\tau^{\nu_j})$ by $T^{-\langle \nu_j, a \rangle}$. The effect on $\val \circ \rho(\nu_j)$ in turn is thus to subtract $\langle \nu_j, a \rangle$. So to be able to make $\val \circ \rho$ identically zero, we need to be able to choose $a$ such that $\val \circ \rho(\nu_j) = \langle \nu_j, a \rangle$ for all $j$, and $\langle \nu_j, a \rangle > -\lambda_j$ for all $j$. In other words, the first part of the lemma reduces to showing that $\val \circ \rho(\nu_j) > -\lambda_j$ for all $j$.

Next note that $\rho(\nu_j)$ can be written as $\eps \circ \pi \circ \sigma(\tau^{\nu_j}) = T^{-\lambda_j} \eps \circ \pi \circ \sigma(\tau_j)$, where $\tau_j$ denotes $T^{\lambda_j}\tau^{\nu_j} \in \Lambda[\rH_1(L; \ZZ)]$. Letting $r_j = \pi \circ \sigma(\tau_j)$ be the image of $\tau_j$ in $R$, our task reduces to showing that $\val \circ \eps(r_j) > 0$ for all $j$. Since $\tau_j$ lies in filtration level $0$ in $\overline{A}$ we have that $\sigma(\tau_j)$ lies in $F^0\QH^*(X)$, and hence that $r_j$ lies in $F^0R$. Since $\eps$ is filtered (Lemma~\ref{lemRproperties}), we immediately get $\val \circ \eps(r_j) \geq 0$ for all $j$. We thus want to show that there are no $j$ with $\val \circ \eps(r_j) = 0$.

Suppose then, for contradiction, that such $j$ do exist: call them $j_1, \dots, j_k$, and let $\eps(r_{j_i}) = a_i + b_i$, where $a_i \in \CC^\times$ and $b_i \in \Lambda_{>0}$. If the toric divisors $D_{j_1}, \dots, D_{j_k}$ have empty intersection, then $H_{j_1} \cdots H_{j_k}$ lies in $F^{>0}\QH^*(X)$. Since $\sigma(\tau_j) \in H_j + F^{>0}\QH^*(X)$, we deduce that $r_{j_1} \cdots r_{j_k} \in F^{>0}R$. Applying $\eps$ we conclude that $a_1 \cdots a_k \in \Lambda_{>0}$, which is nonsense. Therefore, the divisors $D_{j_1}, \dots, D_{j_k}$ have non-empty intersection, and hence the vectors $\nu_{j_1}, \dots, \nu_{j_k}$ are linearly independent. Modulo elements of $F^{>0}\overline{A}$, the ideal $I \subset \overline{A}$ contains the components of~$\sum_j \nu_j \tau_j$. We therefore have $\sum_j \nu_j r_j \in F^{>0}R$ and hence $\sum_j \nu_j \eps(r_j) \in \Lambda_{>0}$. But the constant term in $\sum_j \nu_j \eps(r_j)$ is $\sum_i \nu_{j_i} a_i$, which is non-zero since the $\nu_{j_i}$ are linearly independent. This gives the desired contradiction, and we conclude that there are no $j$ with $\val \circ \eps(r_j) = 0$. Hence $\val \circ \eps(r_j) > 0$ for all $j$ and thus we can indeed translate $\Delta$ to make $\val \circ \rho$ identically zero.
\end{proof}

\begin{Remark}
In the presence of a bulk deformation from $F^{>0}\QH^\mathrm{even}(X)$, the proofs go through unchanged. The results similarly hold in the presence of a B-field $B \in \QH^2(X; \CC)$, but now with one small modification: in the proofs of Lemmas~\ref{lemQHsurj} and~\ref{lemInteriorPoint}, we have $\sigma(\tau_j) = a_jH_j + F^{>0}\QH^*(X)$ for some $a_j \in \CC^\times$ (previously $a_j$ was $1$); but this does not affect the arguments.
\end{Remark}

\appendix
\section{The generation criterion}\label{secGen}

In this appendix, we review the decomposition of the Fukaya category induced by decompositions of quantum cohomology, and derive the corresponding generation criterion. We then look at how this relates to the more standard splittings of the Fukaya category by eigenvalues of the first Chern class. This material is all well-known to experts, but we are not aware of the details being spelled out in the literature, especially Lemma~\ref{lemSubcatsZero}.

We remind the reader that, having finished Section~\ref{secNonMonotone}, we are now back in the monotone setting.

\subsection{Decompositions of the Fukaya category}\label{sscDecomp}

Suppose that $\QH^*(X)$ decomposes into a product of $\ZZ/2$-graded $\kk$-algebras $Q_i$. For each $\lambda$, this induces a splitting of $\Fuk(X)_\lambda$ into pairwise orthogonal subcategories $\Fuk(X)_{\lambda,i}$, and the goal of this subsection is to summarise the construction and prove that it is compatible with the closed-open map. See also \cite[Section~(5b)]{SeidelFlux}.

\begin{Remark}
Similarly, a decomposition of $\QH^*(X)$ into a product of ungraded $\kk$-algebras~$Q_i$ gives rise to a decomposition of the ungraded category $\Fuk(X)_\lambda^\mathrm{un}$ into pairwise orthogonal subcategories $\Fuk(X)_{\lambda, i}^\mathrm{un}$. We will not separately mention this case further as the arguments are identical.
\end{Remark}

The decomposition of $\QH^*(X)$ into the $Q_i$ corresponds to a decomposition of its unit $1_X$ into pairwise orthogonal idempotents $e_i$ of even degree. Roughly speaking, $\Fuk(X)_{\lambda,i}$ is the image of~$\Fuk(X)_\lambda$ under the module action of $e_i$ via $\CO^0$. More precisely, let $\calQ$ denote the $A_\infty$-category $\operatorname{mod} \Fuk(X)_\lambda$ of $A_\infty$-modules over $\Fuk(X)_\lambda$, and let $\operatorname{Tw} \Fuk(X)_\lambda$ be the smallest full subcategory of~$\calQ$ that contains the image of the Yoneda embedding $\Fuk(X)_\lambda \to \calQ$ and is closed under taking shifts and mapping cones. We may implicitly view $\Fuk(X)_\lambda$ as a full subcategory of $\operatorname{Tw} \Fuk(X)_\lambda$ or of $\calQ$, via Yoneda. Given an object $T$ in $\operatorname{Tw} \Fuk(X)_\lambda$ and an idempotent $\wp$ in $\hom^*(T, T)$, Seidel shows in \cite[Section~(4b)]{SeidelBook} that we can construct an object in $\calQ$ representing the \emph{abstract image} of $\wp$. Let $\operatorname{\Pi}\operatorname{Tw} \Fuk(X)_\lambda$ be the full subcategory of $\calQ$ comprising all such abstract images of idempotents. The category $\Fuk(X)_{\lambda,i}$ is defined to be the full subcategory of~$\operatorname{\Pi}\operatorname{Tw} \Fuk(X)_\lambda$ given by the abstract images of specific idempotents, constructed as follows. There is a homomorphism~\cite[Section~(1c)]{SeidelFlux}
\[
\Gamma^\mathrm{mod}\colon \ \HH^*(\Fuk(X)_\lambda) \to \HH^*(\calQ),
\]
such that the maps
\[
\HH^*(\Fuk(X)_\lambda) \xto{\text{restrict} \,\circ\, \Gamma^\mathrm{mod}} \HH^*(\operatorname{\Pi} \operatorname{Tw} \Fuk(X)_\lambda) \xto{\text{restrict}} \HH^*(\Fuk(X)_\lambda)
\]
are mutually inverse isomorphisms. We can therefore view $\CO_\lambda$ as a map to $\HH^*(\operatorname{\Pi} \operatorname{Tw} \Fuk(X)_\lambda)$, and similarly talk about $\CO_T$ and $\CO_T^0$ for objects $T$ in $\operatorname{\Pi} \operatorname{Tw} \Fuk(X)_\lambda$. For each $T$ in $\operatorname{Tw} \Fuk(X)_\lambda$ we get an idempotent $p_{T,i} \coloneqq \CO^0_T(e_i)$ in $\Hom^*(T, T) \coloneqq \rH^*(\hom(T, T))$, and Seidel shows in~\mbox{\cite[Section~(4b)]{SeidelBook}} that we can lift this $p_{T,i}$ to a chain-level idempotent up to homotopy $\wp_{T,i}$ in~$\hom^*(T, T)$. He further shows that the abstract image $T_i$ of $\wp_{T, i}$ is independent of the choice of lift $\wp_{T,i}$, up to quasi-isomorphism. The category $\Fuk(X)_{\lambda, i}$ is defined to be the full subcategory of~$\operatorname{\Pi} \operatorname{Tw} \Fuk(X)_\lambda$ comprising these images $T_i$.

For any objects $S$ and $T$ in $\Fuk(X)_\lambda$, and any $i$ and $j$, we have natural isomorphisms \cite[Lemmas~4.4 and 4.5]{SeidelBook}
\begin{align*}
&\Hom(S, T_i)= p_{T,i}\Hom(S, T), \qquad
 \Hom(S_j, T)= \Hom(S, T)p_{S,j}.
\end{align*}
Since $\CO_\lambda(e_i)$ is a Hochschild cocycle of even degree, we also have
\[
p_{T,i}\Hom(S, T) = \Hom(S, T)p_{S,i}.
\]
Combining these results with the fact that $p_{S,i}p_{S,j} = 0$ for $i \neq j$ (since $e_ie_j = 0$), we see that for $i \neq j$ the objects $S_j$ and $T_i$---and hence the categories $\Fuk(X)_{\lambda,i}$ and $\Fuk(X)_{\lambda,j}$---are indeed orthogonal. Hence the Hochschild cohomology $\HH^*(\Fuk(X)_\lambda) \cong \HH^*(\operatorname{\Pi} \operatorname{Tw} \Fuk(X)_\lambda)$ decomposes as $\bigoplus_i \HH^*(\Fuk(X)_{\lambda,i})$ as a $\ZZ/2$-graded algebra.

\begin{Proposition}
\label{propCOblock}
The closed-open map
\[
\CO_\lambda\colon \ \QH^*(X) = \bigoplus_i Q_i \to \HH^*(\Fuk(X)_\lambda) = \bigoplus_i \HH^*(\Fuk(X)_{\lambda,i})
\]
is block-diagonal with respect to this decomposition.
\end{Proposition}
\begin{proof}
Given distinct $i$ and $j$, we need to show that
\begin{equation}
\label{eqCOi}
\QH^*(X) \xto{\CO_\lambda} \HH^*(\Fuk(X)_\lambda) \cong \HH^*(\operatorname{\Pi} \operatorname{Tw} \Fuk(X)_\lambda) \xto{\text{restrict}} \HH^*(\Fuk(X)_{\lambda, i})
\end{equation}
vanishes on $Q_j$. Since this map is an algebra homomorphism, and we have $e_iQ_j = 0$, it suffices to show that the image of $e_i$ under \eqref{eqCOi} is invertible. By an argument \cite[Lemma 2.8]{SheridanFano} with the Eilenberg--Moore comparison theorem, it suffices in turn to show that there exists a constant $\kappa \neq 0$ such that for each object $Y$ in $\Fuk(X)_{\lambda, i}$ we have $\CO^0_Y(e_i) = \kappa \cdot 1_Y$, where $1_Y$ denotes the unit in $\Hom^*(Y, Y)$. But we can write $Y$ as $T_i$ (using the above notation) for some $T$ in $\Fuk(X)_\lambda$, and then have
\[
\CO_Y^0(e_i) = p_{T,i}\CO^0_T(e_i) = (p_{T,i})^2.
\]
Since $p_{T,i}$ is exactly $1_Y$, we are done with $\kappa = 1$.
\end{proof}

\subsection{The generation criterion for summands}

Recall that the generation criterion stated in Theorem~\ref{thmGeneration} assumed injectivity of the whole map~$\CO_\lambda$. If, however, one is only interested in split-generating one of the pieces $\Fuk(X)_{\lambda, i}$, then it suffices to check injectivity on $Q_i$. The result is well-known to experts, and the proof we give is based on \cite[Corollary~2.19]{SheridanFano}, which in turn uses important ideas from~\cite{AbouzaidGeometricCriterion} and~\cite{RitterSmith}, but we present it for completeness.

\begin{Remark}
Again, identical arguments apply to the ungraded categories $\Fuk(X)_{\lambda,i}^\mathrm{un}$.
\end{Remark}

\begin{Theorem}
\label{thmGenerationFine}
Assuming that $X$ is compact, if $\subFuk$ is a full subcategory of $\Fuk(X)_{\lambda,i}$ and if the composition
\begin{equation*}
Q_i \xto{\CO_\lambda} \HH^*(\Fuk(X)_{\lambda,i}) \xto{\mathrm{restrict}} \HH^*(\subFuk)
\end{equation*}
is injective, then $\subFuk$ split-generates $\Fuk(X)_{\lambda, i}$. In particular, a single Lagrangian $\LL$ in $\Fuk(X)_{\lambda, i}$ split-generates if
\begin{equation*}
\CO_{\LL}\colon \ \QH^*(X) \xto{\CO_\lambda} \HH^*(\Fuk(X)_\lambda) \xto{\mathrm{restrict}} \HH^*\bigl(\CF^*\bigl(\LL, \LL\bigr)\bigr)
\end{equation*}
is injective on $Q_i$.
\end{Theorem}
\begin{proof}
In order to show that $\subFuk$ split generates an object $Y$ in $\Fuk(X)_{\lambda,i}$, it suffices to show that the composition
\begin{equation}
\label{eqOCcriterion}
\HH_*(\subFuk) \xto{\OC_\lambda} \QH^{*+n}(X; \kk) \xto{\CO^0_{Y}} \Hom^{*+n}(Y, Y)
\end{equation}
hits $1_{Y}$, where
\[
\OC_\lambda\colon \ \HH_*(\Fuk(X)_\lambda) \cong \HH_*(\operatorname{\Pi} \operatorname{Tw} \Fuk(X)_\lambda) \to \QH^{*+n}(X; \kk)
\]
is the \emph{open--closed string map}. We refer the reader to \cite{AbouzaidGeometricCriterion} for the definition of this map and for the proof that hitting $1_Y$ implies split-generation; the argument breaks into a geometric part \cite[Proposition 1.3]{AbouzaidGeometricCriterion} based on the Cardy relation and a purely algebraic part \cite[Lemma~1.4]{AbouzaidGeometricCriterion}. Writing $Y$ as $T_i$ as in the proof of Proposition~\ref{propCOblock}, and using the fact that $1_{Y} = p_{T,i} =\CO^0_T(e_i)$, we see that for \eqref{eqOCcriterion} to hit $1_Y$ it is enough to show that $\OC_\lambda$ restricted to (the image of) $\HH_*(\subFuk)$ hits $e_i$. This is what we shall prove.

By \cite[Proposition 2.6]{SheridanFano}, there is a commutative diagram
\begin{equation}
\label{eqOCdual}
\begin{tikzcd}[column sep=large]
\QH^*(X) \arrow{r}{\alpha \mapsto \langle\alpha, {-}\rangle}[swap]{\cong} \arrow{d}{\CO_\lambda} & \QH^*(X)^\vee[-2n] \arrow{d}{\OC_\lambda^\vee}
\\ \HH^*(\Fuk(X)_\lambda) \arrow{r}[swap]{\cong} & \HH_*(\Fuk(X)_\lambda)^\vee[-n],
\end{tikzcd}
\end{equation}
where ${}^\vee$ denotes linear dual and $\langle {-}, {-}\rangle$ is the Poincar\'e duality pairing on $\QH^*(X)$ (this is where compactness of $X$ enters). Assume for now the following claim: the $Q_j$ are pairwise orthogonal with respect to this pairing. Then the top horizontal arrow decomposes into isomorphisms $Q_j \to Q_j^\vee[-2n]$. It follows from the construction of the bottom horizontal arrow that it respects passage to full subcategories, so it also decomposes into isomorphisms
\[
\HH^*(\Fuk(X)_{\lambda,i}) \to \HH_*(\Fuk(X)_{\lambda,i})^\vee[-n].
\]
We then deduce from Proposition~\ref{propCOblock} that $\OC_\lambda$ is block diagonal{\samepage
\[
\OC_\lambda\colon \ \HH_*(\Fuk(X)_\lambda) \cong \bigoplus_j \HH_*(\Fuk(X)_{\lambda, j}) \to \QH^{*+n}(X; \kk) = \bigoplus_j Q_j[n],
\]
and that each block is dual to the corresponding block in $\CO_\lambda$.}

Now focus on the $i$th blocks. We get that
\[
\HH_*(\Fuk(X)_{\lambda,i}) \xto{\OC_\lambda} Q_i[n] \qquad \text{and} \qquad Q_i \xto{\CO_\lambda} \HH^*(\Fuk(X)_{\lambda,i})
\]
are dual to each other, and by using again the compatibility of the bottom horizontal arrow in~\eqref{eqOCdual} with passage to full subcategories we can upgrade this to
\begin{equation}
\label{eqOCfinal}
\HH_*(\subFuk) \xto{\HH_*(\text{inclusion})} \HH_*(\Fuk(X)_{\lambda,i}) \xto{\OC_\lambda} Q_i[n]
\end{equation}
being dual to
\begin{equation}
\label{eqCOfinal}
Q_i \xto{\CO_\lambda} \HH^*(\Fuk(X)_{\lambda,i}) \xto{\text{restrict}} \HH^*(\subFuk).
\end{equation}
We are assuming that \eqref{eqCOfinal} is injective, so \eqref{eqOCfinal} is surjective and hence hits $e_i$, as we wanted.

It remains to prove the claim that the $Q_j$ are pairwise orthogonal with respect to the Poincar\'e pairing. So take distinct $j$ and $k$, and arbitrary elements $\alpha \in Q_j$ and $\beta \in Q_k$. We then have
\[
\langle \alpha, \beta \rangle = \langle e_j\alpha, e_k\beta \rangle = \langle \alpha e_j, e_k\beta \rangle = \langle \alpha, e_je_k \beta \rangle = 0,
\]
completing the proof, where the second equality uses graded-commutativity of $\QH^*(X)$ and the third uses the Frobenius algebra relation.
\end{proof}


\subsection[Eigenvalues of c\_1]{Eigenvalues of $\boldsymbol{c_1}$}
\label{sscEval}

\textit{This appendix only applies to the $\ZZ/2$-graded Fukaya category. We shall therefore assume throughout that all Lagrangians are orientable.}

Let $c_1$ denote the first Chern class of $X$, and consider the $\kk$-linear endomorphism~$c_1\qprod$ of~$\QH^*(X)$ given by quantum multiplication by $c_1$. The decomposition of $\QH^*(X)$ that is used most often is into generalised eigenspaces $Q_i$ of $c_1\qprod$, corresponding to eigenvalues $\lambda_i$. It is this decomposition that Sheridan considers in \cite{SheridanFano}, and for which he proves the analogue of Theorem~\ref{thmGenerationFine}. He works over $\CC$, where, as we will prove shortly,
\begin{equation}
\label{eqEvalFukSplitting}
\Fuk(X)_{\lambda,i} = \begin{cases} \Fuk(X)_\lambda & \text{if } \lambda = \lambda_i,
\\ 0 & \text{otherwise.}\end{cases}
\end{equation}
In particular, $\Fuk(X)_\lambda$ is zero unless $\lambda$ is equal to some $\lambda_i$, and in this case a full subcategory~$\subFuk$ split-generates if $\CO_\lambda$ is injective on the generalised eigenspace $Q_i$ (this is precisely \cite[Corollary~2.19]{SheridanFano}; or, assuming \eqref{eqEvalFukSplitting}, it is a special case of Theorem~\ref{thmGenerationFine}).

\begin{Remark}
Since $\CC$ is algebraically closed, $\QH^*(X; \CC)$ decomposes fully into generalised eigenspaces. We will phrase our results for general $\kk$, without assuming this full decomposition.
\end{Remark}

The goal of this subsection is to explain the general version of \eqref{eqEvalFukSplitting} and the corresponding generation result. The first lemma we need is a precise statement of the result mentioned above Corollary~\ref{CorollaryI}.

\begin{Lemma}[Auroux--Kontsevich--Seidel]
\label{lemAKS}
For all $T$ in $\operatorname{\Pi} \operatorname{Tw} \Fuk(X)_\lambda$, we have
\[
CO^0_T(c_1) = \lambda \cdot 1_T.
\]
\end{Lemma}
\begin{proof}
It suffices to prove the result when $T$ is a single (orientable!) Lagrangian $\LL$ in $\Fuk(X)_\lambda$. In this case, the argument of \cite[Lemma 6.7]{AurouxTDuality} and \cite[Lemma 2.7]{SheridanFano}, shows that
\begin{equation*}
\CO^0_{\LL}(2c_1) = 2\lambda \cdot 1_{\LL}.
\end{equation*}
The proof involves lifting $2c_1 \in \rH^2(X; \ZZ)$ to the Maslov class $\mu \in \rH^2(X, L; \ZZ)$ of $L$, and picking a cycle $Z \subset X \setminus L$ Poincar\'e dual to $\mu$. Then $\CO^0_{\LL}(2c_1)$ can be computed as the class swept by pseudoholomorphic discs with an interior input on $Z$ and a boundary output. Index $0$ discs are (almost) constant so do not meet $Z$ and hence do not contribute; index~$2$ discs without the interior marked point constraint sweep $\lambda \cdot 1_{\LL}$, and reintroducing the constraint doubles the result since each disc has intersection number $2$ with $Z$; and higher index discs do not contribute for degree~reasons.

Since $L$ is orientable, we can refine this using an observation of Tonkonog \cite[Section~1.2]{TonkonogCO}, as follows. The bundle pair $\bigl(\Lambda^n_\CC TX, \Lambda_\RR^n TL\bigr)$ over $(X, L)$ is classified by a homotopy class of map
\[
\phi\colon \ (X, L) \to (\mathrm{BU}(1), \mathrm{B}(\ZZ/2)) \cong (\CC\PP^\infty, \RR\PP^\infty),
\]
and $\mu$ is defined to be $\phi^*g$, where $g \in \rH^2(\CC\PP^\infty, \RR\PP^\infty; \ZZ)$ is the unique generator which maps to twice the hyperplane class $h \in \rH^2(\CC\PP^\infty; \ZZ)$. As $L$ is orientable, $\phi$ factors as
\[
(X, L) \xto{\psi} (\CC\PP^\infty, \text{point}) \xto{\text{inclusion}} (\CC\PP^\infty, \RR\PP^\infty),
\]
so $\mu = \phi^*g = \psi^* (2h)$. By construction, $\psi^* h$ maps to $c_1$ under $\rH^2(X, L; \ZZ) \to \rH^2(X; \ZZ)$. Now pick a cycle $Z \subset X \setminus L$ Poincar\'e dual to $\psi^*h$, and compute $\CO^0(c_1)$ directly using the above argument but with the new $Z$.
\end{proof}

Now fix an arbitrary decomposition $\QH^*(X) = \bigoplus_i Q_i$ as in the preceding subsections, and let the $e_i$ be the corresponding idempotents. The generalisation of \eqref{eqEvalFukSplitting} is then the following.

\begin{Lemma}
\label{lemSubcatsZero}
Given $\lambda \in \kk$, let $Q_\lambda$ denote the generalised $\lambda$-eigenspace of $c_1\qprod$ $($this is zero unless $\lambda$ is an eigenvalue$)$ and let $Q_\lambda^\perp$ denote the unique $c_1\qprod$-invariant complement, which we will construct during the proof. For each $i$ we obtain factors $Q_i \cap Q_\lambda$ and $Q_i \cap Q_\lambda^\perp$ of $\QH^*(X)$. The corresponding pieces $\Fuk(X)_{\lambda, i, \cap}$ and $\Fuk(X)_{\lambda, i, \perp}$ of $\Fuk(X)_\lambda$ then satisfy
\[
\Fuk(X)_{\lambda, i, \cap} = \Fuk(X)_{\lambda,i} \qquad \text{and} \qquad \Fuk(X)_{\lambda, i, \perp} = 0.
\]
\end{Lemma}

Before proving Lemma~\ref{lemSubcatsZero}, we mention some easy consequences. Morally, these results mean that any given splitting may as well be refined by intersecting with the generalised eigenspace splitting.

\begin{Corollary}[{\cite[Lemma 6.8]{AurouxTDuality}}]
We have $\Fuk(X)_\lambda = 0$ unless $\lambda$ is an eigenvalue of $c_1\qprod$.
\end{Corollary}
\begin{proof}
Consider the trivial splitting $Q_1 = \QH^*(X)$ of $\QH^*(X)$. We get from Lemma~\ref{lemSubcatsZero} that
\[
\Fuk(X)_{\lambda, 1, \cap} = \Fuk(X)_{\lambda,1} = \operatorname{\Pi} \operatorname{Tw} \Fuk(X)_\lambda.
\]
The left-hand side is the subcategory of $\operatorname{\Pi} \operatorname{Tw} \Fuk(X)_\lambda$ associated to the factor $Q_1 \cap Q_\lambda = Q_\lambda$. But this factor, and hence this subcategory, is zero unless $\lambda$ is an eigenvalue.
\end{proof}

\begin{Remark}
This can be proved directly as in \cite{AurouxTDuality}.
\end{Remark}

\begin{Remark}
The result may fail for the ungraded category. For example, in characteristic~$2$ the category \smash{$\Fuk\bigl(\CC\PP^2\bigr)^\mathrm{un}_0$} contains the non-zero object $\RR\PP^2$, but the eigenvalues of $c_1\qprod$ are the cube roots of unity. We learnt of this subtlety from Dmitry Tonkonog via \cite[Example 7.2.4]{EvansLekiliGeneration}.
\end{Remark}

\begin{Corollary}
If $Q_i \subset Q_{\lambda'}$, then $\Fuk(X)_{\lambda, i} = 0$ unless $\lambda = \lambda'$.
\end{Corollary}
\begin{proof}
If $\lambda \neq \lambda'$, then $Q_{\lambda'} \subset Q_\lambda^\perp$. If also $Q_i \subset Q_{\lambda'}$, then we get $Q_i \cap Q_\lambda = 0$ and hence
\[
\Fuk(X)_{\lambda, i} = \Fuk(X)_{\lambda, i, \cap} = 0.\qedhere
\tag*{\qed}
\]
\renewcommand{\qed}{}
\end{proof}

\begin{proof}[Proof of Lemma~\ref{lemSubcatsZero}]
Our first task is to construct the splitting $Q_\lambda \oplus Q_\lambda^\perp$. Let $\chi \in \kk[t]$ be the characteristic polynomial of $c_1\qprod$, of degree $n$, and let $q$ be the quotient of $\chi$ by the greatest power of $t-\lambda$ dividing it. By B\'ezout's lemma, there exist $f$ and $g$ in $\kk[t]$ such that
\begin{equation}
\label{eqfgdefn}
(t-\lambda)^nf+qg = 1.
\end{equation}
Let $e$ and $e^\perp$ be the elements of $\QH^*(X)$ obtained by evaluating $qg$ and $(t-\lambda)^nf$ at $c_1$ respectively.

We claim that these are orthogonal idempotents generating the required splitting. (Here we mean orthogonal in the algebraic sense, i.e., $ee^\perp = 0$, but by the Frobenius algebra property this implies orthogonality with respect to the inner product.) To prove this claim, note that by \eqref{eqfgdefn} we have $e+e^\perp = 1_X$. Note also that $\chi$ divides $(t-\lambda)^nq$ and $\chi(c_1) = 0$ by Cayley--Hamilton, so we have $ee^\perp = 0$. Combining these two facts, we get that $e$ and $e^\perp$ are indeed orthogonal idempotents. Next observe that the kernel of $(c_1 - \lambda \cdot 1_X)^n\qprod$ is annihilated by $e^\perp$ and, again using $\chi(c_1) = 0$, that it contains $e$. It follows that this kernel, which is precisely $Q_\lambda$, is given by $e \qprod \QH^*(X)$. It is clear that $Q_\lambda^\perp \coloneqq e^\perp \qprod \QH^*(X)$ is a $c_1\qprod$-invariant complement to $Q_\lambda$, so to complete the claim it suffices to explain why this is unique. Well, given any other $c_1\qprod$-invariant complement $V$ we have $e \qprod V \subset V$, since $e$ is a polynomial in $c_1$. Then
\[
e \qprod V \subset V \cap e \qprod \QH^*(X) = V \cap Q_\lambda = 0,
\]
so $V$ lies in the kernel of $e \qprod$, which is $Q_\lambda^\perp$. Thus, by dimension, $V$ must be equal to $Q_\lambda^\perp$.

With this in hand, the idempotents generating $Q_i \cap Q_\lambda$ and $Q_i \cap Q_\lambda^\perp$ are $e_ie$ and $e_ie^\perp$. Given an object $T$ in $\Fuk(X)_\lambda$, its projections to $\Fuk(X)_{\lambda, i}$, $\Fuk(X)_{\lambda,i,\cap}$, and $\Fuk(X)_{\lambda, i, \perp}$ are thus determined respectively by the idempotents $\CO^0_T(e_i)$, $\CO^0_T(e_ie)$, and $\CO^0_T(e_ie^\perp)$ in $\Hom^*(T, T)$. To prove the lemma, we therefore need to show that
\[
\CO^0_T(e_i) = \CO^0_T(e_ie) \qquad \text{and} \qquad \CO^0_T\bigl(e_ie^\perp\bigr) = 0.
\]
Since $e + e^\perp = 1_X$, it suffices to prove \smash{$\CO^0_T\bigl(e^\perp\bigr) = 0$}, and the latter follows from Lemma~\ref{lemAKS} via
\[
\CO^0_T\bigl(e^\perp\bigr) = \bigl(\CO^0_T(c_1) - \lambda\bigr)^nf\bigl(\CO^0_T(c_1)\bigr) = (\lambda-\lambda)^nf(\lambda) \cdot 1_T = 0.
\tag*{\qed}
\]
\renewcommand{\qed}{}
\end{proof}

Combining Lemma~\ref{lemSubcatsZero} with Theorem~\ref{thmGenerationFine} gives the following.

\begin{Corollary}[{\cite[Corollary 2.19]{SheridanFano}}]
Assuming $X$ is compact, a full subcategory $\subFuk$ of $\Fuk(X)_{\lambda, i}$ split-generates if the composition
\[
Q_i \cap Q_\lambda \xto{\CO_\lambda} \HH^*(\Fuk(X)_{\lambda,i}) \xto{\mathrm{restrict}} \HH^*(\subFuk)
\]
is injective, where $Q_\lambda$ denotes the generalised $\lambda$-eigenspace of $c_1 \qprod$ as above.
\end{Corollary}

\begin{proof}
Using the notation of Lemma~\ref{lemSubcatsZero}, consider the refined decomposition
\[
\bigoplus_i \bigl(Q_i \cap Q_\lambda \oplus Q_i \cap Q_\lambda^\perp\bigr)
\]
of $\QH^*(X)$. By Theorem~\ref{thmGenerationFine} for this new decomposition, a full subcategory $\subFuk$ of $\Fuk(X)_{\lambda,i,\cap}$ split-generates if $\CO_\lambda$ injects $Q_i \cap Q_\lambda$ into $\HH^*(\subFuk)$. But by Lemma~\ref{lemSubcatsZero}, we have
\[
\Fuk(X)_{\lambda,i,\cap} = \Fuk(X)_{\lambda,i}.
\tag*{\qed}
\]
\renewcommand{\qed}{}
\end{proof}

\section[Some A\_infty-algebra]{Some $\boldsymbol{A_\infty}$-algebra}
\label{appHH}

The purpose of this appendix is to establish some lemmas in $A_\infty$-algebra that play a key role in the proof of Theorem~\ref{TheoremC} in Section~\ref{secThmC}. They are not deep, and are probably well-known, but we need explicit formulae compatible with our sign conventions so we give the gory details.

Fix an $A_\infty$-algebra $\calA$ over $\kk$. Throughout this appendix, everything is implicitly either $\ZZ/2$-graded or, if $\Char \kk = 2$, ungraded. Given an $\calA$-bimodule $\calP$, with operations \smash{$\bim{k}{l}_\calP$}, the~Hochschild cochain complex $\rCC^*(\calA, \calP)$ is
\[
\prod_{r \geq 0} \hom^*_\kk\bigl(\calA[1]^{\otimes r}, \calP\bigr)
\]
with differential given by
\begin{align}
\nonumber
\mu^1_{\rCC}(\phi)(\dots) ={}& \sum (-1)^{\lvert \phi \rvert \maltese_l+1} \bim{k}{l}_\calP\bigl(\stackrel{k}{\dots}, \phi(\dots), \stackrel{l}{\dots}\bigr) \\
&{}{+}\, \sum (-1)^{\lvert \phi \rvert + \maltese_i}\phi(\dots, \mu(\dots), \stackrel{i}{\dots}).\label{eqCCbimod}
\end{align}
Undecorated $\mu$ operations are implicitly those on $\calA$. When $\calP$ is the diagonal bimodule, as defined in \cite[(A.2.13)]{SheridanFano}, this reduces to the ordinary Hochschild complex $\rCC^*(\calA, \calA)$ as defined by \eqref{eqCCdiff1}.

Suppose now that $\calM$ and $\calN$ are (left) $\calA$-modules. The $\kk$-module $\hom^*_\kk(\calM, \calN)$ is naturally an $\calA$-bimodule, with operations
\begin{gather}
\bim{0}{0}_{\kk-\hom}(\zeta)(m) = (-1)^{\lvert m \rvert}\bigl(\mu_\calN^1(\zeta(m)) - \zeta\bigl(\mu_\calM^1(m)\bigr)\bigr), \nonumber\\
\bim{k}{0}_{\kk-\hom}(a_k, \dots, a_1, \zeta)(m) = (-1)^{\lvert m \rvert}\mu_\calN^k(a_k, \dots, a_1, \zeta(m)), \nonumber\\
\bim{0}{l}_{\kk-\hom}(\zeta, a_l \dots, a_1)(m) = (-1)^{\lvert m \rvert + 1} \zeta\bigl(\mu_\calM^l(a_l, \dots, a_1, m)\bigr), \nonumber\\
\bim{k}{l}_{\kk-\hom} = 0.\label{eqhombimod}
\end{gather}
for $k, l > 0$. We can thus consider the Hochschild cochain complex $\rCC^*(\calA, \hom^*_\kk(\calM, \calN))$.

The other object we will need is the space
\[
\hom^*_\calA(\calM, \calN) = \prod_{r \geq 0} \hom_\kk^*\big(\calA[1]^{\otimes r} \otimes \calM, \calN\big)
\]
of $\calA$-module pre-morphisms from $\calM$ to $\calN$. Recall that this is a cochain complex with differential
\begin{align*}
\mu^1_{\calA-\hom}(\psi)(\dots, m) = {}&\sum \mu_\calN(\dots, \psi(\dots, m)) + \sum (-1)^{\lvert \psi \rvert + 1} \psi(\dots, \mu_\calM(\dots, m))
\\ &{}{+}\, \sum (-1)^{\lvert \psi \rvert + \maltese_i + \lvert m \rvert + 1} \psi(\dots, \mu(\dots), \stackrel{i}{\dots}, m).
\end{align*}
The cocycles are the actual $\calA$-module morphisms.

Our first result relates the two complexes $\rCC^*\bigl(\calA, \hom^*_\kk(\calM, \calN)\bigr)$ and $\hom^*_\calA(\calM, \calN)$.

\begin{Lemma}
There is an isomorphism of cochain complexes over $\kk$
\[
U\colon \ \rCC^*\bigl(\calA, \hom^*_\kk(\calM, \calN)\bigr) \to \hom^*_\calA(\calM, \calN).
\]
\end{Lemma}
\begin{proof}
Define the map $U$ (which denotes `uncurrying') by
\[
U(\phi)(a_k, \dots, a_1, m) = (-1)^{\lvert \phi \rvert (\lvert m \rvert + 1)} \phi(a_k, \dots, a_1)(m).
\]
By hom-tensor adjunction this is an isomorphism of the underlying $\kk$-modules of the two cochain complexes. Combining \eqref{eqCCbimod} and \eqref{eqhombimod}, we see that the differential on $\rCC^*(\calA, \hom^*_\kk(\calM, \calN))$ is given by
\begin{align*}
\mu^1_{\rCC}(\phi)(\dots)(m) ={}& \sum (-1)^{\lvert m \rvert + 1} \mu_\calN(\dots, \phi(\dots)(m)) \!+\! \sum(-1)^{\lvert \phi \rvert \maltese_i + \lvert m \rvert} \phi(\dots)(\mu_\calM(\stackrel{i}{\dots}, m))
\\ &{}{+}\, \sum (-1)^{\lvert \phi \rvert + \maltese_i} \phi(\dots, \mu(\dots), \stackrel{i}{\dots})(m).
\end{align*}
A straightforward calculation shows that $U$ intertwines this with \smash{$\mu^1_{\calA-\hom}$}.
\end{proof}

Now suppose that $\calM$ is actually $\calA$, viewed as a left module over itself. Denoting the module operations by $\mu_\calM$ still, to avoid confusing them with the algebra operations on $\calA$, they are given by $\mu_\calM(a_k, \dots, a_1) = - \mu(a_k, \dots, a_1)$. We then have the following reassuring result.

\begin{Lemma}
If $\calA$ and $\calN$ are cohomologically unital, then the map
$\pi\colon \hom_\calA^*(\calA, \calN) \to \calN$ given by $\pi(\psi) = \psi(e_\calA)$
is a quasi-isomorphism. Here $e_\calA$ is a chain-level representative for the cohomological unit~$1_\calA$ in~$\calA$.
\end{Lemma}

\begin{Remark}
Cohomological unitality for the $\calA$-module $\calN$ means that for all cocycles $n$ the cocycles $\mu_\calN^2(e_\calA, n)$
and $(-1)^{\lvert n \rvert + 1}n$ are cohomologous. In particular, $\calA$ being cohomologically unital as an $\calA$-module reduces to $\calA$ being cohomologically unital as an $A_\infty$-algebra.
\end{Remark}

\begin{proof}
Define a map $\iota\colon \ \calN \to \hom_\calA^*(\calA, \calN)$ by
\[
\iota(n)(a_k \dots a_1) = (-1)^{\lvert n \rvert + 1} \mu_\calN(a_k, \dots, a_1, n).
\]
The differential on $\hom^*_\calA(\calA, \calN)$ is given by
\begin{equation}
\label{eqhomANdiff}
\mu^1_{\calA-\hom}(\psi)(\dots) = \sum \mu_\calN(\dots, \psi(\dots)) + \sum (-1)^{\lvert \psi \rvert + \maltese_i} \psi(\dots, \mu(\dots), \stackrel{i}{\dots}),
\end{equation}
and it is easily verified from this that both $\pi$ and $\iota$ are chain maps. We claim that they are mutually quasi-inverse. It is straightforward to check that $\pi \circ \iota$ induces the identity map on~$\rH^*(\calN)$, so it remains to show that $\iota \circ \pi$ induces an isomorphism on $\Hom^*_\calA(\calA, \calN)$. For this, we will use the Eilenberg--Moore comparison theorem.

Equip $\hom_\calA^*(\calA, \calN)$ with the length filtration, whose $p$th filtered piece $F^p\hom_\calA^*(\calA, \calN)$ comprises those $\psi$ whose $r$th component $\psi^r$ vanishes for all $r \leq p$. This filtration is exhaustive and complete, and is respected by the map $\iota \circ \pi$ since $\pi$ annihilates $F^1$. By Eilenberg--Moore, it thus suffices to show that, for some $r$, $\iota \circ \pi$ induces an isomorphism on the $E_r$ page of the associated spectral sequence. We claim that this happens for $r=2$.

The zeroth page of the spectral sequence is
\[
E_0^{p,q} = \begin{cases}\hom^q_\kk\bigl(\calA^{\otimes (p+1)}, \calN\bigr) & \text{if $p\geq 0$}, \\ 0 & \text{otherwise.}\end{cases}
\]
The differential coincides with that on the space of maps of complexes from $\calA^{\otimes (p+1)}$ to $\calN$. Its cohomology is therefore the space of chain maps $\calA^{\otimes (p+1)} \to \calN$ modulo chain homotopy. Because we are working over a field, this simplifies to the space of linear maps between the corresponding cohomology groups. In other words, we have
\[
E_1^{p,q} = \begin{cases}\Hom^q_\kk\bigl(\rH^*(\calA)^{\otimes (p+1)}, \rH^*(\calN)\bigr) & \text{if $p\geq 0$}, \\ 0 & \text{otherwise.}\end{cases}
\]

On this $E_1$ page the differential $\mu^1_{E_1} \colon E_1^{p,q} \to E_1^{p+1,q}$ is given by the $\mu_\calN^2$ and $\mu^2$ terms in~\eqref{eqhomANdiff}.~Now define the maps \smash{$h \colon E_1^{p, q} \to E_1^{p-1, q}$} by
\[
h(\psi)(a_{p-1}, \dots, a_1) = (-1)^{\lvert \psi \rvert - 1}\psi(a_{p-1}, \dots, a_1, 1_\calA).
\]
These satisfy
\[
\mu^1_{E_1}(h(\psi)) + h\bigl(\mu^1_{E_1}(\psi)\bigr) = \psi - \iota \circ \pi(\psi),
\]
so $\iota \circ \pi$ is homotopic to the identity on~$E_1$ and hence induces an isomorphism (the identity) on~$E_2$, completing the proof.
\end{proof}

Combining the two previous results gives the following.

\begin{Corollary}\label{corCCN}
Suppose $\calA$ is an $A_\infty$-algebra and $\calN$ is an $\calA$-module, and that both are cohomologically unital. Then the map
$
\pi \circ U\colon \rCC^*(\calA, \hom^*_\kk(\calA, \calN)) \to \calN$
given by $\phi \mapsto (-1)^{\lvert \phi \rvert} \phi^0(e_\calA)$ is a quasi-isomorphism.
\end{Corollary}

\section{Quantum cohomology of monotone toric manifolds}

Throughout this appendix, let $X$ be a compact toric manifold and let $L \subset X$ be a toric fibre. In Appendix~\ref{sscToricSetup}, we give a more precise description of this setup, part way through which we impose the additional condition that $X$ and $L$ are monotone, which will persist to the end of the appendix. In Appendices~\ref{sscHFtoric} and~\ref{sscSmallQH}, we then compute $\HF^*_S(\bL, \bL)$ and use this to give a~presentation of $\QH^*(X)$. Finally, in Appendix~\ref{sscQHcoeff} we discuss the enlargement of the coefficients of quantum cohomology, necessary for our results on real Lagrangians in Section~\ref{secRealToric}. As in the previous appendices, this material is well-known to experts, and we spell it out for completeness and compatibility with our conventions.

\subsection{Toric geometry setup}
\label{sscToricSetup}

The toric manifold $X$ is described by a moment polytope $\Delta$ as follows. Fix an abstract $n$-torus~$T$ with Lie algebra $\mathfrak{t}$, let $M \subset \mathfrak{t}$ be the lattice given by the kernel of the exponential map, and let $\langle {-}, {-} \rangle$ denote the pairing between $\mathfrak{t}$ and its dual $\mathfrak{t}^\vee$. The polytope $\Delta$ is a compact subset of~$\mathfrak{t}^\vee$, of the form
\begin{equation}
\label{eqPolytope}
\Delta = \bigl\{x \in \mathfrak{t}^\vee \mid \langle \nu_i, x \rangle \geq - \lambda_i \text{ for } i=1,\dots,N\bigr\},
\end{equation}
where $\nu_1, \dots, \nu_N$ are elements of $M$ and $\lambda_1, \dots, \lambda_N$ are real numbers. We require that exactly~$k$ facets (codimension-$1$ faces) of $\Delta$ meet at each codimension-$k$ face, and that at each such meeting point the corresponding normal vectors $\nu_i$ can be extended to a $\ZZ$-basis for $M$. We also require that none of the inequalities in \eqref{eqPolytope} is redundant, so there are exactly $N$ facets.

The manifold $X$ is constructed by symplectic reduction of $\CC^N$, using the data of $\Delta$. See \cite[Chapter 29]{CdS} for details. $X$ inherits a Hamiltonian $T$-action, with moment map $\mu \colon X \to \mathfrak{t}^\vee$ whose image is exactly $\Delta$. The \emph{toric fibres} are the preimages of interior points of $\Delta$ under $\mu$, each of which is a Lagrangian free $T$-orbit. Letting $F_i$ denote the $i$th facet, namely
\[
\{x \in \Delta \mid \langle \nu_i, x \rangle = -\lambda_i\},
\]
the \emph{toric divisors} $D_i$ are the preimages $\mu^{-1}(F_i)$. These form a free basis for $\rH_{2n-2}(X, L; \ZZ)$. We write $H_i$ for the Poincar\'e dual of $D_i$, in $\rH^2(X)$.

The holomorphic index $2$ discs in $X$ with boundary on $L$, for the standard complex structure, were explicitly described by Cho and Cho--Oh \cite{ChoClifford,ChoOh}. To state their result, let $\calM_1(\beta)$ denote the moduli space of (smooth) holomorphic discs in class $\beta \in \rH_2(X, L; \ZZ)$ with one boundary marked point. Let $\beta_1, \dots, \beta_N$ denote the basis of $H_2(X, L; \ZZ)$ dual to the toric divisors $D_1, \dots, D_N$. Equip $L$ with the \emph{standard} spin structure, meaning the spin structure induced by the trivialisation of $TL$ induced by any diffeomorphic identification of $L$ with $\RR^n/\ZZ^n$. Cho and Cho--Oh proved the following.

\begin{Proposition}
\label{propToricInd2Count}
The index $2$ disc classes $\beta$ for which $\calM_1(\beta)$ is non-empty are precisely $\beta_1, \dots, \beta_N$. For each $i$, the moduli space $\calM_1(\beta_i)$ is a free $T$-orbit, and the evaluation map
\[
\ev\colon \ \calM_1(\beta_i) \to L
\]
at the marked point is $T$-equivariant. With respect to the standard spin structure, $\ev$ has degree~$+1$.
\end{Proposition}

The disc class $\beta_i$ has area $2\pi(\langle \nu_i, \mu(L) \rangle + \lambda_i)$. There is a natural identification between $\rH_1(L; \ZZ)$ and $\rH_1(T; \ZZ)$ since $L$ is a free $T$-orbit, and between $\rH_1(T; \ZZ)$ and $M$ since $T \cong \mathfrak{t} / M$. Under these identifications, the boundary $\partial \beta_i \in \rH_1(L; \ZZ)$ corresponds to $\nu_i \in M$.

Assume from now on that $X$ is monotone. This is equivalent to being able to translate $\Delta$ so that all $\lambda_i$ are equal, which we assume has been done. We also assume that the symplectic form has been rescaled so that each $\lambda_i$ is $1$. In this case, exactly one of the toric fibres is monotone, namely $\mu^{-1}(0)$. This will be our $L$. Viewing the $\nu_i$ as elements of $\rH_1(L; \ZZ)$ via the identification from the previous paragraph, Proposition~\ref{propToricInd2Count} gives the following.

\begin{Corollary}
\label{corWL}
The superpotential $W_L$ is $z^{\nu_1} + \dots + z^{\nu_N}$.
\end{Corollary}

\subsection{Floer cohomology of the monotone fibre}
\label{sscHFtoric}

Throughout this appendix, we shall use a pearl model for Floer cohomology. So Floer cochains are Morse cochains associated to a fixed choice of Morse function, and the differential counts pearly trajectories as described in Section~\ref{sscTonkRevisit}. In particular, when computing $\HF^*_S$ we weight each trajectory by the monomial $z^{[\partial u_1] + \dots + [\partial u_k]}$ in $S$, where $u_1, \dots, u_k$ are the discs in the trajectory.

With this in place, we can compute $\HF^*_S(\bL, \bL)$ using the Oh spectral sequence \cite{biran2007quantum, OhSS}
\[
E_1 = \rH^*(L; S) \implies \HF^*_S(\bL, \bL).
\]

\begin{Proposition}\label{propHFtoric}
There is a canonical $\kk$-algebra isomorphism $\HF^*_S(\bL, \bL) \cong \Jac W_L$.
\end{Proposition}
\begin{proof}
The restriction of the $E_1$ differential to
\[
\rH^1(L) \to \rH^0(L; S) = S
\]
counts pearly trajectories that start at an index $1$ critical point, contain a single index $2$ disc, and then end at a fixed index $0$ critical point $p$. More precisely, the input should be a coexact linear combination of index $1$ critical points, representing a class $b \in \rH^1(L)$, and we should weight each trajectory by the monomial representing the boundary of its disc. This amounts to counting index $2$ discs $u$ through $p$, weighted by \smash{$z^{[\partial u]}$}, and additionally weighted by the pairing $\langle [\partial u], b \rangle$. Using Corollary~\ref{corWL}, the result is \smash{$\sum_i \langle \nu_i, b \rangle z^{\nu_i}$}.

By $S$-linearity and the Leibniz rule, we see that the full $E_1$ differential on $\rH^*(L; S) = S \otimes_\kk \Lambda_\kk^* \rH^1(L)$ is given by contraction with $\sum_i \nu_i z^{\nu_i}$. Thus the $E_1$ page is a Koszul complex resolving
\[
S\bigg/ \biggl( \sum_i \nu_i z^{\nu_i} \biggr);
\]
see Remark~\ref{rmkKoszulRegular}, below. (As usual, an ideal generated by vectors implicitly means the ideal generated by their components with respect to any basis.) The spectral sequence therefore collapses and
\[
\HF^*_S(\bL, \bL) = \HF^0_S(\bL, \bL) = S\bigg/ \biggl( \sum_i \nu_i z^{\nu_i} \biggr).
\]
To identify this quotient, fix a $\ZZ$-basis $e_1, \dots, e_n$ for $M$, and let $z_1, \dots, z_n$ be the corresponding monomials $z^{e_i}$ in $S$. Then for each $j$ the $e_j$-component of $\sum_i \nu_iz^{\nu_i}$ is
\[
z_j \frac{\partial W_L}{\partial z_j},
\]
so $\HF^*_S(\bL, \bL)$ is exactly $\Jac W_L$.
\end{proof}

\begin{Remark}
\label{rmkKoszulRegular}
The components $c_1, \dots, c_n$ of $\sum_i \nu_i z^{\nu_i}$ with respect to an arbitrary basis form a Koszul-regular sequence in $S$, i.e., their Koszul complex has vanishing (co)homology outside degree $0$, for the following reason. Since the vanishing of a module can be checked locally at maximal ideals, and since localisation is exact, it suffices to show that the sequence is Koszul-regular in each localisation $S_\frn$ of $S$ at a maximal ideal $\frn$. For each such $\frn$, either all $c_i$ are contained in $\frn$ or some $c_i$ is not contained in $\frn$. In the former case, the $c_i$ form a regular---and hence Koszul-regular---sequence by \cite[Exercise~26.2.D]{VakilAG}, because $S_\frn$ is a regular local ring of Krull dimension $n$ and $S_\frn / (c_1, \dots, c_n)$ has Krull dimension zero (since it is finite-dimensional over $\kk$). In the latter case, the Koszul complex over $S_\frn$ is acyclic since multiplication by some~$c_i$ is an isomorphism $S_\frn \to S_\frn$.
\end{Remark}

\subsection{Small quantum cohomology}\label{sscSmallQH}

The map
$\CO^0_\bL\colon \QH^*(X) \to \HF^*_S(\bL, \bL) \cong \Jac W_L$
sends each toric divisor class $H_i$ to the count of index $2$ discs $u$ meeting $D_i$, weighted by $z^{[\partial u]}$. Since the basic classes $\beta_j$ are dual to the~$D_j$, Proposition~\ref{propToricInd2Count} tells us that \smash{$\CO^0_\bL(H_i)$} counts exactly one disc, which is in class $\beta_i$. Thus \smash{$\CO^0_\bL(H_i) = z^{\nu_i}$}. The following is proved in \cite[Proposition~4.6]{SmithQH}, building on work of Fukaya--Oh--Ohta--Ono~\cite{FOOOtoricMS}.

\begin{Theorem}
\label{thmSmallQH}
The map $\CO^0_\bL\colon \QH^*(X) \to \Jac W_L$ is an isomorphism of $\kk$-algebras.
\end{Theorem}

\begin{Remark}
The statement of \cite[Proposition 4.6]{SmithQH} is superficially different, stating that a~map
\[
\operatorname{\mathfrak{ks}_\mathrm{mon}} \colon \ \QH^*(X; R[T]) \to R[\Gamma] \bigg/ \biggl(\sum_i \nu_i z_i\biggr)
\]
is an isomorphism of $R[T]$-algebras, for any ring $R$. Here $T$ is the Novikov variable, and $\Gamma$ is the submonoid of $\ZZ \oplus \rH_1(L; \ZZ)$ generated by $(1, \nu_1), \dots, (1, \nu_N)$. The element $z_i \in R[\Gamma]$ is the monomial corresponding $(1, \nu_i)$, and $R[\Gamma]$ is viewed as an $R[T]$-algebra by letting $T$ acts as the monomial corresponding to $(1, 0)$. To obtain the statement we want, take $R$ to be $\kk$ and reduce both sides modulo $(T-1)$. Then $R[\Gamma]$ becomes $\kk[\rH_1(L; \ZZ)] = S$, and $\sum_i \nu_i z_i$ becomes $\sum_i \nu_iz^{\nu_i}$ as above, so the codomain of $\operatorname{\mathfrak{ks}_\mathrm{mon}}$ becomes $\Jac W_L$. By inspecting the definition of the map itself in \cite[Section~2.3]{SmithQH}, one sees that it becomes exactly the weighted closed-open map $\CO^0_\bL$.
\end{Remark}

More generally, it is proved in \cite[Proposition 4.14]{SmithQH}, again building on \cite{FOOOtoricMS}, that the result still holds after twisting by any class in $\chi \in \rH^2(X; \kk^\times)$ (called $\rho$ in \cite{SmithQH}). Explicitly, fix such a~$\chi$ and~choose a lift to $\widehat{\chi} \in \rH^2(X, L; \kk^\times)$; the set of lifts (is non-empty and) forms a torsor for $\rH^1(L; \kk^\times)$ since the restriction map $\rH^*(X; \kk^\times) \to \rH^*(L; \kk^\times)$ vanishes. Define $\QH^*_\chi(X)$ by~modifying the usual definition of quantum cohomology by weighting each count of holomorphic spheres~$u$ by~$\chi([u])$. Similarly define $\HF^*_{S, \widehat{\chi}}(\bL, \bL)$ by weighting each count of holomorphic discs~$u$ by $\widehat{\chi}([u])$ in addition to the monomial~\smash{$z^{[\partial u]}$}. The same argument as for Proposition~\ref{propHFtoric} shows~that
\[
\HF^*_{S, \widehat{\chi}}(\bL, \bL) \cong \Jac W_{L, \widehat{\chi}},
\]
where \smash{$W_{L, \widehat{\chi}} = \sum_i \widehat{\chi}(\beta_i) z^{\nu_i}$}. Here $\beta_i$ is the $i$th basic disc class as above. The statement is then the following.

\begin{Proposition}
\label{propTwistedCO}
The twisted closed-open map
\[
\CO^0_{\bL, \widehat{\chi}} \colon \ \QH^*_\chi(X) \to \HF^*_{S, \widehat{\chi}}(\bL, \bL)
\]
is an isomorphism of $\kk$-algebras.
\end{Proposition}

\subsection{Enlarging the coefficients}
\label{sscQHcoeff}

We wish to extend Proposition~\ref{propTwistedCO} slightly, by incorporating `every $\chi$ at once', analogously to the way that working over $S$ incorporates `every local system at once'. More precisely, define~\smash{$\QH^*_{\rH_2}(X)$} by working over $\kk[\rH_2(X; \ZZ)]$ and weighting each count of holomorphic spheres~$u$ by the monomial $Z^{[u]}$ corresponding to the homology class $[u]$. Similarly, let \smash{$\tilS = \kk[\rH_2(X, L; \ZZ)]$} and define \smash{$\HF^*_{\tilS}(\bL, \bL)$} by working over \smash{$\tilS$} and weighting each pearly trajectory by the monomial \smash{$Z^{[u_1]+\dots+[u_k]}$} associated to the sum of its disc classes.

Since the basic disc classes $\beta_i$ form a free basis for $\rH_2(X, L; \ZZ)$, we have \smash{$\tilS = \kk\bigl[Z_1^{\pm1}, \dots, Z_N^{\pm 1}\bigr]$}, where $Z_i$ denotes \smash{$Z^{\beta_i}$}. The long exact sequence of the pair $(X, L)$ tells us that
\[
\rH_2(X; \ZZ) = \biggl\{ \sum_i p_i \beta_i \mid p_i \in \ZZ \text{ and } \sum_i p_i\nu_i = 0\biggr\}
\]
so $\kk[\rH_2(X; \ZZ)]$ is the subalgebra of \smash{$\tilS$} spanned by monomials \smash{$Z_1^{p_1}\cdots Z_N^{p_N}$} with $\sum_i p_i \nu_i = 0$. A~class $\chi \in \rH^2(X; \kk^\times)$ induces an algebra homomorphism $\kk[\rH_2(X; \ZZ)] \to \kk$, and we denote its kernel by $\frn$. Reducing \smash{$\QH^*_{\rH_2}(X)$} modulo $\frn$ gives exactly \smash{$\QH^*_\chi(X)$} from the previous subsection. Similarly a choice of lift of $\chi$ to \smash{$\widehat{\chi} \in \rH^2(X, L; \kk^\times)$} induces an isomorphism \smash{$\tilS / \frn \tilS \to S$} by~${Z^\beta \mapsto \widehat{\chi}(\beta) z^{[\partial \beta]}}$. The analogue of $W_L$ in \smash{$\tilS$} is $\sum_i Z_i$, and similar arguments to Proposition~\ref{propHFtoric} and Remark~\ref{rmkKoszulRegular} show that
\[
\HF^*_{\tilS}(\bL, \bL) \cong \tilS \bigg/ \biggl( \sum_i \nu_i Z_i \biggr),
\]
whose reduction modulo $\frn$ is \smash{$\HF^*_{S,\widehat{\chi}}(\bL, \bL)$}.

The main result of this appendix is the following.

\begin{Theorem}
\label{thmSmallishToricQH}
The weighted closed-open map
\[
\CO^0_{\bL, \rH_2}\colon \ \QH^*_{\rH_2}(X) \to \HF^*_{\tilS}(\bL, \bL) = \kk\bigl[Z_1^{\pm 1}, \dots, Z_N^{\pm 1}\bigr] \bigg/ \biggl( \sum_i \nu_i Z_i \biggr)
\]
sends $H_i$ to $Z_i$ and is an isomorphism of $\kk[\rH_2(X; \ZZ)]$ algebras.
\end{Theorem}
\begin{proof}
The computation \smash{$\CO^0_{\bL, \rH_2}(H_i) = Z_i$} is completely analogous to \smash{$\CO^0_\bL(H_i) = z^{\nu_i}$} from the previous subsection. To prove that \smash{$\CO^0_{\bL, \rH_2}$} is an isomorphism, we first show it is surjective. We have just seen that it hits each $Z_i$, and we also know that it hits the image of $\kk[\rH_2(X; \ZZ)]$, namely those monomials $Z_1^{p_1}\cdots Z_N^{p_N}$ with $\sum_i p_i \nu_i = 0$, so to prove surjectivity it suffices to show that together these elements generate each \smash{$Z_i^{-1}$}.

To do this, fix an $i$ and consider the linear functional $\langle \nu_i, \bullet \rangle \colon \mathfrak{t}^\vee \to \RR$. Its restriction to~$\Delta$ attains its maximum at some vertex $v$, which is the intersection of $n$ facets, with normals $\nu_{j_1}, \dots, \nu_{j_n}$ say. The $\nu_{j_k}$ form a $\ZZ$-basis for $M = \rH_1(L; \ZZ)$ so we can write $-\nu_i$ as $\sum_k r_k \nu_{j_k}$ for unique integers $r_k$. Then \smash{$\bigl(Z_i \prod_k Z_{j_k}^{r_k}\bigr)^{-1}$} lies in $\kk[\rH_2(X; \ZZ)]$. And since $\langle \nu_i, \bullet \rangle$ attains its maximum at $v$ the $r_k$ are all non-negative. Therefore,{\samepage
\[
Z_i^{-1} = Z_{j_1}^{r_1} \cdots Z_{j_n}^{r_n} \biggl(Z_i \prod_k Z_{j_k}^{r_k} \biggr)^{-1}
\]
is generated by $Z_{j_1}, \dots, Z_{j_n}$ and the image of $\kk[\rH_2(X; \ZZ)]$, as wanted, so \smash{$\CO^0_{\bL, \rH_2}$} is surjective.}

It now just remains to show that \smash{$\CO^0_{\bL, \rH_2}$} is injective. Assume first that $\kk$ is infinite. For \smash{$\chi \in \rH^2(X; \kk^\times)$}, let $\frn$ denote the kernel of the induced algebra homomorphism $\kk[\rH_2(X; \ZZ)] \to \kk$ as above. Reducing the domain and codomain of \smash{$\CO^0_{\bL, \rH_2}$} modulo $\frn$ gives exactly \smash{$\CO^0_{\bL, \chi}$}, which we know is injective by Proposition~\ref{propTwistedCO}. It therefore suffices to show that for all non-zero $x \in \QH^*_{\rH_2}(X)$ there exists $\chi$ such that $x$ is non-zero modulo $\frn$. And since $\QH^*_{\rH_2}(X)$ is a free $\kk[\rH_2(X; \ZZ)]$-module, it is in fact enough to show that for all non-zero $f \in \kk[\rH_2(X; \ZZ)]$ there exists $\chi$ such that $f$ is non-zero modulo $\chi$.

Suppose then that we are given non-zero $f \in \kk[\rH_2(X; \ZZ)]$. After picking a basis for $\rH_2(X; \ZZ)$, we can view $\kk[\rH_2(X; \ZZ)]$ as a Laurent polynomial ring $\kk\bigl[Y_1^{\pm 1}, \dots, Y_{N-n}^{\pm 1}\bigr]$, where $N-n$ is the rank of $\rH_2(X; \ZZ)$ (recall that $n$ is the complex dimension of $X$ and $N$ is the number of facets of its moment polytope $\Delta$). Then $f$ is a non-zero Laurent polynomial $f(Y_1, \dots, Y_{N-n})$ in the $Y_i$, and we need to show that there exist $y_1, \dots, y_{N-n} \in \kk^\times$ such that $f(y_1, \dots, y_{N-n}) \neq 0$. This follows by induction on $N-n$, using the fact that $\kk$ is infinite but a non-zero polynomial can have only finitely many zeros.

We are left to deal with finite $\kk$. In this case note that if \smash{$\CO^0_{\bL, \rH_2}$} were not injective then it would continue to be non-injective after tensoring with $\kk(t)$. But the resulting map would necessarily be injective by the argument just given for infinite fields.
\end{proof}

\begin{Remark}
The isomorphism \smash{$\QH^*_{\rH_2}(X) \cong \kk\bigl[Z_1^{\pm 1}, \dots, Z_N^{\pm 1}\bigr] / (\sum_i \nu_i Z_i)$} still holds if $\kk$ is allowed to be an arbitrary ring, by arguing as in \cite[Section~4.2, especially Proposition~4.6]{SmithQH}. The isomorphism can still be constructed as a weighted closed-open map, but it is no longer clear that the analogue of Remark~\ref{rmkKoszulRegular} holds, so \smash{$\kk\bigl[Z_1^{\pm 1}, \dots, Z_N^{\pm 1}\bigr] / (\sum_i \nu_i Z_i)$} may in principle be a~proper subalgebra of \smash{$\HF^*_{\tilS}(\bL, \bL)$}.
\end{Remark}

Finally, let $R = \kk[\rH_2(X; \ZZ/2)]$, viewed as a $\kk[\rH_2(X; \ZZ)]$-algebra via the natural map
\[
\rH_2(X; \ZZ) \to \rH_2(X; \ZZ/2),
\]
and let
\[
\QH^*_R(X) = R \otimes_{\kk[\rH_2(X; \ZZ)]} \QH^*_{\rH_2}(X).
\]
The consequence of Theorem~\ref{thmSmallishToricQH} that we actually need is as follows.

\begin{Corollary}
\label{corToricQH}
We have identifications of $\kk$-algebras
\[
\QH^*_R(X) \cong \kk\bigl[Z_1^{\pm 1}, \dots, Z_N^{\pm 1}\bigr] \bigg/ \biggl(\sum_j \nu_j Z_j\biggr) + \biggl(\prod_j Z_j^{\langle 2A, H_j \rangle}-1:A \in \rH_2(X; \ZZ)\biggr)
\]
and
\[
\QH^*(X) \cong \kk\bigl[Z_1^{\pm 1}, \dots, Z_N^{\pm 1}\bigr] \bigg/ \biggl(\sum_j \nu_j Z_j\biggr) + \biggl(\prod_j Z_j^{\langle A, H_j \rangle}-1:A \in \rH_2(X; \ZZ)\biggr),
\]
under which $H_i$ on the left-hand side corresponds to $Z_i$ on the right-hand side.
\end{Corollary}
\begin{proof}
Take the identification of $\kk[\rH_2(X; \ZZ)]$-algebras
\[
\QH^*_{\rH_2}(X) = \kk\bigl[Z_1^{\pm 1}, \dots, Z_N^{\pm 1}\bigr] \bigg/ \biggl( \sum_i \nu_i Z_i \biggr)
\]
from Theorem~\ref{thmSmallishToricQH}, under which $H_i$ is sent to $Z_i$. Recall that the $\kk[\rH_2(X; \ZZ)]$-action on the right-hand side arises from the map
\[
\kk[\rH_2(X; \ZZ)] \to \kk[\rH_2(X, L; \ZZ)] = \tilS = \kk\bigl[Z_1^{\pm 1}, \dots, Z_N^{\pm 1}\bigr].
\]
Under this map, the monomial $Z^A \in \kk[\rH_2(X; \ZZ)]$ is sent to $\prod_j Z_j^{p_j}$, where $A = \sum_j p_j \beta_j$ in~$\rH_2(X, L; \ZZ)$. Since each $\beta_j$ has intersection number $\delta_{ij}$ with the $i$th toric divisor $D_i$, we~can express each $p_j$ as $\langle A, H_j \rangle$, and hence $Z^A$ as \smash{$\prod_j Z_j^{\langle A, H_j \rangle}-1$}. The corollary then follows immediately from the fact that $\QH^*(X)$ and $\QH^*_R(X)$ are obtained from $\QH^*_{\rH_2}(X)$ by setting each~$Z^A$, respectively $Z^{2A}$, to $1$.
\end{proof}

\subsection*{Acknowledgements}

I am grateful to Oscar Randal-Williams, Ivan Smith, and anonymous referees for useful comments, to Nick Sheridan for the triangle algebra argument in Lemma~\ref{lemThetaIndep2}, and to Leonie Luginb\"uhl and Clara Manco for initiating the sequence of events that led to this paper. This work was carried out at St John's College, Cambridge, whom I thank for their support and hospitality.


\pdfbookmark[1]{References}{ref}
\LastPageEnding

\end{document}